\documentclass[12pt, reqno]{amsart}
\usepackage{cooltooltips}
\usepackage[all,arc,cmtip]{xy}
\usepackage{color,soul}

\usepackage{enumerate}
\usepackage{mathrsfs}
\usepackage{euscript}
\usepackage{hyperref}
\usepackage{amsthm}
\usepackage{amsmath}
\usepackage{MnSymbol}
\usepackage{comment} 
\usepackage{tikz}
\usepackage{tikz-cd}
\tikzcdset{scale cd/.style={every label/.append style={scale=#1},
    cells={nodes={scale=#1}}}}
\tikzcdset{arrow cd/.style={every arrow/.append style={scale=#1}}}

\usepackage{xspace}
\usetikzlibrary{patterns}
\usepackage{tikz-3dplot}
\usepackage{subcaption}
\usetikzlibrary{arrows,positioning,shapes}
\usepackage[semibold,tt=false]{libertine}
\usepackage{libertinust1math}
\usepackage[font={sf,small},labelsep=quad,labelfont=sc]{caption}
\usepackage[subrefformat=parens]{subcaption}
\usepackage{tikz}
\usetikzlibrary{arrows.meta}
\usepackage{siunitx}
\usepackage{pgfplots}
\usepgfplotslibrary{colormaps}%
\usepgfplotslibrary{colorbrewer} 
\usepgfplotslibrary[colorbrewer] 
\usetikzlibrary{pgfplots.colorbrewer} 
\usetikzlibrary[pgfplots.colorbrewer] 
\usepackage{ifthen}

\usetikzlibrary{cd, arrows, matrix, intersections, math, calc}

\xdefinecolor{O}{RGB}{255, 102, 17}
\xdefinecolor{B}{RGB}{17, 87, 221}
\usepackage{xcolor}

   \usepackage{bm}  

\usepackage{esvect}
\usepackage{adjustbox}

\usepackage[latin1]{inputenc}

\usepackage{textcomp}
\usepackage{usnomencl}

\usepackage{graphicx}
\usepackage{color}
\usepackage{eso-pic}

\usepackage{mathtools}
\usepackage{multicol}
\usepackage{paralist}
\usepackage{geometry}
\usepackage{algorithmic}
\usepackage[ruled,vlined]{algorithm2e}
\usepackage{comment}
\usepackage{tikz-cd}
\usepackage{url}
\usepackage{graphicx,scalerel}
\usetikzlibrary{decorations.pathmorphing}
\usepackage{float}
\usepackage{tabulary}
\usepackage{booktabs}
\usepackage{tikz}
\usepackage{pgfplots}

\usepgfplotslibrary{colormaps}
\pgfplotsset{
    compat=newest,
    colormap={mycolormap}{color=(lightgray) color=(white) color=(lightgray)}
}

\newtheorem{theorem}{Theorem}[section]

\newtheorem{lemma}[theorem]{Lemma}
\newtheorem{proposition}[theorem]{Proposition}
\theoremstyle{definition}
\newtheorem{definition}[theorem]{Definition}
\newtheorem{remark}[theorem]{Remark}
\newtheorem{example}[theorem]{Example}
\newtheorem{deflem}[theorem]{Definition-Lemma}

\newsavebox{\pullback}
\sbox\pullback{%
\begin{tikzpicture}%
\node at (0ex,0.5ex) (v100) {$\dag$};
\end{tikzpicture}}

\newcommand{\Ops}[1]{\mathbb{Ops}_{\substack{\text{\scalebox{0.8}{$#1$}}}}}
\newcommand{\Psh}[2]{\mathbb{Psh}_{\scalebox{0.5}{$\mathbf{#1},\! \mathbf{#2}$}}}
\newcommand{\EPsh}[2]{\mathbb{EPsh}_{\scalebox{0.5}{$\mathbf{#1},\! \mathbf{#2}$}}}
\newcommand{\SPsh}[2]{\mathbb{SPsh}_{\scalebox{0.5}{$\mathbf{#1},\! \mathbf{#2}$}}}
\newcommand{\ESPsh}[2]{\mathbb{ESPsh}_{\scalebox{0.5}{$\mathbf{#1},\! \mathbf{#2}$}}}
\newcommand{\Sh}[2]{\mathbb{Sh}_{\scalebox{0.5}{$\mathbf{#1},\! \mathbf{#2}$}}}
\newcommand{\ESh}[2]{\mathbb{ESh}_{\scalebox{0.5}{$\mathbf{#1},\! \mathbf{#2}$}}}
\newcommand{\iuv}[2]{\mathfrak{i}_{\scalebox{0.8}{$\scriptstyle #1$}\mkern-0.1mu\scalebox{0.9}{$\scriptstyle, $}\mkern-0.1mu\scalebox{0.8}{$\scriptstyle #2$}}}

\newcommand{\iuvop}[3]{\mathfrak{i}_{\scalebox{0.8}{$\scriptstyle #1$}\mkern-0.7mu\scalebox{0.8}{$\scriptstyle , $}\mkern-1mu\scalebox{0.8}{$\scriptstyle #2$}}^{\operatorname{#3}}}

\newcommand{\Uij}[2]{U_{\substack{\text{\scalebox{0.7}{$\mkern-1mu #1 \! \wedge\! #2$}}}}}

\newcommand{\subsm}[2]{{#1}_{\substack{\text{\scalebox{0.8}{$\mkern-0.8mu #2$}}}}}

\newcommand{\op}{{\substack{\text{\scalebox{0.7}{$\operatorname{op}$}}}}}

\newcommand{\Rel}[1]{\mathcal{R}_{\substack{\text{\scalebox{0.8}{$#1$}}}}}
\newcommand{\Co}[1]{{#1}_{\substack{\text{\scalebox{0.7}{$\mkern-1.5mu 0$}}}}}

\newcommand{\mbf}[1]{\mathbf{#1}}

\newcommand{\limi}[1]{\operatorname{lim}\!#1}
\newcommand{\Goi}[2]{{#1} \mkern-0.2mu (\mkern-1mu #2 \mkern-1mu)}
\newcommand{\Goij}[3]{{#1} \mkern-0.2mu (\mkern-1mu #2 \mkern-1mu ,\mkern-1mu #3 \mkern-1mu )}

\newcommand{\glI}[1]{\mathbb{Gl}(\mathrm{#1})}
\newcommand{\Gnij}[2]{{#1} \mkern-0.2mu (\mkern-1.2mu #2  \mkern-1.2mu )}

\newcommand{\dindi}[3]{{#1}_{{\substack{\text{\scalebox{0.8}{$\mkern-1mu #2$}}}}_{{\mkern-1mu}_{\substack{\text{\scalebox{0.7}{$\mkern-1mu #3$}}}}}}}

\newcommand{\dindiv}[4]{{#1}_{{\substack{\text{\scalebox{0.8}{$\mkern-1mu #2$}}}}_{{\mkern-1mu}_{\substack{\text{\scalebox{0.7}{$\mkern-1mu #3$}}}_{{}_{\substack{\text{\scalebox{0.6}{$\mkern-1mu #4$}}}}}}}}}

\newcommand{\dindgiv}[5]{{#1}_{{\substack{\text{\scalebox{0.6}{$\mkern-1mu #2$}}}_{{{}_{\substack{\text{\scalebox{0.5}{$\mkern-1mu #3$}}}_{{{}_{\substack{\text{\scalebox{0.6}{$\mkern-1.2mu #4\mkern-1.2mu$}}}}}_{{}_{\substack{\text{\scalebox{0.5}{$\mkern-1mu #5$}}}}}}}}}}}}

\newcommand{\dindsi}[3]{{#1}_{{\substack{\text{\scalebox{0.8}{$\mkern-1mu #2$}}}}_{{}_{\substack{\text{\scalebox{0.7}{$#3$}}}}}}}

\newcommand{\gcov}[1]{\bm{\mathbf C}_{\mkern-1mu\substack{\text{\scalebox{0.8}{$#1$}}}}\mkern-0.1mu}
\newcommand{\gcovv}[1]{\bm{\mathbf C}_{\mkern-1mu\substack{\text{\scalebox{0.8}{$#1$}}}}\mkern-0.1mu}

\newcommand{\proj}[2]{{\pi}_{\substack{\text{\scalebox{0.7}{$\mkern-1.5mu #1$}}}}^{\substack{\text{\scalebox{0.7}{$\mkern-1.5mu #2$}}}}}

\newcommand{\suchthat}{\;\ifnum\currentgrouptype=16 \middle\fi|\;}

\newcommand\restr[2]{{
\left.\kern-\nulldelimiterspace 
#1 
\vphantom{\big|} 
\right|_{#2} 
}}

\makeatletter
\newcommand*{\doublerightarrow}[2]{\mathrel{
		\settowidth{\@tempdima}{$\scriptstyle#1$}
		\settowidth{\@tempdimb}{$\scriptstyle#2$}
		\ifdim\@tempdimb>\@tempdima \@tempdima=\@tempdimb\fi
		\mathop{\vcenter{
				\offinterlineskip\ialign{\hbox to\dimexpr\@tempdima+1em{##}\cr
					\rightarrowfill\cr\noalign{\kern.5ex}
					\rightarrowfill\cr}}}\limits^{\!#1}_{\!#2}}}

\title{\bfseries A note on gluing: a pillar of algebraic geometry}
\author{Sophie Marques and Damas Mgani}

\begin{document}

\begin{abstract}


\noindent \textsf{2020 Math. Subject Class:} 14A15, 18F20, 18F60.

\noindent \begin{center}
\rm e-mail: \href{mailto:d smarques@sun.ac.za}{ smarques@sun.ac.za}

\it
Department of Mathematical Sciences, 
University of Stellenbosch, \\
Stellenbosch, 7600, 
South Africa\\
\&
NITheCS (National Institute for Theoretical and Computational Sciences), \\
South Africa \\ \bigskip

\rm e-mail: \href{mailto:d.mgani99@gmail.com}{d.mgani99@gmail.com}

\it
Department of Mathematical Sciences, 
University of Dar es Salaam, \\
Dar es Salaam, 
Tanzania.

\end{center} 

\end{abstract}

\setcounter{tocdepth}{3}
\maketitle
  \tableofcontents

\newpage
\begin{abstract} 
This paper examines the concept of gluing, placing it within its most general categorical context and tracing its foundational role in the broader architecture of algebraic geometry.
\end{abstract}

 \section*{Acknowledgement} 
We are grateful to George Janelidze and James Gray for their insightful suggestions regarding Section 2. in particular 2.1.1. We also thank St\'ephane Vinatier for his kind assistance in producing Figure~\ref{fig7}.

\section*{Introduction} 
Alexander Grothendieck introduced the concept of schemes in the mid-20th century, with its foundational development laid out in \cite{grothendieck1971}. This revolutionary notion emerged from several converging lines of inquiry: the need to systematically address local-global principles, the increasing prominence of descent theory and field extensions, and the earlier work of Weil on the foundations of algebraic varieties \cite{weil}. Grothendieck's central idea was to reinterpret a commutative ring as a geometric object, specifically, as a pair consisting of a topological space (the spectrum of the ring, endowed with the Zariski topology) and a structure sheaf of rings. This pairing captures both the local behavior of functions and their global organization, providing a powerful and flexible framework.

This construction gives rise to a new category, equivalent to the opposite of the category of commutative rings, but with a significant structural enhancement: it admits certain categorical limits, those arising from gluing. By formally closing this category under gluing operations, Grothendieck defined the category of schemes. This enriched framework not only encompasses classical objects such as projective spaces, but also offers a more robust and flexible setting for studying algebraic varieties. In particular, it ensures better behavior under standard constructions, such as fiber products, and accommodates more general forms of descent and base change.

We focus on this core construction that underlies much of the theory: the process of gluing. Our objective is to pause and examine the concept of gluing in depth, placing it within its most general categorical context and tracing its foundational role in the broader architecture of algebraic geometry. We take care to analyze its formal definition and explore its connections to key concepts in algebraic geometry. Through this investigation, we aim to highlight the central importance of gluing in the theory and to clarify its logical structure and far-reaching implications.

Some of the results presented may be familiar to readers with expertise in category theory or/and algebraic geometry, as these topics are well-documented in the existing literature. However, to the best of our knowledge, many of the results without a citation may not have been stated in the precise form we develop. Rather than relying directly on existing formulations, we have deliberately relaxed assumptions where possible or decomposed results into more elementary components in order to better isolate the essential conditions at play. This methodological choice reflects our broader aim, not merely to reproduce known results, but to clarify their logical foundations and to identify the minimal structural requirements for their validity. We hope that the clarity this approach has brought us may also prove useful to others in the field. Regarding notation, we have fused terminology from both algebraic geometry and category theory, aiming to strike a balance between the two perspectives.

The paper is structured as follows. 

In Section 2, following a brief review of notational conventions, we establish the categorical framework that underpins our study of gluing. We introduce two classes of indexing categories, split and non-split truncated power set categories, that form the basis for defining gluing via limits of what we call gluing functors (see Definition \ref{gluingfunctor}). In Section 2.1.1, we present the non-split version, originally suggested to us by James Gray and George Janelidze, which, as they anticipated, proves to be theoretically sufficient. However, for practical reasons that arise in Section 3 and in the final section, we found it more convenient to work with the split variant, which aligns more closely with our original formulation; this motivation and distinction is explained in Section 2.1.2. As a result, both approaches are included, and we default to the simpler non-split version whenever feasible. Within this framework, we formalize the notion of universal gluing in Section 2.4. In Section 2.5, we demonstrate that in categories equipped with products and equalizers, the gluing process reduces, as anticipated, to a standard equalizer construction. In section 2.6, we also identify and examine a class of gluable categories, in which every functor from a truncated power set category admits a limit; Grothendieck categories emerge as an example. In Section 2.7, we analyze how the {\sf Hom} functor interacts with gluing, providing the first indication of its connection to Grothendieck topologies and sheaf theory. Finally, in Section 3, we introduce the concept of refinement between gluing functors and describe how such functors can be formally composed.

Section 3 shifts focus to Grothendieck topologies and sites, where we revisit the standard definitions. To each sink, we associate a canonical functor. The glued object, defined as the limit of a gluing functor, naturally determines a new sink, which in turn gives rise to a gluing functor (see Lemma \ref{gluthru}).  In section 4.2, we introduce the notion of effective gluing on a Grothendieck site. This generalization of the cocycle condition is carefully examined, and we verify that it coincides with the classical condition in the typical context (see Proposition \ref{eqrel}). We justify the term "effective" by proving that the canonical functor is an effective split gluing functor if and only if the family defined by covering maps forms an effective epimorphism (Proposition \ref{lemm39}),  connecting the notion of effective epimorphism with gluing. In section 4.4, we propose an alternative approach to composing gluing functors. The section concludes with the definition of (pre)sheaves in terms of gluing and recalls the concepts of direct image and restriction of a presheaf.

In Section 6, we examine the gluing of morphisms of (pre)sheaves and (pre)sheaves. While presheaves, separated presheaves, and sheaves can be glued using similar techniques, we identify a crucial limitation: the class of gluable presheaves lacks a straightforward categorical characterization. In fact, we have not yet found a general categorical expression for this property, and our results suggest it may not behave well from a categorical perspective. We intentionally omit the cocycle condition from the outset, as it is well-known and easily verified to be unnecessary for the construction. Instead, we reinterpret the cocycle condition as a property of gluing functors and explore its role through the definition of effective gluing functors in Section 6.3. This shift in perspective allows us to isolate the core of the cocycle condition within the broader categorical structure (see Proposition \ref{preashf}).

In Section 7, we introduce the notion of enriched (pre)sheaves, extending the idea of ringed topological spaces to a more general categorical setting. Here, we simultaneously consider an object in the site and an associated (pre)sheaf, and we develop a general gluing framework that encompasses the classical case of ringed spaces. We provide conditions under which gluing functors can be induced both on the site and on the (pre)sheaves, and we show how gluing in the enriched setting arises from these two components.

While much remains to be explored in the foundations of algebraic geometry, we found that pausing to examine the notion of gluing proved to be particularly insightful. Along the way, we identified several instances where we had previously overlooked which assumptions were truly necessary and why. By stepping back and attempting to generalize constructions to arbitrary categories wherever possible, and by refining statements to their essential content, we gained a clearer understanding of where the core of the concept lies. This reflection ultimately shaped the focus and motivation of the present paper.

\newpage
\section{Notation}
\noindent In this section, we introduce a few notation that will be used consistently throughout the paper.
\noindent Given $X,Y, U$ objects of $\mathbb{C}$, $f: X\rightarrow U$, $g: Y\rightarrow U$ be morphisms such that $X\subsm{\times}{\scalebox{0.6}{$U$}} Y$ exists, we denote $\proj{k}{X\subsm{\times}{\scalebox{0.6}{$U$}} Y}$, the canonical $k^{th}$ projection from the pullback identified by the subscript, for any $k\in \{ 1, 2\}$. Given a family \(\subsm{(U_i)}{i\!\in\! \rm{I}}\) of objects in \(\mathbb{C}\):  
\begin{itemize}  
\item If the coproduct \(\subsm{\coprod}{i\!\in\! \rm{I}} U_i\) exists and, for each \(i \in \rm{I}\), we are given a morphism \(\eta_i: U_i \rightarrow Z\) in \(\mathbb{C}\), we denote by \(\subsm{\langle \eta_i \rangle}{i\!\in\! \rm{I}}\) the unique morphism from \(\subsm{\coprod}{i\!\in\! \rm{I}} U_i\) to \(Z\) determined by the universal property of the coproduct.  

\item If the product \(\subsm{\prod}{i\!\in\! \rm{I}} U_i\) exists and, for each \(i \in \rm{I}\), we are given a morphism \(\rho_i: Z \rightarrow U_i\) in \(\mathbb{C}\), we denote by \(\subsm{\langle \rho_i \rangle}{i\!\in\! \rm{I}}\) the unique morphism from \(Z\) to \(\subsm{\prod}{i\!\in\! \rm{I}} U_i\) determined by the universal property of the product.  
\end{itemize}
In this paper, we adopt the following notational conventions:
\begin{itemize}
    \item $\mathbb{C}$ denotes an arbitrary category.
    \item $\mathrm{I}$ denotes an index set.
    \item $R$ denotes a unital commutative ring.
    \item $\mathbb{Sets}$ denotes the category of sets.
    \item $\mathbb{Top}$ denotes the category of topological spaces, and $\mathbb{oTop}$ denotes the subcategory of $\mathbb{Top}$ whose morphisms are open continuous maps.
    \item $\mathbb{Rings}$ denotes the category of (not necessarily commutative) rings.
    \item $\mathbb{Ab}$ denotes the category of abelian groups.
    \item $\mathbb{Grp}$ denotes the category of groups.
    \item $R\text{-}\mathbb{Mod}$ denotes the category of $R$-modules.
    \item $R\text{-}\mathbb{Alg}$ denotes the category of $R$-algebras.
\end{itemize}
For foundational concepts in category theory, we refer the reader to seminal works such as \cite{adamek2009abstract}, \cite{Joji}, and \cite{maclane}. 
\section{Defining gluing categorically} 
\subsection{Categories of interest} 
In this section, we define two categories that will serve as the indexing categories for the diagrams used to construct limits in the context of gluing.
\subsubsection{Truncated power set category of $\mathrm{I}$} \label{truncated}
We begin by defining a truncated power set category of a set. Specifically, this paper will focus on truncated power set categories of order $2$ and define gluing across categories.

\begin{definition} \label{trunc}  
Let $n\in \mathbb{N}$. The {\sf truncated power set category of \(\mathrm{I}\) of order \(n\)}, denoted by \(\mathbb{P}_n (\mathrm{I})\), is defined as follows:  
\begin{enumerate}  
    \item {\sf Objects:} The subsets \(S \subseteq \mathrm{I}\) such that \(1 \leq |S| \leq n\).  
    \item {\sf Morphisms:} Inclusions between these subsets. For any subsets \(S, T \subseteq \mathrm{I}\) with \(S \subseteq T\), the inclusion morphism is denoted by \(\mathfrak{i}_{S, T}\).   
    We use the notation \(\mathfrak{i}_{i, j}\) to represent the inclusion \(\{i\} \subseteq \{i, j\}\) for all \(i, j \in \mathrm{I}\). Moreover, the inclusion \(\mathfrak{i}_{i, i}\) corresponds to the identity morphism on \(\{i\}\) for all \(i \in \mathrm{I}\).  
\end{enumerate}  
In other words, \(\mathbb{P}_n (\mathrm{I})\) is a poset category, which is the full subcategory of the power set category, whose objects are precisely the subsets of \(\mathrm{I}\) with at least one element and at most \(n\) elements.
\end{definition}

\noindent In subsequent examples, we will illustrate the truncated power set categories of order $2$ and $3$ for a set $\mathrm{I}$, assuming $\mathrm{I}$ has a cardinality of four or fewer elements.

\begin{example} 
\begin{enumerate} 
\item The following three figures represent  the truncated power set category of $\mathrm{I}$ of order $2$ omitting the identity morphisms and the names of the morphisms where $\mathrm{I}=\{i,j\}$, $\mathrm{I}=\{i,j, k \}$ and $\mathrm{I}=\{i,j, k,l \}$ respectively where $i,j,k, l $ are distinct elements:

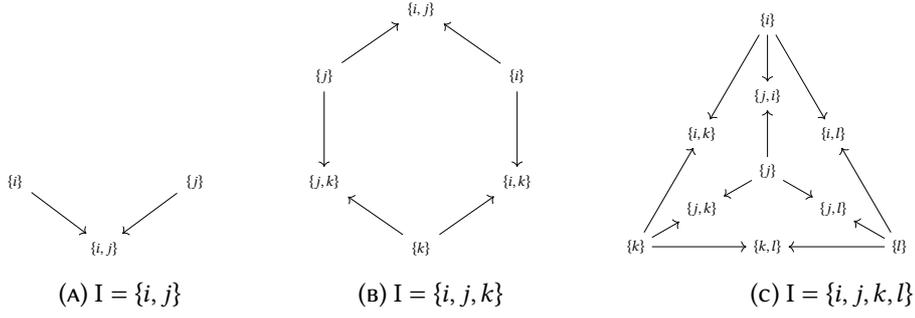
\begin{figure}[H]
    \centering
    \begin{subfigure}[b]{0.2\textwidth}
       {\tiny\begin{tikzcd}[column sep=normal]
\scalebox{0.7}{$\{ i\}$} \arrow{rd}&           & \scalebox{0.7}{$\{j\}$} \arrow{ld}\\
                                             & \scalebox{0.7}{${\{i,j\}}$} &                                            
\end{tikzcd}}
        \caption{$\mathrm{I}=\{i,j\}$}
    \end{subfigure}
    \begin{subfigure}[b]{0.3\textwidth}
     {\tiny\begin{tikzcd}[column sep=normal]
      & \\ & & \scalebox{0.7}{${\{i,j\}}$} & \\
      & \scalebox{0.7}{$\{j\}$} \arrow{ru} \arrow{dd} &                                                                                                           & \scalebox{0.7}{$\{i\}$} \arrow{lu} \arrow{dd} \\&                                                                                                &  &  \\& \scalebox{0.7}{${\{j,k\}}$} & & \scalebox{0.7}{${\{i,k\}}$} \\&                                                                                                           & \scalebox{0.7}{$\{k\}$} \arrow{lu} \arrow{ru} &                                                                                                          
\end{tikzcd}}

        \caption{$\mathrm{I}=\{i,j,k\}$}
    \end{subfigure}
    ~ 
    \begin{subfigure}[b]{0.35\textwidth}

{\tiny\begin{tikzpicture}[scale=1.0, commutative diagrams/every diagram,
    declare function={R=2;Rs=R*cos(60);}]
 \path
  (150:Rs) node(A) {\scalebox{0.7}{$\{i,k\}$}} 
  (270:Rs) node(B) {\scalebox{0.7}{$\{k,l\}$}} 
  (30:Rs) node(C) {\scalebox{0.7}{$\{i,l\}$}} 
  (0,0)  node(D) {\scalebox{0.7}{$\{j\}$}} 
  (-30:R) node (X) {\scalebox{0.7}{$ \{l\}$}} 
  (90:R) node (Y) {\scalebox{0.7}{$ \{i\}$}}
  (90:0.5*R) node (J) {\scalebox{0.7}{$\{j,i\}$}}
   (-30:0.5*R) node (K) {\scalebox{0.7}{$\{j,l\}$}}
    (210:0.5*R) node (L) {\scalebox{0.7}{$\{j,k\}$}}
           
  (210:R) node (Z) {\scalebox{0.7}{$ \{k\}$}};
 \path[commutative diagrams/.cd, every arrow, every label]
 (D) edge (K)
 (D) edge (J)
 (Y) edge (A)
 (Z) edge (A)
  (D) edge (L)
   (Z) edge (L)
   (X) edge (K)
  
   (Y) edge (J)
 (Z) edge (B) 
 (X) edge (B)
  (X) edge(C)
  (Y) edge (C)

  
   
  
;

   
\end{tikzpicture}}
 \\
   \caption{$\mathrm{I}=\{i,j,k,l\}$}
  
    \end{subfigure}
\end{figure} 
\item The following three figures represent the truncated power set category of $\mathrm{I}$ of order $3$ omitting the identity morphisms and the names of the morphisms where $\mathrm{I}=\{i,j, k \}$ and $\mathrm{I}=\{i,j, k,l \}$ respectively where $i,j,k, l $ are distinct elements:
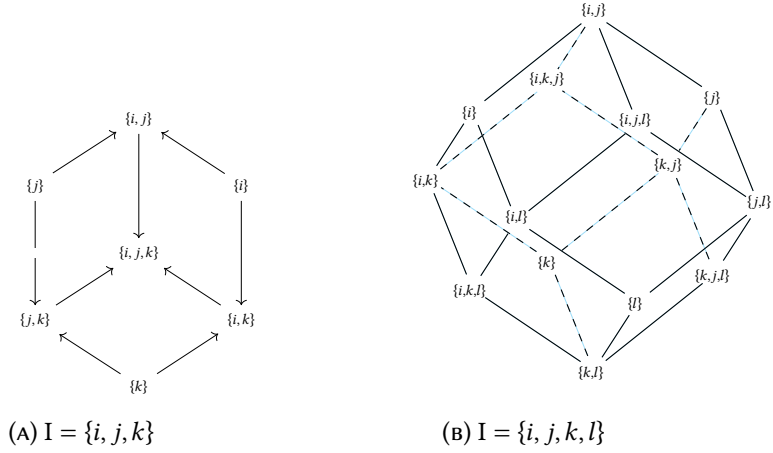
\hspace{2em}
\begin{figure}[H]
    \begin{subfigure}[b]{0.2\textwidth}
     {\tiny\begin{tikzcd}[column sep=normal]
                                                                                                &                                                                                                           \\
      &                                                                                                           & \scalebox{0.7}{${\{i,j\}}$} \arrow[]{dd}{}                                          &                                                                                                           \\
      & \scalebox{0.7}{$\{j\}$} \arrow[]{ru}{} \arrow[labels=description]{dd}{} &                                                                                                           & \scalebox{0.7}{$\{i\}$} \arrow[]{lu}{} \arrow[]{dd}{} \\
      &                                                                                                   &   \scalebox{0.7}{$ \{i,j,k\}$}
                                                                                               &                                                                                                        \\
      & \scalebox{0.7}{${\{j,k\}}$} \arrow[]{ru}{}                                          &                                                                                                           & \scalebox{0.7}{${\{i,k\}}$} \arrow[]{lu}{}                                         \\
      &                                                                                                           & \scalebox{0.7}{$\{k\}$} \arrow[]{lu}{} \arrow[]{ru}{} &                                                                                                          
\end{tikzcd}}

        \caption{$\mathrm{I}=\{i,j,k\}$}
    \end{subfigure}
     \qquad  \qquad \qquad
    \begin{subfigure}[b]{0.2\textwidth}
    
\tdplotsetmaincoords{80}{115}
\begin{tikzpicture}[scale=0.8, tdplot_main_coords,baseline]
	\pgfmathsetmacro{\arista}{3}
	\pgfmathsetmacro{\unidad}{1.0}
	\coordinate[label=center:$\scalebox{0.7}{$\substack{ \ \\ \\
	 \ \{i,l\} \ \\   \\ \ }$}$, circle]  (13) at (\arista,0,0);
	\coordinate[label=center:$\scalebox{0.7}{$\substack{\ \\ \\ \\ \\ \  \{j,l\} \ \\ \\ \\ \\ \  }$}$, circle]  (14) at (0,\arista,0);
	\coordinate[label=center:$\scalebox{0.7}{$\substack{\ \\  \{k,j\} \ \\  }$}$, circle]  (15) at (-\arista,0,0);
	\coordinate[label=center:$\scalebox{0.7}{$\substack{\ \\ \\ \{i,k\} \ \\ \\ }$}$, circle]  (16) at (0,-\arista,0);
	\coordinate[label=center:$\scalebox{0.7}{$\substack{\ \\ \\ \\ \ \{i,j\} \ \\ \\ \\ \ }$}$, circle]  (17) at (0,0,\arista);
	\coordinate[label=center:$\scalebox{0.7}{$\substack{\  \\ \\ \{k,l\} \  \\ \\ }$}$, circle]  (18) at (0,0,-\arista);
	\coordinate[label=center:$\scalebox{0.7}{$\substack{\ \\ \\ \{j\} \ \\ \\ }$}$, circle]  (19) at (-0.5*\arista,0.5*\arista,0.5*\arista);
	\coordinate[label=center:$\scalebox{0.7}{$\substack{\ \\  \{i,j,l\} \ \\ \\ }$}$, circle]  (20) at (0.5*\arista,0.5*\arista,0.5*\arista);
	\coordinate[label=center:$\scalebox{0.7}{$\substack{\ \\  \{i,k,j\} \ \\ \\ }$}$, circle]  (21) at (-0.5*\arista,-0.5*\arista,0.5*\arista);
	\coordinate[label=center:$\scalebox{0.7}{$\substack{\ \\ \\ \{i\} \ \\ \\ }$}$, circle]  (22) at (0.5*\arista,-0.5*\arista,0.5*\arista);
	\coordinate[label=center:$\scalebox{0.7}{$\substack{\ \\ \\ \{k,j,l\} \ \\ \\ }$}$, circle]  (23) at (-0.5*\arista,0.5*\arista,-0.5*\arista);
	\coordinate[label=center:$\scalebox{0.7}{$\substack{\ \\ \\ \{l\} \ \\ \\ }$}$, circle]  (24) at (0.5*\arista,0.5*\arista,-0.5*\arista);
	\coordinate[label=center:$\scalebox{0.7}{$\substack{\ \\ \\ \{k\} \ \\ \\ }$}$, circle]  (25) at (-0.5*\arista,-0.5*\arista,-0.5*\arista);
	\coordinate[label=center:$\scalebox{0.7}{$\substack{\ \\ \\ \{i,k,l\} \ \\ \\ }$}$, circle]  (26) at (0.5*\arista,-0.5*\arista,-0.5*\arista);

	\draw[dashed,fill=cyan!35,opacity=0.75] (25) -- (15) ;
	
		\draw[fill=cyan!35,opacity=0.75]  (18) -- (23) ;	
		
			\draw[dashed,fill=cyan!35,opacity=0.75]  (25) -- (18);
			
				\draw[dashed,fill=cyan!35,opacity=0.75]   (15)-- (23);

		\draw[dashed,fill=cyan!35,opacity=0.75] (16) -- (21) ;

			
				\draw[dashed,fill=cyan!35,opacity=0.75]   (15)-- (21);

	
		\draw[dashed,fill=cyan!35,opacity=0.75]  (19) -- (15) ;	
		
			

		\draw[fill=cyan!35,opacity=0.75] (22) -- (17) ;
	
		
			
				\draw[dashed,fill=cyan!35,opacity=0.75]   (17)-- (21);

	
		
			\draw[fill=cyan!35,opacity=0.75]  (14) -- (23);
			
			

		\draw[fill=cyan!35,opacity=0.75] (16) -- (26) ;
	
		
			
				\draw[dashed,fill=cyan!35,opacity=0.75]   (25)-- (16);

	
		
			\draw[fill=cyan!35,opacity=0.75]  (13) -- (26);
			
				\draw[fill=cyan!35,opacity=0.75]   (22)-- (16);
	
	
		\draw[fill=cyan!35,opacity=0.75]  (24) -- (18) ;	
		
			
				\draw[fill=cyan!35,opacity=0.75]   (24)-- (14);

		\draw[fill=cyan!35,opacity=0.75] (17) -- (20) ;
	
		\draw[fill=cyan!35,opacity=0.75]  (19) -- (14) ;	
		
			\draw[fill=cyan!35,opacity=0.75]  (14) -- (20);
			
				\draw[fill=cyan!35,opacity=0.75]   (19)-- (17);

	
		\draw[fill=cyan!35,opacity=0.75]  (22) -- (13) ;
		
			\draw[fill=cyan!35,opacity=0.75]  (13) -- (20);
			

	
		\draw[fill=cyan!35,opacity=0.75]  (18) -- (26) ;	
		
			

	
		\draw[fill=cyan!35,opacity=0.75]  (24) -- (13) ;	
		
			
	
	%
	\end{tikzpicture} \\

        \caption{$\mathrm{I}=\{i,j,k,l\}$}
    \end{subfigure}
   
\end{figure} 
\end{enumerate} 
\end{example}

\subsubsection{Split truncated power set category of $\rm{I}$ of order $2$}
In this section, we introduce a split variant of the category defined in Section~\ref{truncated}. This refinement becomes especially pertinent in Section~\ref{grothe}, where Grothendieck topologies are discussed and constructions such as fibered products \( U_i \times_U U_i \) play a central role. The split version also proves advantageous in concrete contexts, for example, in Section~6, where working without the axiom of choice is desirable. Furthermore, in Section~4.2, effective gluing functors are defined only within the split framework. Nonetheless, we introduce a stronger notion of effectiveness that is specifically adapted to the non-split setting.


\begin{definition} \label{trunc}  
The \textsf{split truncated power set category of \rm{I} of order $2$}, denoted by \(\mathbb{S}_2(\mathrm{I})\), is defined as follows:

\begin{enumerate}
    \item \textsf{Objects:} The objects are elements of \(\mathrm{I}\) and ordered pairs \((i,j) \in \mathrm{I}^2\). That is, 
    \[
    \mathrm{Ob}(\mathbb{S}_2(\mathrm{I})) = \mathrm{I} \cup \mathrm{I}^2.
    \]

    \item \textsf{Morphisms:} The morphisms are precisely as follows:
    \begin{itemize}
        \item For each \(i, j \in \mathrm{I}\), there is a morphism from \(i\) to \((i,j)\), denoted \(\mathfrak{i}^s_{i,j} : i \to (i,j)\).
        \item For each \(i, j,k \in \mathrm{I}\), there is an invertible morphism \(\tau_{i,j} : (i,j) \to (j,i)\) that is not identity when $i=j$ such that $\tau_{i,j}^{-1}
= \tau_{j,i}$.      
\item For every object \(x \in \mathbb{S}_2(\mathrm{I})\), there exists an identity morphism \(\mathrm{id}_x : x \to x\).
        \item All morphisms are composable whenever the domain and codomain match appropriately.
    \end{itemize}
\end{enumerate}

\end{definition}

\begin{example}
	The following three figures represent  the split truncated power set category of \rm{I} of order $2$ omitting the identities morphisms when $\mathrm{I}=\{i,j\}$, and  $\mathrm{I}=\{i,j, k \}$ respectively:
	\hspace{2em}

\begin{enumerate}
	\item 
	
	\begin{figure}[H]
		\centering
		
		{\tiny\begin{tikzcd}[scale=0.2]
				& 	\scalebox{0.7}{$\{i\}$} \arrow[]{ld}{} \arrow[]{rd}{} &                                               &  &                                               & 	\scalebox{0.7}{$\{j\}$} \arrow[]{ld}{} \arrow[]{rd}{} &         \\
				\scalebox{0.7}{${\{i,i\}}$} \arrow[ loop below]{}{} \arrow[ loop above]{}{} &                                                            & \scalebox{0.7}{${\{i,j\}}$} \arrow[ bend left]{rr}{}  &  & \scalebox{0.7}{${\{j,i\}}$} \arrow[ bend left]{ll}{} &                                                            & \scalebox{0.7}{${\{j,j\}}$}\arrow[ loop below]{}{} \arrow[ loop above]{}{}
		\end{tikzcd}}
	\end{figure}
	
	\item
	\begin{figure}[H]
		{\tiny 
			\begin{tikzcd}[scale=0.2,row sep=normal]
				&  & {\scalebox{0.7}{$\{i,j\}$}}  \arrow[ bend left ]{rrrr}{} &  &  &  & {\scalebox{0.7}{$\{j,i\}$}} \arrow[ bend left]{llll}{} &    &   \\
				{\scalebox{0.7}{$\{i,i\}$}}\arrow[ loop below]{}{} \arrow[ loop above]{}{}  & {\scalebox{0.7}{$\{i\}$}} \arrow[]{ru}{} \arrow[]{l}{} \arrow[]{d}{} &   &  &  &  &  & {\scalebox{0.7}{$\{j\}$}} \arrow[]{lu}{}\arrow[]{r}{} \arrow[]{d}{} & {\scalebox{0.7}{$\{j,j\}$}} \arrow[ loop below]{}{} \arrow[ loop above]{}{} \\
				& {\scalebox{0.7}{$\{i,k\}$}} \arrow[ bend left]{rd}{}  &   &  &   &  & & {\scalebox{0.7}{$\{j,k\}$}} \arrow[ bend right]{ld}{}         &   \\
				&  & {\scalebox{0.7}{$\{k,i\}$}} \arrow[ bend left]{lu}{}   &  &   &  & {\scalebox{0.7}{$\{k,j\}$}} \arrow[ bend right]{ru}{}  &  &   \\
				&   &  &  & {\scalebox{0.7}{$\{k\}$}} \arrow[]{rru}{} \arrow[]{llu}{} \arrow[]{d}{} &  & & &   \\
				&  &   &  & {\scalebox{0.7}{$\{k,k\}$}} \arrow[ loop right]{}{} \arrow[ loop left]{}{}  &  &   &  &  
			\end{tikzcd}
		}
	\end{figure}

\end{enumerate}
\end{example}

\begin{remark}
The gluing category \(\mathbb{S}_2(\mathrm{I})\) can be understood as a splitting refinement of $\mathbb{P}_2(I)$. This category refines the structure by distinguishing between ordered pairs and introducing morphisms that represent gluing relations. Those categories are not equivalent.
\end{remark}

\noindent The advantage of using the split truncated power set category lies in its practical convenience: within concrete categories, it avoids the need to explicitly sort unordered pairs using an auxiliary map. This makes it easier to work with in applications. Nonetheless, as we will see, from a purely theoretical perspective, the (non-split) truncated power set category is sufficient to define a robust notion of a gluing functor. 
\begin{definition}
A \textsf{pair sorting map of \( \mathrm{I} \)} is a function
\[
c: \operatorname{Ob}(\mathbb{P}_2(\mathrm{I})) \longrightarrow \operatorname{Ob}(\mathbb{S}_2(\mathrm{I}))
\]
that assigns to each unordered pair \( \{i, j\} \subseteq \mathrm{I} \) (with \( i \neq j \)) a chosen ordered pair \( (i, j) \) or \( (j, i) \), i.e.,
\[
c(\{i, j\}) \in \{(i, j), (j, i)\},
\]
in such a way that the assignment is made systematically for all \( i, j \in \mathrm{I} \) thanks to the axiom of choice. For singleton sets \( \{i\} \), the map is defined by \( c(\{i\}) = (i) \).
\end{definition}

\begin{definition}
Let \( \mathrm{I} \) be an index set, and let \( c \) be a pair sorting map on \( \mathrm{I} \). We define the following functors:

\begin{itemize}
    \item The functor \( \mathbf{A}_c : \mathbb{P}_2(\mathrm{I}) \to \mathbb{S}_2(\mathrm{I}) \):
    \begin{itemize}
        \item On objects, \( \mathbf{A}_c \) acts as \( c \).
        \item On morphisms, given \( i,j \in \mathrm{I} \) with $i\neq j$ and \( c(\{i,j\}) = (i,j) \), it sends \( \mathfrak{i}_{i,j} \mapsto \mathfrak{i}^s_{i,j} \) and \( \mathfrak{i}_{j,i} \mapsto \tau_{j,i} \circ \mathfrak{i}^s_{j,i} \).
    \end{itemize}

    \item The functor \( \mathbf{B}_I : \mathbb{S}_2(\mathrm{I}) \to \mathbb{P}_2(\mathrm{I}) \): for all \( i,j \in \mathrm{I} \) with $i \neq j$,
    \begin{itemize}
        \item On objects, it sends \( (i) \mapsto \{i\} \) and \( (i,j) \mapsto \{i,j\} \).
        \item On morphisms, it sends \( \tau_{i,j} \mapsto \operatorname{id}_{\{i,j\}} \), \( \tau_{i,i} \mapsto \operatorname{id}_{\{i\}} \), \( \mathfrak{i}^s_{i,i} \mapsto \operatorname{id}_{i} \) and \( \mathfrak{i}^s_{i,j} \mapsto \mathfrak{i}_{i,j} \),
    \end{itemize}
    
    \item The functor \( \mathbf{A}_c' : \mathbb{P}_2(\mathrm{I} \coprod \mathrm{I}) \to \mathbb{S}_2(\mathrm{I}) \):
    \begin{itemize}
        \item On objects, it sends, for all $i,j \in \mathrm{I}$ with $i\neq j$, \( c(\{i,j\}) = (i,j) \), and $k \in \{1,2\}$,
        \[
        (i,k) \mapsto i, \quad    \{(i,1), (i,2)\} \mapsto (i,i),
        \]
        \[
        \{(i,1), (j,k)\} \mapsto (i,j), \quad \{(i,2), (j,k)\} \mapsto (j,i).
        \]
        \item On morphisms, it sends:
        \begin{align*}
            \mathfrak{i}_{(i,1),(j,k)} &\mapsto \mathfrak{i}^s_{i,j}, \\
            \mathfrak{i}_{(j,k),(i,1)} &\mapsto \tau_{j,i} \circ \mathfrak{i}^s_{j,i}, \\
            \mathfrak{i}_{(i,2),(j,k)} &\mapsto \tau_{i,j} \circ \mathfrak{i}^s_{i,j}, \\
            \mathfrak{i}_{(j,k),(i,2)} &\mapsto \mathfrak{i}^s_{j,i}.
        \end{align*}
    \end{itemize}
\end{itemize}

\noindent We note that \( \mathbf{B}_I \circ \mathbf{A}_c = \operatorname{id}_{\mathbb{P}_2(\mathrm{I})} \), i.e., \( \mathbf{B}_I \) is a left inverse to \( \mathbf{A}_c \).
\end{definition}
\subsection{Gluing functors}
We introduce gluing functors as tools for formalizing gluing operations across various mathematical categories.

\begin{definition} \label{gluingfunctor}
Let \( \mathrm{I} \) be a set and let \( \mathbb{C} \) be a category.  
\begin{enumerate}
    \item A \(\mathbb{C}\)\textsf{-gluing functor of type \(\mathrm{I}\)} (respectively, a split \(\mathbb{C}\)\textsf{-gluing functor of type \(\mathrm{I}\)}) is a functor 
    \[
        \mathbf{G} : \mathbb{P}_2(\mathrm{I}) \to \mathbb{C} \quad \text{(resp. } \mathbf{G} : \mathbb{S}_2(\mathrm{I}) \to \mathbb{C} \text{)}
    \]
    such that the limit \( \lim \mathbf{G} \) exists in \( \mathbb{C} \). For convenience, we write \( \mathbf{G}(i) := \mathbf{G}(\{i\}) \) (resp. \( \mathbf{G}(i) := \mathbf{G}((i)) \)) and \( \mathbf{G}(i,j) := \mathbf{G}(\{i,j\}) \) for all \( i, j \in \mathrm{I} \) with $i\neq j$ (resp. \( \mathbf{G}(i,j) := \mathbf{G}((i,j)) \) for all \( i, j \in \mathrm{I} \)).

    \item A functor \( \mathbf{G} \) is called a \(\mathbb{C}\)\textsf{-gluing functor} (resp. split \(\mathbb{C}\)\textsf{-gluing functor}) if there exists a set \( \mathrm{J} \) such that \( \mathbf{G} \) is a gluing functor of type \( \mathrm{J} \).

    \item Let \( \mathbf{G} \) be a \(\mathbb{C}\)-gluing functor of type \( \mathrm{I} \) (respectively, a split \(\mathbb{C}\)-gluing functor of type \(\mathrm{I}\)). If \( (L, (\pi_a^L)_{a \in \mathbb{P}_2( \mathrm{I})}) \) (respectively, \( (L, (\pi_a^L)_{a \in \mathbb{S}_2( \mathrm{I})}) \)) is a cone over \( \mathbf{G} \) that is isomorphic, as a cone, to \( \lim \mathbf{G} \), we say that \( L \) is a \textsf{glued-up object over \( \mathbf{G} \)} through \( (\pi_i^L)_{i \in \mathrm{I}} \).

    \item The category \( \mathbb{C} \) is said to be \textsf{gluable} if for any set \( \mathrm{J}\), every functor from \( \mathbb{P}_2(\mathrm{J}) \) to \( \mathbb{C} \) is a \(\mathbb{C}\)-gluing functor.

\end{enumerate}
\end{definition}

\begin{remark}
\begin{itemize}
    \item Any functor \( \mathbf{G} : \mathbb{P}_2(\mathrm{I}) \to \mathbb{C} \) defines a diagram indexed by a poset category, and thus induces a direct system in \( \mathbb{C} \). Therefore, whenever such a functor \( \mathbf{G} \) is a \(\mathbb{C}\)-gluing functor of type \( \mathrm{I} \), the limit \( \lim \mathbf{G} \) is a direct limit.
    
    \item Given a $\mathbb{C}$-gluing functor $\mathbf{G}$, a glued up object is a joint pullbacks $\mathbf{G}(a)$ and $\mathbf{G}(b)$ over $\mathbf{G}(a,b)$ for all $a, b\in I$ with $| \{a, b\} |=2 $.
    \item Let \( c \) be a pair sorting map on \( \mathrm{I} \), and let \( \mathbf{G} : \mathbb{S}_2(\mathrm{I}) \to \mathbb{C} \) be a functor such that:
    \[
    \mathbf{G}(i,i) = \mathbf{G}(i), \quad \mathbf{G}(\mathfrak{i}^s_{i,i}) = \mathbf{G}(\tau_{i,i}) = \operatorname{id}_{\mathbf{G}(i)} \quad \text{for all } i \in \mathrm{I}.
    \] 
We note that these assumptions align with the classical gluing data used in algebraic geometry.
    Then, the cones over \( \mathbf{G} \) are in one-to-one correspondence with the cones over \( \mathbf{G} \circ \mathbf{A}_c \), where \( \mathbf{A}_c : \mathbb{P}_2(\mathrm{I}) \to \mathbb{S}_2(\mathrm{I}) \) is the functor associated with the pair sorting map \( c \).
    It follows that \( \mathbf{G} \) is a \(\mathbb{C}\)-split gluing functor if and only if \( \mathbf{G} \circ \mathbf{A}_c \) is a \(\mathbb{C}\)-gluing functor. \\
\noindent Furthermore, given any \(\mathbb{C}\)-gluing functor \( \mathbf{G} : \mathbb{P}_2(\mathrm{I}) \to \mathbb{C} \), there is a one-to-one correspondence between cones over \( \mathbf{G} \) and cones over \( \mathbf{G} \circ \mathbf{A}'_c \), and hence:
    \[
    \mathbf{G} \text{ is a } \mathbb{C}\text{-split gluing functor } \iff \mathbf{G} \circ \mathbf{A}'_c \text{ is a } \mathbb{C}\text{-gluing functor}.
    \]
In the remainder of this section, we work exclusively with gluing functors. The corresponding results for split gluing functors can be obtained by selecting a pair sorting map \( c \) on \( \mathrm{I} \), composing a given split gluing functor \( \mathbf{G} \) with the associated functor \( \mathbf{A}_c' \), and applying the results to the composite \( \mathbf{G} \circ \mathbf{A}_c' \). 
\end{itemize}
\end{remark}

\begin{example}
Assume $\mathbb{C}$ includes a terminal object denoted by $\mathbf{1}$. For a set $\mathrm{I}$, we consider a functor $\mathbf{G}: \mathbb{P}_2(\mathrm{I}) \rightarrow \mathbb{C}$ such that $\mathbf{G}(i, j) = \mathbf{1}$ for all $i, j \in \mathrm{I}$ with $i\neq j$. If $\lim \mathbf{G}$ exists, it is the product $\prod_{i \in \mathrm{I}} \mathbf{G}(i)$.
\end{example}

\begin{example}\label{pushut} Illustrations depict glued-up objects for \(\mathrm{I}\) with cardinality less than or equal to 3,  for each corresponding set size when  $\mathbf{G}(i, i) = \mathbf{G}(i)$ and $\mathbf{G}( \tau_{i,i})= \operatorname{id}_i$, for all $i \in \mathrm{I}$:

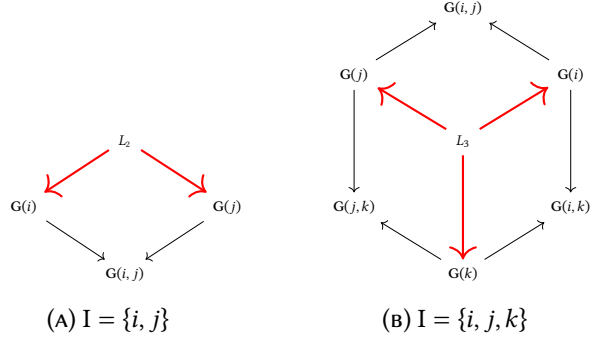
\begin{figure}[H]
    \centering
    \begin{subfigure}[b]{0.26\textwidth}
       {\tiny  \begin{tikzcd}[column sep=normal]
                                                                                                &                                                                                                           \\
      &                                                                                                           &                                          &                                                                                                           \\
      &   &                                                                                                           & \\
      &                                                                                                   &   \scalebox{0.7}{$ \subsm{L}{2}$}\arrow[thick,red]{ld}{} \arrow[red,thick]{rd}{}  
                                                                                               &                                                                                                        \\
      & \scalebox{0.7}{${\mathbf{G}}(i)$}     \arrow[]{rd}{}                                     &                                                                                                           & \scalebox{0.7}{${\mathbf{G}}(j)$} \arrow[]{ld}{}                                       \\
      &                                                                                                           & \scalebox{0.7}{${\mathbf{G}}(i,j)$}   &                                                                                                          
\end{tikzcd}}
        \caption{$\mathrm{I}=\{i,j\}$}
    \end{subfigure}
    \begin{subfigure}[b]{0.3\textwidth}
    
 {\tiny\begin{tikzcd}[column sep=normal]
      & \\ & & \scalebox{0.7}{${\mathbf{G}}(i,j)$} & \\
      & \scalebox{0.7}{${\mathbf{G}}(j)$} \arrow[]{ru}{} \arrow[]{dd}{} &                                                                                                           &  \scalebox{0.7}{${\mathbf{G}}(i)$} \arrow[]{lu}{} \arrow[]{dd}{} \\&                                                                                                &   \scalebox{0.7}{$L_3$} \arrow[thick,red]{lu}{} \arrow[thick,red]{ru}{} \arrow[thick,red]{dd}{} &  \\& \scalebox{0.7}{${\mathbf{G}}(j,k)$} & & \scalebox{0.7}{${\mathbf{G}}(i,k)$} \\&                                                                                                           & \scalebox{0.7}{${\mathbf{G}}(k)$} \arrow[]{lu}{} \arrow[]{ru}{} &                                                                                                          
\end{tikzcd}}
  \caption{$\mathrm{I}=\{i,j,k\}$}
      
    \end{subfigure}
    ~ 
\end{figure}
\end{example}

\subsection{Universal Gluing}

In this section, we introduce the concept of a \emph{universal gluing functor}.

\begin{definition}
Let $\mathbf{G}$ be a $\mathbb{C}$-gluing functor, and let \( L \) be a glued-up object over \( \mathbf{G} \) through a family of morphisms \( (\pi_i^L)_{i \in \mathrm{I}} \). For each object \( (V, \delta) \in (L \downarrow \mathbb{C}) \), define the functor \(\mathbf{G}_V\) by:
\[
\mathbf{G}_V(a) := \mathbf{G}(a) \coprod_L V, \quad \text{for all } a \in \mathbb{P}_2(\mathrm{I}).
\]
and 
\[
\mathbf{G}_V(f) := \mathbf{G}(f) \coprod \operatorname{id}_V, \quad \text{for all } f \text{ morphism  in } \mathbb{P}_2(\mathrm{I}).
\]

\noindent We say that \(\mathbf{G}\) is a \textsf{universal $\mathbb{C}$-gluing functor} if \(\mathbf{G}_V\) is a $\mathbb{C}$-gluing functor for all \( (V, \delta) \in (L \downarrow \mathbb{C}) \). In this case, we say that \( L \) be a \textsf{ universal glued-up object over \( \mathbf{G} \)} through \( (\pi_i^L)_{i \in \mathrm{I}} \).
\end{definition}


\subsection{Glued up objects as equalizer}
It is well known that pullbacks can be viewed as equalizers. Similarly, glued-up objects can also be described as equalizers, as stated in the following proposition.
\begin{proposition} \label{equalizer}
Let \(\mathbb{C}\) be a category that admits products. Suppose \(\mathbf{G}\) is a \(\mathbb{C}\)-gluing functor of type \(\mathrm{I}\), and let \(L \in \mathbb{C}\) with a family of morphisms \(\subsm{(\pi_i)}{i\! \in\! \mathrm{I}}\), where \(\pi_i: L \to \mathbf{G}(i)\) is a morphism in \(\mathbb{C}\), for any $i \in \rm{I}$. Then, the following statements are equivalent:
\begin{enumerate}
    \item \(L\) is a glued-up object over \(\mathbf{G}\) through \(\subsm{(\pi_i)}{i \!\in\! \mathrm{I}} \).
    \item The following diagram is an equalizer:
    \begin{center}
        {\footnotesize 
        \begin{tikzcd}[column sep=large]
            L \arrow{rr}{\subsm{\left\langle \pi_i\right\rangle}{i\! \in\! \mathrm{I}}} & & \subsm{\prod}{i\! \in\! \mathrm{I}} \mathbf{G}(i) 
            \arrow[shift left]{rr}{\subsm{\left\langle \mathbf{G}(\iuvop{i}{j}{}) \circ p_i\right\rangle}{(i,j)\! \in\! \mathrm{I}^2}} 
            \arrow[shift right,swap]{rr}{\subsm{\left\langle \mathbf{G}(\iuvop{j}{i}{}) \circ p_j\right\rangle}{(i,j)\! \in\! \mathrm{I}^2}} 
            & & \subsm{\prod}{(i,j) \!\in\! \mathrm{I}^2} \mathbf{G}(i,j).
        \end{tikzcd}
        }
    \end{center}
    where $p_i : \subsm{\prod}{i\! \in\! \mathrm{I}} \mathbf{G}(i) \rightarrow  \mathbf{G}(i)$, for all $i \in \rm{I}$. 
\end{enumerate}
\end{proposition}
\begin{proof}
This follows directly from a generalization of \cite[Proposition 2.12]{Acts}.
\end{proof}

\subsection{Examples of gluable categories}
In this section, we present examples of gluable families and explicitly construct their corresponding glued objects in various categories. We deduce the following result from Proposition \ref{equalizer}.
\begin{lemma}
	Let $\mathbb{C}$ be a category admitting products and equalizers, then $\mathbb{C}$ is a gluable category.
\end{lemma}

\begin{remark}
	If $\mathbb{C}$ is a Grothendieck category, then $\mathbb{C}^\op$ is a gluable category. In particular, $\mathbb{Ab}^\op,  R\text{-}\mathbb{Mod}^\op$ and $ R\text{-}\mathbb{Alg}\op$ are gluable categories.
	
\end{remark}



\noindent For some concrete categories, we can not only prove that they are gluable categories but also give a standard representative for the glued -up object. The next Definition-Lemma gives a few examples of such categories along with standard representative for their glued-up objects.

\begin{deflem}\label{genera}
	Let $\mathbb{C}\in \big\{ \mathbb{Sets}, \mathbb{Grp},\mathbb{Ab},\mathbb{(o)Top}, R\text{-}\mathbb{Mod}, R\text{-}\mathbb{Alg}\;|\; R \text{ is a ring}\big\}$. Let $\mathbf{G}$ be a functor from $\mathbb{P}_2 (\mathrm{I})$ to $\mathbb{C}^{\op}$. We define the {\sf standard representative of the limit of $\mathbf{G}$} as the pair $(\subsm{Q}{\mathbf{G}}, \dindi{\iota}{Q}{\mathbf{G}}^{\op})$ where         
	\begin{itemize}
		\item $\subsm{Q}{\mathbf{G}}:={\subsm{\coprod\nolimits}{i\!\in\! \mathrm{I}} \Goi{\mbf{G}}{i}}/\widebar{\Rel{\mbf{G}}}$ where $\widebar{\Rel{\mbf{G}}}$ is the congruence relation generated by the relation $(x,i)\Rel{\mbf{G}} (y, j)$ if there exists $u\in \Goij{\mbf{G}}{i}{j}$ such that 
		$$x= \Gnij{\mbf{G}}{\iuv{i}{j}}^{\op} (u) \text{ and } y= \Gnij{\mbf{G}}{\iuv{j}{i}}^{\op}(u),$$ for any $(x,i), (y, j) \in \subsm{\coprod\nolimits}{i\!\in\! \mathrm{I}}\Goi{\mbf{G}}{i}$ where $i,j\in \mathrm{I}$.
		\item $\dindi{\iota}{Q}{\mathbf{G}}^{\op}=\subsm{\big(\dindiv{\iota}{Q}{\mbf{G}}{a}: \Goi{\mbf{G}}{a}\rightarrow {\subsm{Q}{\mathbf{G}}}\big)}{a\!\in\! \mathbb{P}_2 (\mathrm{I})}$, with ${\dindiv{\iota}{Q}{\mbf{G}}{\{i\}}}:= \pi \circ \subsm{\boldsymbol{\varepsilon}}{\Goi{\mbf{G}}{i},\subsm{\coprod\nolimits}{i\!\in\! \mathrm{I}} \Goi{\mbf{G}}{i}}$, and ${\dindiv{\iota}{Q}{\mathbf{G}}{\{i,j\}}} :={\dindiv{\iota}{Q}{\mbf{G}}{\{i\}}} \circ \Gnij{\mbf{G}}{\iuv{i}{j}}^{\op} $ where $\subsm{\boldsymbol{\varepsilon}}{\Goi{\mbf{G}}{i},\subsm{\coprod\nolimits}{j\!\in\! \mathrm{I}} \Goi{\mbf{G}}{j}}$ is the canonical map from $ \Goi{\mbf{G}}{i}$ to $\subsm{\coprod\nolimits}{j\!\in\! \mathrm{I}} \Goi{\mbf{G}}{j}$, and $\pi: \subsm{\coprod\nolimits}{i\!\in\! \mathrm{I}} \Goi{\mbf{G}}{i} \rightarrow {\subsm{Q}{\mathbf{G}}}$ is the quotient map, for all $i,j,k\in \mathrm{I}$.  
	\end{itemize} 
	Moreover, when $\mathbb{C}\in\{ \mathbb{Top}, \mathbb{(o)Top}\}$, we take $\widebar{\Rel{\mbf{G}}}$ to be the equivalence relation generated by $\Rel{\mbf{G}}$, and $\subsm{Q}{\mathbf{G}}$ is a topological space via the final topology with respect to the family $\subsm{\left( {\dindiv{\iota}{Q}{\mbf{G}}{\{i\}}}\!\right)}{i\! \in\! \mathrm{ I}}$. 
	For all $i \in \mathrm{ I}$, ${\dindiv{\iota}{Q}{\mbf{G}}{\{i\}}}\!\!: \Goi{\mbf{G}}{i}\rightarrow {\subsm{Q}{\mathbf{G}}}$ is open continuous when $\mathbb{C}=\mathbb{(o)Top}$. \\
	In all cases, $\subsm{Q}{\mathbf{G}}$ is  a glued-up object over $\mathbf{G}$ through $\dindi{\iota}{Q}{\mathbf{G}}^{\op}$.
\end{deflem}

\begin{proof}
	The fact that $ \subsm{Q}{\mathbf{G}}$ is a glued-up object over $\mbf{G}$ though $\dindi{\iota}{Q}{\mathbf{G}}^{\op}$ follows directly from the definition of $\widebar{\Rel{\mbf{G}}}$. 
	
\noindent We only prove the last assertion of the Theorem. Let $i \in \mathrm{I}$. We already know that $\dindiv{\iota}{Q}{\mbf{G}}{\{i\}}$ is a continuous map by definition of the final topology when $\mathbb{C}=\mathbb{Top}$. When $\mathbb{C}=\mathbb{(o)Top}$, we prove that $\dindiv{\iota}{Q}{\mbf{G}}{\{i\}}$ is also an open map. Consider $V\subsm{\subseteq}{\operatorname{op}} \Goi{\mbf{G}}{i}$. Our goal is to show that ${\dindiv{\iota}{Q}{\mbf{G}}{\{i\}}}\!\!(V)\subsm{\subseteq}{\operatorname{op}} \subsm{Q}{\mathbf{G}}$. In other words, we need to prove that ${\dindiv{\iota}{Q}{\mbf{G}}{\{j\}}^{-1}}\!({\dindiv{\iota}{Q}{\mbf{G}}{\{i\}}}\!\!(V))\subsm{\subseteq}{\operatorname{op}}\Goi{\mbf{G}}{j}$ for all $j\in \mathrm{I}$, considering the definition of the final topology on $\subsm{Q}{\mathbf{G}}$ with respect to $\dindi{\iota}{Q}{\mathbf{G}}$. Let $i,j \in \mathrm{I}$. To establish this, we prove that ${\dindiv{\iota}{Q}{\mbf{G}}{\{j\}}^{-1}}\!({\dindiv{\iota}{Q}{\mbf{G}}{\{i\}}}\!\!(V))= {\Gnij{\mbf{G}}{\iuv{j}{i}}^{\op}}( {\Gnij{\mbf{G}}{\iuv{i}{j}}^{\op}}^{-1} (V))$. We have
	\begin{align*}
		&y\in {\Gnij{\mbf{G}}{\iuv{j}{i}}^{\op}}( {\Gnij{\mbf{G}}{\iuv{i}{j}}^{\op}}^{-1} (V))  \\&\Leftrightarrow y=\text{$\Gnij{\mbf{G}}{\iuv{j}{i}}$}^{\op}( z)\; \text{for some $z \in {\Gnij{\mbf{G}}{\iuv{i}{j}}^{\op}}^{-1} (V)$}  
		\\& \Leftrightarrow  \pi(y,j)=\pi({\Gnij{\mbf{G}}{\iuv{i}{j}}^{\op}} (z),i)  \\& \Leftrightarrow y\in {\dindiv{\iota}{Q}{\mbf{G}}{\{j\}}^{-1}}\!({\dindiv{\iota}{Q}{\mbf{G}}{\{i\}}}\!\!(V)), \; \text{since}\; \pi({\Gnij{\mbf{G}}{\iuv{i}{j}}^{\op}} (z),i)={\dindiv{\iota}{Q}{\mbf{G}}{\{i\}}}\!\!({\Gnij{\mbf{G}}{\iuv{i}{j}}^{\op}} (z)))\; \quad \quad \text{and}\; \pi(y,j)={\dindiv{\iota}{Q}{\mbf{G}}{\{j\}}}\!\!(y).
	\end{align*}
	Since ${\Gnij{\mbf{G}}{\iuv{i}{j}}^{\op}}$ is continuous, ${\Gnij{\mbf{G}}{\iuv{i}{j}}^{\op}}^{-1} (V)$ is open.  Since ${\Gnij{\mbf{G}}{\iuv{j}{i}}^{\op}}$ is a morphism in ${\mathbb{(o)Top}}$, \\$ {\Gnij{\mbf{G}}{\iuv{j}{i}}^{\op}}( {\Gnij{\mbf{G}}{\iuv{i}{j}}^{\op}}^{-1} (V))$ is open. Thus, successfully demonstrating that for al $i\in \mathrm{I}$, $\dindi{\iota}{Q}{\{i\}}$ is a topological embedding when $\mathbf{G}$ is an $\mathbb{(o)Top}^{\op}$-gluing functor.
\end{proof}

\begin{remark}
	Using the notation as in Definition-Lemma \ref{genera}, we observe that $$\subsm{Q}{\mbf{G}}=\subsm{\cup}{i\!\in\! \mathrm{I}} \dindiv{\iota}{Q}{\mbf{G}}{\{i\}}\!(\Goi{\mbf{G}}{i}).$$
\end{remark}

\noindent We obtain the following dual version of Definition-Lemma \ref{genera}, the proof follows easily from the statement.
\begin{deflem}\label{stdrep}
	Let $\mathbb{C}\in \big\{ \mathbb{Sets}, \mathbb{Grp},\mathbb{Ab}, \mathbb{(o)Top}, R\text{-}\mathbb{Mod}, R\text{-}\mathbb{Alg}\;|\; R \text{ is a ring}\big\}$. Let $\mathbf{G}$ be a functor from $\mathbb{P}_2 (\mathrm{I})$ to $\mathbb{C}$. We define the {\sf standard representative of the limit of $\mathbf{G}$} as the pair $(\subsm{L}{\mathbf{G}}, \dindi{\kappa}{L}{\mathbf{G}})$ where   
	\begin{itemize}
		\item
		$
		\subsm{L}{\mathbf{G}} :=\begin{Bmatrix}\subsm{(\subsm{s}{i})}{i\!\in \!\rm{I}} \in \subsm{\prod\nolimits}{i\! \in\! \rm{I} } \Goi{\mbf{G}}{i} \;|\; \Gnij{\mbf{G}}{\iuv{i}{j}} ({\subsm{s}{i}})=\Gnij{\mbf{G}}{\iuv{j}{i}} ( {\subsm{s}{j}} ), \; \text{for all}\; i,j\in\rm{I}\end{Bmatrix}.$
		\item $\dindi{\kappa}{L}{\mathbf{G}}=\subsm{\big(\dindiv{\kappa}{L}{\mbf{G}}{a}: {\subsm{L}{\mathbf{G}}}\rightarrow  \Goi{\mbf{G}}{a}\big)}{a\!\in\! \mathbb{P}_2 (\mathrm{I})}$, with ${\dindiv{\kappa}{L}{\mbf{G}}{\{i\}}}$ is the morphism sending $\subsm{(\subsm{s}{i})}{i\!\in \!\rm{I}}$ to $\subsm{s}{i}$, and ${\dindiv{\kappa}{L}{\mathbf{G}}{\{i,j\}}} := \Gnij{\mbf{G}}{\iuv{i}{j}}\circ {\dindiv{\kappa}{L}{\mbf{G}}{\{i\}}}$ for all $i,j\in \mathrm{I}$.  
	\end{itemize} 
	Moreover, when $\mathbb{C}\in\{ \mathbb{Top}, \mathbb{(o)Top}\}$,  $\subsm{L}{\mathbf{G}}$ is a topological space via the initial topology with respect to the family $\subsm{( {\dindiv{\kappa}{L}{\mbf{G}}{\{i\}}}\!)}{i\! \in\! \mathrm{ I}}$. 
	For all $i \in \mathrm{ I}$, $ {\dindiv{\kappa}{L}{\mbf{G}}{\{i\}}}\!\!:  {\subsm{L}{\mathbf{G}}} \rightarrow\Goi{\mbf{G}}{i}$ is open continuous when $\mathbb{C}=\mathbb{(o)Top}$. \\ 
	In all cases, $\subsm{L}{\mathbf{G}}$ is  a glued-up object over $\mathbf{G}$ through $\dindi{\kappa}{L}{\mathbf{G}}$.	
\end{deflem}

\subsection{\(\textsf{Hom}\) functor induced by a gluing functor}  
To establish our notation, we begin by recalling the definition of the \(\operatorname{\textsf{Hom}}\) functor. While the material in this section is likely familiar to many readers, we include it here for completeness.

\begin{definition}\label{Homfunctor}  
Let \(X \in \mathbb{C}\). We define the functor \(\operatorname{\textsf{Hom}}(X,-)\) (resp. \(\operatorname{\textsf{Hom}}(-,X)\)) from \(\mathbb{C}\) to \(\mathbb{Set}\) (resp. \(\mathbb{Set}^\op\)) as follows:  

\begin{enumerate}  
\item For any object \(Y \in \mathbb{C}\), we have \(
\operatorname{\textsf{Hom}}(X,Y) \quad \text{(resp. } \operatorname{\textsf{Hom}}(Y,X) \text{).}
\)  

\item For any morphism \(f: Y \to Z\) in \(\mathbb{C}\), we have:  
\[
\begin{aligned}
\operatorname{\textsf{Hom}}(X,-)(f):  \operatorname{\textsf{Hom}}(X,Y) &\to \operatorname{\textsf{Hom}}(X,Z), \\
 \quad g& \mapsto f \circ g.
\end{aligned}
\]
\[
\left( \text{resp. }  
\begin{aligned}
\operatorname{\textsf{Hom}}(-,X)(f):  \operatorname{\textsf{Hom}}(Z,X) &\to \operatorname{\textsf{Hom}}(Y,X), \\  \quad g & \mapsto g \circ f.
\end{aligned}
\right)
\]
\end{enumerate}  

\end{definition}  
\noindent We also introduce two natural transformations that will be used throughout the document.  
\begin{definition}
Let \(\phi: X \to Y\) be a morphism in \(\mathbb{C}\). We define the following natural transformations:  

\begin{itemize}  
    \item We denote by \(\phi^\ast\) the natural transformation  
    \[
    \phi^\ast: \textsf{Hom}(Y, -) \to \textsf{Hom}(X, -)
    \]  
    given by  
    \[
    (\phi^\ast)_Z: \textsf{Hom}(Y, Z) \to \textsf{Hom}(X, Z), \quad \psi \mapsto \psi \circ \phi.
    \]  

    \item We denote by \(\phi_\ast\) the natural transformation  
    \[
    \phi_\ast: \textsf{Hom}(-, X) \to \textsf{Hom}(-, Y)
    \]  
    given by  
    \[
    (\phi_\ast)_Z: \textsf{Hom}(Z, X) \to \textsf{Hom}(Z, Y), \quad \psi \mapsto \phi \circ \psi.
    \]  
\end{itemize}  
\end{definition}
\begin{remark}  
For any object \(X \in \mathbb{C}\) and any morphism \(f: Y \to Z\), we have:  
\[
\operatorname{\textsf{Hom}}(X,-)(f) = (f_*)_X, \quad \text{and} \quad \operatorname{\textsf{Hom}}(-,X)(f) = (f^*)_X.
\]  
\end{remark}

\noindent Next, we introduce the \(\textsf{Hom}\) functor categories.  

\begin{definition}  
The \emph{left Hom functor category} \(\mathbb{LHom}_{\mathbb{C}}\) (resp. the \emph{right Hom functor category} \(\mathbb{RHom}_{\mathbb{C}}\)) of a category \(\mathbb{C}\) is the full subcategory of  \({\mathbb{C}^\op}^{\mathbb{Set}}\) (resp. \({\mathbb{C}}^{\mathbb{Set}}\)) whose objects are functors of the form  
\(\textsf{Hom}(-, Z) \:\: \text{(resp. } \textsf{Hom}(Z,-) \text{)}
\) for some \(Z \in \mathbb{C}\).
\end{definition}  

\begin{remark}  
We note that \(\mathbb{LHom}_{\mathbb{C}}^\op = \mathbb{RHom}_{\mathbb{C^\op}}\).
\end{remark}  

\begin{lemma}\label{nattr}  
Given a natural transformation \( \theta: \emph{\textsf{Hom}}(Y, -) \to \emph{\textsf{Hom}}(Z, -)\) 
(resp. \(\theta: \emph{\textsf{Hom}}(-, Y) \to \emph{\textsf{Hom}}(-, Z)\)),  
we have the identity  \(\theta = \theta_Z(\operatorname{id}_Z)^\ast\) (resp. \(\theta = \theta_Z(\operatorname{id}_Z)_\ast\)).  
\end{lemma}  

\begin{proof}  
Let \(\theta\) be a natural transformation from \(\textsf{Hom}(Y, -)\) to \(\textsf{Hom}(Z, -)\), and let \(f: Z \to X\) be a morphism in \(\mathbb{C}\). Consider the following commutative diagram:  
\[
\xymatrix{  
\textsf{Hom}(Z,Z) \ar[d]_{\theta_{Z}} \ar[r]^{(f_\ast)_Z} &  \textsf{Hom}(Z,X) \ar[d]^{\theta_X} \\  
\textsf{Hom}(Y,Z) \ar[r]_{(f_\ast)_Y}  &  \textsf{Hom}(Y,X)  
}
\]  
Evaluating this diagram at \(\operatorname{id}_Z\), we obtain  
\[
\theta_X(f) = f \circ \theta_Z(\operatorname{id}_Z) = \theta_Z(\operatorname{id}_Z)^\ast_X (f).
\]  
Thus, we conclude that \(\theta = \theta_Z(\operatorname{id}_Z)^\ast\). A similar argument proves the dual case.  
\end{proof}  

Any \(\mathbb{C}\)-gluing functor naturally induces a \(\textsf{Hom}\) functor, as described below.

\begin{definition}\label{homsets}  
Let \(\mathbf{G} \in \mathbb{P}_2(\mathrm{I})^{\mathbb{C}^\op}\). Then, \(\mathbf{G}\) induces a functor, denoted \(\textsf{Hom}(\mathbf{G})\), in \(  \mathbb{P}_2(\mathrm{I})^{\mathbb{RHom}_{\mathbb{C}}}\), which is defined as follows:  
\begin{enumerate}  
\item For every \(a \in \mathbb{P}_2(\mathrm{I})\), \(\textsf{Hom}(\mathbf{G})(a) = \textsf{Hom}(\mathbf{G}(a),-) \).  
\item For all \(i,j \in \mathrm{I}\), \(\textsf{Hom}(\mathbf{G})(\iuv{i}{j}) = {{ \mathbf{G}(\iuvop{i}{j}{})^\op}^\ast }\).
\end{enumerate}  
\end{definition}  

\noindent The following result characterizes glued objects in terms of the \(\textsf{Hom}\) functor.
\begin{proposition}\label{HomG}  
Let \(\mathbf{G} \in \mathbb{P}_2(\mathrm{I})^{\mathbb{C}^\op}\).  
Suppose there exists an object \(L \in \mathbb{C}\) and a family of morphisms  
\(\iota_a: \mathbf{G}(a) \to L\) in \(\mathbb{C}\) for all \(a \in \mathbb{P}_2(\mathrm{I})\).  
The following conditions are equivalent:  
\begin{enumerate}  
    \item \(\mathbf{G}\) is a \(\mathbb{C}^\op\)-gluing functor with limit  \(  \left(L, \subsm{\left(\iota_a^\op \right)}{a\!\in\! \mathbb{P}_2(\mathrm{I})}\right).\) 
    \item The functor \(\emph{\textsf{Hom}}(\mathbf{G})\) is a \(\mathbb{RHom}_{\mathbb{C}}\)-gluing functor whose limit is \( \left(\emph{\textsf{Hom}}(L,-), \subsm{\left(\iota_a^\ast \right)}{a\!\in\! \mathbb{P}_2(\mathrm{I})}\right). \)
\end{enumerate}  
In particular, when \(\mathbf{G}\) is a \(\mathbb{C}\)-gluing functor, we have the canonical isomorphism:  
\[
\lim \emph{\textsf{Hom}}(\mathbf{G}) \simeq \emph{\textsf{Hom}}(\lim \mathbf{G}).
\]  
\end{proposition}  

\begin{proof}  
Let \(L' \in \mathbb{C}\). Clearly, \(L'\) is a cone over \(\mathbf{G}\) if and only if  
\(\textsf{Hom}(L',-)\) is a cone over \(\textsf{Hom}(\mathbf{G})\).  

\noindent Applying Lemma~\ref{nattr}, we deduce that \(L\) is a terminal cone over \(\mathbf{G}\)  
if and only if \(\textsf{Hom}(L,-)\) is a terminal cone over \(\textsf{Hom}(\mathbf{G})\).  
\end{proof}

\section{Operations between gluing functors}
\subsection{Refinements}
Inspired by the concept of refinement in coverings, as discussed in \cite[\S 1]{Conrad}, we introduce a corresponding notion of refinement between gluing functors of different types. This concept is induced by a mapping between their index sets. Drawing from the theory of ringed topological spaces, we explore the idea of refinement morphisms between two such gluing functors.

\begin{definition}\label{refino}
Let $\mathrm{I}$ and $\mathrm{J}$ be two index sets, and let $\gamma : \mathrm{I} \rightarrow \mathrm{J}$ be a map. We define $\mathbf{P_2}(\gamma): \mathbb{P}_2(\mathrm{I}) \rightarrow \mathbb{P}_2(\mathrm{J})$ as the functor such that, for all $i, j\in \mathrm{I}$, we have:
\begin{itemize}
    \item $\mathbf{P_2}(\gamma)(\{i\}) = \{\gamma(i)\}$;
    \item $\mathbf{P_2}(\gamma)(\{i, j\}) = \{\gamma(i), \gamma(j)\}$;
    \item $\mathbf{P_2}(\gamma)(\iuv{i}{j}) = \iuv{\gamma(i)}{\gamma(j)}$;
\end{itemize}
\end{definition}

\begin{remark}
\begin{enumerate}
\item When $\gamma$ is a bijection, then $\mathbf{P_2}(\gamma)$ and $\mathbf{P_2}(\gamma^{-1})$ are inverses of each other;
\item When $\gamma$ is injective (resp. surjective), then $\gamma$ has left (resp. right) inverse $\beta$ and  $\mathbf{P_2}(\beta)$ is a left inverse of  $\mathbf{P_2}(\gamma)$.
\end{enumerate}
\end{remark}

\begin{definition}
Let $\mathbb{C}$ be a category, and let $\mathrm{I}$ and $\mathrm{J}$ be sets. We say that a $\mathbb{C}$-gluing functor $\mathbf{G}$ of type $\mathrm{J}$ {\sf refines} a $\mathbb{C}$-gluing functor $\mathbf{F}$ of type $\mathrm{I}$ if there exists a map $\gamma: \mathrm{I} \rightarrow \mathrm{J}$ and a natural transformation from $\mathbf{G}\circ \mathbf{P_2}(\gamma)$ to $\mathbf{F}$. In this case, we also say that $\mathbf{G}$ is {\sf a refinement of $\mathbf{F}$}, and the natural transformation from $\mathbf{G}\circ \mathbf{P_2}(\gamma)$ to $\mathbf{F}$ is called {\sf a refinement morphism from $\mathbf{G}$ to $\mathbf{F}$}. 
\end{definition}

\noindent In the next lemma, we observe that the refinement of gluing functions induces a unique canonical morphism between their respective glued-up objects when they exist.
\begin{deflem}  \label{refine} 
Let $\mathbb{C}$ be a category, $\mathbf{G}$ be a $\mathbb{C}$-gluing functor of type $\mathrm{J}$, $\mathbf{F}$ be a $\mathbb{C}$-gluing functor of type $\mathrm{I}$, and $\rho$ be a refinement from $\mathbf{G}$ to $\mathbf{F}$. 
We define $\operatorname{lim} \rho$ to be the unique cone morphism from $ \limi\;{\mathbf{G}}$ to $ \limi\;{\mathbf{F}}$ induced by $\rho$.
\end{deflem}

\begin{proof}
We write 
\begin{itemize}
\item $ \limi\;{\mathbf{G}}$ as $(\subsm{L}{\mathbf{G}}, \subsm{\pi}{\mathbf{G}})$;
\item $ \limi\;{\mathbf{F}}$ as $(\subsm{L}{\mathbf{F}}, \subsm{\pi}{\mathbf{F}})$.
\end{itemize}
We denote $\gamma : \mathrm{I} \rightarrow \mathrm{J}$ to be the map inducing the refinement morphism and $\rho$ to be the refinement morphism from $\mathbf{G}$ to $\mathbf{F}$.
Since $\limi{\ \mathbf{F}}$ is a terminal cone over $\mathbf{F}$, we obtain the existence of the unique map $\mu: \subsm{L}{\mathbf{G}}\rightarrow \subsm{L}{\mathbf{F}}$ such that the following commutative diagram commutes for all distinct indices $i,j\in \mathrm{I}$:

\begin{figure}[H]
	\begin{center}
		\adjustbox{scale=0.8,center}{
			\begin{tikzcd}
				 \scalebox{0.7}{${\mathbf{F}}(i,j) $}                                                                                                                                                                        &  &   \scalebox{0.7}{${\mathbf{F}}(j)$} \arrow[swap, labels=description]{ll} { \scalebox{0.7}{${\mathbf{F}}( \iuv{j}{i})$}} &                                                                                                   \\
			  &  \scalebox{0.7}{${\mathbf{F}}(i)$} \arrow[swap, labels=description]{lu} {\scalebox{0.7}{${\mathbf{F}}( \iuv{i}{j})$}} &             &  \scalebox{0.7}{$\subsm{L}{\mathbf{F}}$}  \arrow[swap, labels=description]{lu}{\scalebox{0.7}{$\dindsi{\pi}{\mathbf{F}}{\{j\}}$}}  \arrow[near end, labels=description]{ll}{\scalebox{0.7}{$\dindsi{\pi}{\mathbf{F}}{\{i\}}$}}\\
				 \scalebox{0.7}{${\mathbf{G}}(\gamma(i), \gamma(j))$}  \arrow[labels=description]{uu}{\scalebox{0.7}{$ \subsm{\rho}{\{i,j\}}$}}                                                                                                                                                                                                                                                  &  &   \scalebox{0.7}{$ {\mathbf{G}}(\gamma(j))$} \arrow[swap, labels=description, near end]{ll} {\scalebox{0.7}{${\mathbf{G}}(\iuv{\gamma (j)}{\gamma (i)})$}} \arrow[near end, labels=description]  {uu}{\scalebox{0.7}{$\subsm{\rho}{\{j\}}$}}   &    \\	  &                                                                                                                                                                         \scalebox{0.7}{${\mathbf{G}}(\gamma(i))$} \arrow[labels=description]{lu}{\scalebox{0.7}{${\mathbf{G}}(\iuv{\gamma (i)}{\gamma (j)})$}}      \arrow[near start, labels=description]{uu}{\scalebox{0.7}{$\subsm{\rho}{\{i\}}$}}                                                                                                                                                         &  & \scalebox{0.7}{$ \subsm{L}{\mathbf{G}}$}\arrow[labels=description]{ll} {\scalebox{0.7}{$\dindsi{\pi}{\mathbf{G}}{\{i\}}$}}  \arrow[swap, labels=description]{lu}{\scalebox{0.7}{$\dindsi{\pi}{\mathbf{G}}{\{j\}}$}}   \arrow[dotted, labels=description]{uu}{\scalebox{0.7}{$\exists ! \mu$}}                                                                           
			\end{tikzcd}}
		\end{center}
\end{figure}	
\end{proof}

\begin{remark}

Let $\mathbf{F}$ be a $\mathbb{C}$-gluing functor. We observe that the operation $\mathbf{F}\circ \mathbf{P_2}(\gamma)$ can be interpreted as the restriction of $\mathbf{F}$ to $\mathrm{I}$ via the map $\gamma$. Specifically, when $\gamma$ is an inclusion map from $\mathrm{I}$ to $\mathrm{J}$, it effectively acts as a restriction map. More precisely, an inclusion naturally defines a refinement. We can establish a natural transformation:
$$\alpha: \mathbf{F}\circ \mathbf{P_2}(\gamma) \to \mathbf{F},$$
where $\alpha_a = \operatorname{id}_{\mathbf{F}(a)}$ for every $a \in  \mathbb{P}_2(\mathrm{I})$. According to Lemma \ref{refine}, this induces a canonical morphism: 
$$\mu: \lim\mathbf{F}\circ \mathbf{P_2}(\gamma) \to \lim \mathbf{F}.$$
In Example \ref{pushut}, this phenomenon is illustrated through figures, showing the inclusion sequence $\{i\} \subseteq \{i,j\} \subseteq \{i,j,k\} \subseteq \{i,j,k,l\}$.
 \end{remark}

We can define the category of gluing functors, where we define morphisms to be refinement morphisms.

\begin{definition}\label{gdff}
Let $\mathbb{C}$ be a category. We define the category of $\mathbb{C}$-{\sf gluing functors}, denoted as $\mathbf{Gf}(\mathbb{C})$, to be the category whose objects are $\mathbb{C}$-gluing functors, and morphisms are refinement morphisms.
\end{definition}

\subsection{Formally composing gluing}

Building on the concepts defined in previous sections, we now introduce the notion of composing gluing functors within their respective category. This structured approach allows for the systematic combination of multiple gluing functors.

\begin{definition}\label{refinery}
Let $\mathrm{I}$ be a set and $\subsm{( {\subsm{\mathbf{G}}{i}})}{i\! \in\! \mathrm{I}}$ be a family gluing functors. We say the family $\subsm{( {\subsm{\mathbf{G}}{i}})}{i\! \in\! \mathrm{I}}$ is {\sf composable}, if there exists a $\mathbf{Gf}(\mathbb{C})$-gluing functor $\mathbf{G}$ of type $\mathrm{I}$ such that ${\mathbf{G}}(\{ i\})={\subsm{\mathbf{G}}{i}}$, for all $i\in \mathrm{I}$. We call ${\lim} \mathbf{G}$ a composite of $\subsm{({\subsm{\mathbf{G}}{i}})}{i\! \in\! \mathrm{I}}$ and if $(L, \pi)\simeq {\lim} {\lim} \mathbf{G}$, we say that $L$ is a {\sf composite glued up object} over $\mathbf{G}$. \end{definition}

\noindent We illustrate the composition of gluing functors with a simple example.
\begin{example} \label{torusglu}
We describe how to glue a cylinder into a torus, as shown in Figure~\ref{fig7}, through the composition of two gluing functors.

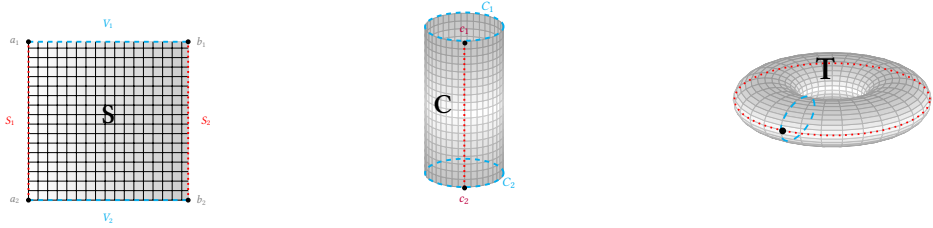
\begin{figure}[H]\begin{center}
	
	\scalebox{0.6}{\begin{tabular}{ccc}
			
			\begin{tikzpicture}[scale=0.7]
				\tikzset{roundnode/.style = {thick, draw = black, fill = black, outer sep = 1, circle, minimum size = 1pt, scale = 0.2}}
				
				\fill[black, left color=white, right color=gray!40] (0,0) rectangle (5,5);
				\draw[step=0.3cm,black,very thin] (0,0) grid (5,5);
				\draw[step=0.3cm,red, dotted, very thick] (0,0) -- (0,5); 
				\draw[step=0.3cm,red, dotted, very thick] (5,0) -- (5,5);
				\draw[step=0.3cm,cyan, dashed, very thick] (0,0) -- (5,0);
				\draw[step=0.3cm,cyan, dashed, very thick] (0,5) -- (5,5);
				\node[roundnode = gray, label = left: {\scalebox{0.6}{$ \color{gray} a_2$}}] (s) at (0,0){};
				\node[roundnode = gray, label = left: {\scalebox{0.6}{$\color{gray} a_1$}}] (s) at (0,5){};
				\node[roundnode = gray, label = right: {\scalebox{0.6}{$ \color{gray} b_1$}}] (s) at (5,5){};
				\node[roundnode = gray, label = right: {\scalebox{0.6}{$ \color{gray} b_2$}}] (s) at (5,0){};
				\node[label = left: {{\scalebox{0.6}{$\color{red} S_1$}}}] (s) at (0,2.5){};
				\node[label = below: {\scalebox{0.6}{$ \color{cyan} V_2$}}] (s) at (2.5,0){};
				\node[label = above: {\scalebox{0.6}{$\color{cyan} V_1$}}] (s) at (2.5,5){};
				\node[label = right: {\scalebox{0.6}{$ \color{red} S_2$}}] (s) at (5,2.5){};
				\node[label = {{\Large $\mathbf{S}$}}] (s) at (2.5,2){};
			\end{tikzpicture}
			
			&
			\begin{tikzpicture}[font=\large]
				\tikzset{
					node distance=0cm,
					buffer/.style={
						shape border rotate=270,
						regular polygon,
						regular polygon sides=3,
						minimum height=2cm,
						shade,shading=axis,left color=gray!40,right color=gray!80,
					}
				}
				
				\scalebox{0.65}{  \node (therectangle) at (0,0) [shade,shading=axis,left color=white,right color=gray!40,minimum width=2.5cm,minimum height=1cm,outer sep=0pt] {\small \color{red}{Glue $S_1$ with $S_2$}};
					
					\node [buffer,outer sep=0pt,right=of therectangle] {};
					
					\node (therectangle) at (0,-1.8) [fill=white,minimum width=.25cm,minimum height=1cm,outer sep=0pt] {};}
			\end{tikzpicture}%
			\begin{tikzpicture}[scale=1]
				\tikzset{roundnode/.style = {thick, draw = black, fill = black, outer sep = 1, circle, minimum size = 1pt, scale = 0.2}}
				\begin{axis}[axis lines=none,
					domain=0:2*pi, y domain=0:40,
					xmin=-3, xmax=3, ymin=-3, ymax=3, zmin
					=0.0, zmax=40,
					samples=40,samples y=20,
					z buffer=sort]
					\addplot3[surf] ({cos(deg(x))},{sin(deg(x))},{y});
					\addplot3[red, dotted,very thick, domain=-3.9:36.1,samples y=0] (0,0,x);
					\begin{scope}[yshift=101]
						\draw[cyan, very thick,dashed] (0,0) ellipse [x radius=.86cm,
						y radius=.3125cm];
					\end{scope}
					
					\begin{scope}[yshift=9]
						\draw[cyan,very thick, dashed] (0,0) ellipse [x radius=.86cm,
						y radius=.3125cm];
						\node (therectangle) at (2,-1.9) [fill=white,minimum width=.5cm,minimum height=1cm,outer sep=0pt] {};
						\node[roundnode = gray, label =below: {\color{purple}{{\tiny$c_2$}}}] (s) at (0.46,-0.97){};
						\node[roundnode = gray, label =above: {\color{purple}{{\tiny$c_1$}}}] (s) at (-3.9,8.4){};
						\node[label =below: {\color{cyan}{{\tiny$C_2$}}}] (s) at (0.8,1){};
						\node[label =above: {\color{cyan}{{\tiny$C_1$}}}] (s) at (-3.9,9.8){};
						\node[label =above: {{{\Large$\mathbf{C}$}}}] (s) at (-1.9,2.8){};
					\end{scope}

				\end{axis}
			\end{tikzpicture}
			
			&
			\begin{tikzpicture}[font=\large]
				\tikzset{
					node distance=0cm,
					buffer/.style={
						shape border rotate=270,
						regular polygon,
						regular polygon sides=3,
						minimum height=2cm,
						shade,shading=axis,left color=gray!40,right color=gray!80,
					}
				}
				
				\scalebox{0.65}{  \node (therectangle) at (0,0) [shade,shading=axis,left color=white,right color=gray!40,minimum width=2.5cm,minimum height=1cm,outer sep=0pt] {\small \color{cyan}{Glue $C_1$ with $C_2$}};
					
					\node [buffer,outer sep=0pt,right=of therectangle] {};
					
					\node (therectangle) at (0,-1.8) [fill=white,minimum width=.25cm,minimum height=1cm,outer sep=0pt] {};}
			\end{tikzpicture}%

			\begin{tikzpicture}[scale=1.4]
				\tikzset{roundnode/.style = {thick, draw = black, fill = black, outer sep = 1, circle, minimum size = 1pt, scale = 0.2}}
				\begin{axis}[axis lines=none,
					xmin=-5, xmax=5, ymin=-5, ymax=5, zmin=-5, zmax=5,
					]
					\addplot3[surf,
					samples=30,
					domain=0:2*pi,y domain=0:2*pi,
					z buffer=sort]
					({(2+cos(deg(x)))*cos(deg(y+pi/2))},
					{(2+cos(deg(x)))*sin(deg(y+pi/2))},
					{sin(deg(x))});
					
					\begin{scope}[yshift=46, xshift=0]
						\draw[red,thick, dotted] (0,0) ellipse [x radius=1.53cm,
						y radius=.57cm];
					\end{scope}
					
					\begin{scope}[canvas is yz plane at x=-1, yshift=-28.2, xshift=-9.8]
						\draw[cyan,thick, dashed] (0,0) ellipse [x radius=1.22cm,
						y radius=.885cm];
						\node (therectangle) at (2,-4.9) [fill=white,minimum width=.5cm,minimum height=1cm,outer sep=0pt] {};
						\node[roundnode = gray] (s) at (-1.07,-0.04){};
						\node[ label =above : {\large {{$\mathbf{T}$}}}] (s) at (2,0){};
					\end{scope}
				\end{axis}
			\end{tikzpicture}

	\end{tabular} }
\end{center}\caption{Gluing a square into a torus}\label{fig7}           \end{figure}

\noindent We use the notation of Figure~\ref{fig7}, along with the following:
\begin{itemize}
  \item $\phi : S_1 \to S_2$ is an isomorphism;
  \item $\psi : V_1 \to V_2$ is an isomorphism;
  \item $\lim \psi : C_1 \to C_2$ is the map induced by $\psi$ after gluing;
  \item $f_{a,b} : \{ a \} \rightarrow \{ b \}$ sends $a$ to $b$;
  \item ${A^\amalg}^2$ denotes the coproduct $A \amalg A$.
\end{itemize}

\noindent The cylinder $C$ is obtained by gluing $S$ along $S_1$ and $S_2$, making it the coequalizer of the maps $\mathfrak{i}_{S_1, S}$ and $\mathfrak{i}_{S_2, S} \circ \phi$. This can also be described as the pushout of these maps via a gluing functor $\mathbf{G}$ defined on the category $\mathbb{P}_2(\mathrm{I})$ for $\mathrm{I} = \{1, 2\}$.

\noindent The data of the functor $\mathbf{G}: \mathbb{P}_2(\mathrm{I}) \rightarrow \mathbb{Top}^\op$ is:
\[
\begin{aligned}
&\Goi{\mbf{G}}{1} = S, \quad \Goi{\mbf{G}}{2} = S_1, \quad \Goij{\mbf{G}}{1}{2} = {S_1^\amalg}^2, \\
&\Gnij{\mbf{G}}{\iuv{1}{2}}^{\op} = (\mathfrak{i}_{S_1, S}, \mathfrak{i}_{S_2, S} \circ \phi), \quad \Gnij{\mbf{G}}{\iuv{2}{1}}^{\op} = (\operatorname{id}_{S_1}, \operatorname{id}_{S_1}).
\end{aligned}
\]

\noindent Thus, $C$ is the glued-up object associated with $\mathbf{G}$.

\noindent We now refine this gluing to describe the next step: going from the cylinder to the torus. For each $i \in \{1,2\}$, define the refinement gluing functor $\subsm{\mathbf{G}}{V_i} : \mathbb{P}_2(\mathrm{I}) \rightarrow \mathbb{Top}^\op$ by:

\[
\begin{aligned}
&\Goi{\subsm{\mathbf{G}}{V_i}}{1} = S \amalg V_i, \quad \Goi{\subsm{\mathbf{G}}{V_i}}{2} = S_1 \amalg \{a_i\}, \quad \Goij{\subsm{\mathbf{G}}{V_i}}{1}{2} = {(S_1 \amalg \{a_i\})^\amalg}^2, \\
&\Gnij{\subsm{\mathbf{G}}{V_i}}{\iuv{1}{2}}^{\op} = \left( \mathfrak{i}_{S_1 \amalg \{a_i\}, S \amalg V_i}, \, \mathfrak{i}_{S_2 \amalg \{b_i\}, S \amalg V_i} \circ (\phi \amalg f_{a_i, b_i}) \right), \\
&\Gnij{\subsm{\mathbf{G}}{V_i}}{\iuv{2}{1}}^{\op} = (\operatorname{id}_{S_1 \amalg \{a_i\}}, \operatorname{id}_{S_1 \amalg \{a_i\}}).
\end{aligned}
\]

\noindent Hence, $C \amalg C_i$ arises as a glued-up object over $\subsm{\mathbf{G}}{V_i}$.

\noindent Finally, the corresponding refinement morphism $\subsm{\rho}{i}: \subsm{\mathbf{G}}{V_i} \rightarrow \mathbf{G}$ is given by:
\[
\begin{aligned}
&\subsm{\subsm{\rho}{i}}{\{1\}} = (\operatorname{id}_S, \mathfrak{i}_{V_i, S}), \quad
\subsm{\subsm{\rho}{i}}{\{2\}} = (\operatorname{id}_{S_1}, \mathfrak{i}_{\{a_i\}, S_1}), \\
&\subsm{\subsm{\rho}{i}}{\{1,2\}} = (\operatorname{id}_{S_1} \amalg \mathfrak{i}_{\{a_i\}, S_1}, \operatorname{id}_{S_1} \amalg \mathfrak{i}_{\{a_i\}, S_1}).
\end{aligned}
\]

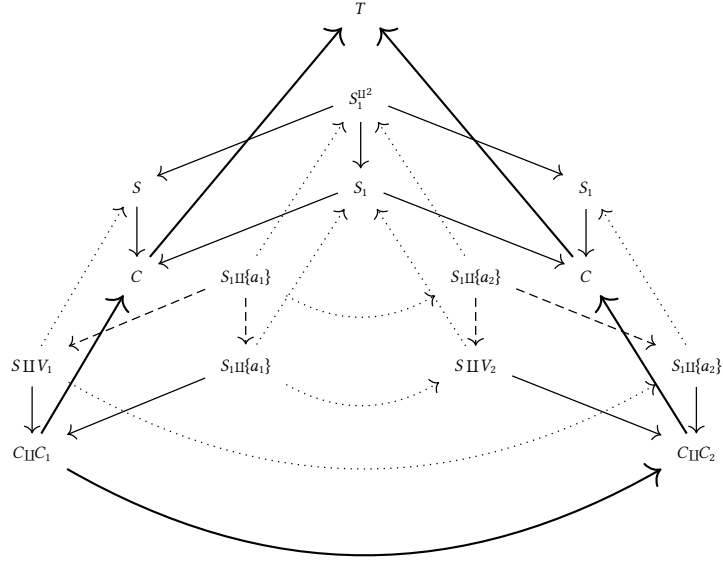
\begin{figure}[H]
\begin{center}
	\begin{tikzcd}[scale cd=1.1, column sep=1.5em, row sep=1.5em]
		&                             &                                                                                            & \scalebox{0.5}{$T$}                                   &                                                              &                              &              \\
		&                             &                                                                                            & \scalebox{0.5}{${S_1^\amalg}^2$}  \arrow[swap]{lld}{} \arrow[near end]{d}{} \arrow[]{rrd}{} &                                                              &                              &              \\
		&  \scalebox{0.5}{$S$}\arrow{d}{}     &                                                                                            & \scalebox{0.5}{$S_1$}  \arrow[swap]{lld}{} \arrow{rrd}{}           &                                                              &  \scalebox{0.5}{$S_1$}  \arrow{d}{}     &              \\
		&  \scalebox{0.5}{$C$}  \arrow[near end,thick]{rruuu}{} & \scalebox{0.5}{$S_1 \scalebox{0.8}{$\amalg$} \{a_1\}$} \arrow[ dotted]{ruu}{} \arrow[dashed,swap]{d}{} \arrow[dashed]{lld}{} \arrow[dotted, bend right]{rr}{} &                                     & \scalebox{0.5}{$S_1 \scalebox{0.8}{$\amalg$} \{a_2\}$} \arrow[swap,dotted]{luu}{} \arrow[ dashed]{d}{} \arrow[dashed, swap]{rrd}{}&  \scalebox{0.5}{$C$}  \arrow[thick,near end,swap]{lluuu}{} &              \\
		\scalebox{0.5}{$S \amalg V_1$}  \arrow[swap]{d}{}  \arrow[dotted]{ruu}{} \arrow[dotted, bend right]{rrrrrr}{}    &                             &  \scalebox{0.5}{$S_1 \scalebox{0.8}{$\amalg$} \{a_1\}$} \arrow[near start]{lld}{} \arrow[dotted]{ruu} \arrow[dotted,bend right]{rr}{}                                           &                                     &  \scalebox{0.5}{$S \amalg V_2 $}  \arrow[dotted,near end]{luu}{}  \arrow[swap,near start]{rrd}{}                                    &                              & \scalebox{0.5}{$S_1 \scalebox{0.8}{$\amalg$} \{a_2\}$} \arrow{d}{} \arrow[dotted,swap]{luu}{} \\
		\scalebox{0.5}{$C\scalebox{0.6}{$\coprod$} C_1$} \arrow[thick,swap]{ruu}\arrow[thick,bend right]{rrrrrr}{} &                             &                                                                                            &                                     &                                                              &                              & \scalebox{0.5}{$C\scalebox{0.6}{$\coprod$} C_2$}      \arrow[thick]{luu}   
	\end{tikzcd}
\end{center}
\caption{Diagram representation of $\mathcal{G}$ and its limits, where the top pushout squares each describe $\mathbf{G}$, the bottom-left pushout square represents $\subsm{\mathbf{G}}{V_1}$, the bottom-right pushout square represents $\subsm{\mathbf{G}}{V_2}$, the pointed arrows indicate the refinement maps, the bold front diagram describes $\lim \mathcal{G}$ and the other arrows are as described above.\\ }
\end{figure}
\end{example}

\section{Grothendieck topologies}\label{grothe}
\noindent The purpose of this section is to review the well-known notions related to Grothendieck topologies, with a particular emphasis on highlighting their intrinsic connection to the idea of gluing. 
Some of the definitions, notations, and results related to Grothendieck topologies presented here can also be found in \cite[\S1]{artin}. In this section, we assume that $\mathbb{C}$ admits pullbacks.

\subsection{Canonical functor associated with a sink}
\noindent We start this section with the definition of a sink in a category.
\begin{definition} We define {\sf a sink to $U$ in $\mathbb{C}$} to be a pair $(U, \subsm{(U_i,\iota_i)}{i\!\in \! \mathrm{I}})$ where $U\in \mathbb{C}$, and for each $i \in \mathrm{I}$, $U_i$ is an object in $\mathbb{C}$ and $\iota_i$ is a morphism in $\mathbb{C}$ from $ U_i$ to $U$. We call  {\sf a sink $(U, \subsm{(U_i,\iota_i)}{i\!\in \! \mathrm{I}})$ in $\mathbb{C}$} to be a sink to $U$ in $\mathbb{C}$, for some object $U$ in $\mathbb{C}$. We denote by $\subsm{\mathbf{Sink}}{}(\mathbb{C})$ (resp. $\subsm{\mathbf{Sink}}{U}(\mathbb{C})$) the set of sinks in $\mathbb{C}$ (resp. the set of sinks to $U$ in $\mathbb{C}$, where $U \in \mathbb{C}$). 
\end{definition}

\noindent To any sink in a category \(\mathbb{C}\), we associate a functor as follows.

\begin{definition} \label{gluingcover}
Let \(\EuScript{U} = \big(U, (U_i, \iota_i)_{i \in \mathrm{I}}\big)\) be a sink in \(\mathbb{C}\).

\begin{enumerate}
    \item The \textsf{canonical functor associated with \(\EuScript{U}\)}, denoted \(\gcov{\EuScript{U}} : \mathbb{S}_2(\mathrm{I}) \to \mathbb{C}^{\op}\), is defined as follows:
    \begin{itemize}
        \item On objects:
        \begin{align*}
            \gcov{\EuScript{U}}(i) &:= U_i, \\
            \gcov{\EuScript{U}}(i,j) &:= U_i \times_U U_j.
        \end{align*}
        \item On morphisms:
        \begin{align*}
            \gcov{\EuScript{U}}(\iuv{i}{j}) &:= {\proj{1}{U_i \subsm{\times}{\scalebox{0.7}{$U$}} U_j}}^\op, \\
            \gcov{\EuScript{U}}(\tau_{i,j}) &:= \varphi_{i,j}^\op,
        \end{align*}
        where \(\varphi_{i,j} : U_i \times_U U_j \to U_j \times_U U_i\) is the canonical isomorphism between the pullbacks.
    \end{itemize}

    \item Let \(f : V \to U\) be a morphism in \(\mathbb{C}\). The \textsf{base change of \(\gcov{\EuScript{U}}\) along \(f\)}, denoted \(\gcov{\EuScript{U}_V} : \mathbb{S}_2(\mathrm{I}) \to \mathbb{C}^\op\), is defined by:
    \begin{itemize}
        \item On objects: \(\gcov{\EuScript{U}_V}(a) := \gcov{\EuScript{U}}(a) \times_U V\), for all \(a \in \mathbb{S}_2(\mathrm{I})\),
        \item On morphisms: \(\gcov{\EuScript{U}_V}(g) := \gcov{\EuScript{U}}(g) \times \mathrm{id}_V\), for all \(g : a \to a'\) in \(\mathbb{S}_2(\mathrm{I})\).
    \end{itemize}
\end{enumerate}
\end{definition}

\noindent It is well known that in any category admitting pullbacks, regular epimorphisms are strict epimorphisms (see \cite[Section~1.4]{Low2016}). The following result generalizes this observation.

\begin{lemma} \label{gluthru}
Let \(\mathbb{C}\) be a category with pullbacks, and let \(\EuScript{U} = \big(U, (U_i, \iota_i)_{i \in \mathrm{I}}\big)\) be a sink in \(\mathbb{C}\). Suppose that \(U\) is a (universal) glued-up object through the family of morphisms \((\iota_i)_{i \in \mathrm{I}}\), over a split \(\mathbb{C}^\op\)-gluing functor \(\mathbf{G}\) of type \(\mathrm{I}\) satisfying \(\mathbf{G}(i) = U_i\) for all \(i \in \mathrm{I}\). Then \(U\) is also a (universal) glued-up \(\mathbb{C}^\op\)-object over the canonical functor \(\gcov{\EuScript{U}}\), through the same family \((\iota_i)_{i \in \mathrm{I}}\).
\end{lemma}

\begin{proof}
We prove that \((U, \subsm{(U_i, \iota_i)}{i\! \in\! \mathrm{I}})\) is a terminal cone over \(\gcov{\EuScript{U}}\) through \(\subsm{(\iota_i)}{i\! \in\! \mathrm{I}}\). Clearly, \((U, \subsm{(U_i, \iota_i)}{i\! \in\! \mathrm{I}})\) is a cone over \(\gcov{\EuScript{U}}\) by the definition of \(U_i \times_U U_j\).

\noindent Let \((D, \subsm{(\dindsi{\iota}{D}{i})}{i\! \in\! \mathrm{I}})\) be a cone over \(\gcov{\EuScript{U}}\). By the universal property of pullbacks, for all \(i, j \in \mathrm{I}\), there exists a unique morphism \(\varphi_{i,j}\) making the following diagram commute:
\begin{center}
\adjustbox{scale=0.7,center}{%
\begin{tikzcd}[column sep=normal, row sep=large]
                                                                                                                                                          &                                                                                     \\
                                                                                &D                                 &                                                                                     \\
                                                                                &  U    \arrow[swap, dashed]{u}{\exists ! \theta}  &                                                                                     \\
U_i \arrow{ru}{\iota_i} \arrow[bend left]{ruu}{\dindsi{\iota}{D}{i}} &                                                                              & U_j \arrow[swap]{lu}{\iota_j} \arrow[swap,bend right]{luu}{\dindsi{\iota}{D}{j}}  \\
                                                                                & U_i \times_U U_j \arrow{lu}{\pi_1^{U_i \times_U U_j}} \arrow[swap]{ru}{\pi_2^{U_j \times_U U_i}}            & \\                                                                                   
                                                                                & \mathbf{G}(i,j) \arrow[swap,dashed]{u}{\varphi_{i,j}}\arrow[bend left]{luu}{\mathbf{G}(\iuv{i}{j})}\arrow[swap,bend right]{ruu}{\mathbf{G}(\tau_{j,i}\circ \iuv{j}{i})} &

\end{tikzcd}}
\end{center}
Therefore, \((D, \subsm{(\dindsi{\iota}{D}{i})}{i\! \in\! \mathrm{I}})\) is a cone over \(\mathbf{G}\). Since \(U\) is a glued-up object over \(\mathbf{G}\) through \(\subsm{(\iota_i)}{i\! \in \!\mathrm{I}}\), there exists a unique morphism \(\theta: U \rightarrow D\) such that \(\dindsi{\iota}{D}{i} = \theta \circ \iota_i\) for all \(i \in \mathrm{I}\). The universal component of the result follows by an analogous argument.
\end{proof}

\subsection{Definitions}

\noindent We recall the definition of Grothendieck topologies on a category equipped with finite limits, which is a fundamental notion in algebraic geometry.

\begin{definition}\label{sito} 
\hspace{2em}
\begin{enumerate}	
\item A {\sf Grothendieck topology} on $\mathbb{C}$ is given by a subset of $\subsm{\mathbf{Sink}}{}(\mathbb{C})$, denoted by $\subsm{\mathbf{Cov}}{}(\mathbb{C})$, satisfying the following conditions:
	\begin{enumerate}
		\item If $\varphi: V\rightarrow U$ is an isomorphism in $\mathbb{C}$, then $(U, (V, \varphi)) \in \subsm{\mathbf{Cov}}{}(\mathbb{C})$.
		\item If $(U,\subsm{(U_i,\iota_i)}{i\!\in \! \mathrm{I}})\in \subsm{\mathbf{Cov}}{}(\mathbb{C})$, and $(U_i, \subsm{(\subsm{V}{ij}, \ell_{ij})}{j\!\in \! \mathrm{J}_i})\in \subsm{\mathbf{Cov}}{}(\mathbb{C})$, then $(U, \subsm{(\subsm{V}{ij}, \iota_i\circ \ell_{ij})}{i\!\in \! \mathrm{I}, j\in \subsm{\mathrm{J}}{i}})$ is in $\subsm{\mathbf{Cov}}{}(\mathbb{C})$. We refer to this property as the composability property of coverings. 
		\item If $(U,\subsm{(\subsm{U}{i},\iota_i)}{i\!\in \! \mathrm{I}})\in \subsm{\mathbf{Cov}}{}(\mathbb{C})$, and $V\rightarrow U$ is a morphism in $\mathbb{C}$, 
		then $(V, \subsm{(\subsm{U}{i}\subsm{\times}{\scalebox{0.7}{$U$}} V, \proj{2}{\subsm{U}{i}\subsm{\times}{\scalebox{0.7}{$U$}} V})}{i\!\in \! \mathrm{I}})\in \subsm{\mathbf{Cov}}{}(\mathbb{C})$. We refer to this property as the base change stable property of coverings.
			\end{enumerate}
\item	A {\sf Grothendieck site} is a pair $(\mathbb{C}, \subsm{\mathbf{Cov}}{}(\mathbb{C}))$, where $\mathbb{C}$ is a category and $\subsm{\mathbf{Cov}}{}(\mathbb{C})$ is a Grothendieck topology. 
\end{enumerate}
\end{definition}
%
%
%
\noindent We now recall the definition of a localized Grothendieck site. 
\begin{definition}
Let $(\mathbb{C}, \subsm{\mathbf{Cov}}{}(\mathbb{C}))$ to be a Grothendieck site and $U \in \mathbb{C}$.
\begin{enumerate}
\item We define the {\sf localized Grothendieck topology of $(\mathbb{C}, \subsm{\mathbf{Cov}}{}(\mathbb{C}))$ at $U$}, denoted by 
$$ \begin{array}{lll} \subsm{\mathbf{Cov}}{}(\mathbb{C} \! \downarrow \! U) &=&\big \{ ((V, \delta),\subsm{((V_i, \delta_i),\iota_i)}{i\!\in \! \mathrm{I}})| (V,\subsm{(\subsm{V}{i},\iota_i)}{i\!\in \! \mathrm{I}})\in \subsm{\mathbf{Cov}}{}(\mathbb{C}), (V,\delta) ,(V_i, \delta_i) \in {(\mathbb{C} \! \downarrow \! U)}, \\
&& \iota_i : (V_i, \delta_i) \rightarrow (V, \delta) \text{ morphism in $(\mathbb{C} \! \downarrow \! U)$}, \forall i \in \mathrm{I} \big\}.\end{array} $$
\item We define the {\sf localized Grothendieck site of $(\mathbb{C}, \subsm{\mathbf{Cov}}{}(\mathbb{C}))$ at $U$} to be the pair $$\big((\mathbb{C} \! \downarrow \! U), \subsm{\mathbf{Cov}}{}(\mathbb{C} \! \downarrow \! U) \big).$$ 
\end{enumerate}
\end{definition}
\begin{example}\label{settop}
Let \(\mathbb{C} \in \{ \mathbb{Sets}, \mathbb{(o)Top} \}\). The category \(\mathbb{C}\) can be equipped with a Grothendieck topology as follows.

\noindent Given an index set \(I\) and a family of objects \((U_i)_{i \in I}\) in \(\mathbb{C}\), we set \(U := \coprod_{i \in I} U_i \), and let \(\iota_i : U_i \to U\) denote the canonical coproduct maps.

\noindent We define the covering families \(\mathbf{Cov}(\mathbb{C})\) to be the collection of all tuples of the form \((U, (U_i, \iota_i)_{i \in I})\), for arbitrary index sets \(I\) and families \((U_i)_{i \in I}\) in \(\mathbb{C}\).
\end{example}

\subsection{Effective gluing on a Grothendieck site }
In algebraic geometry, gluing typically involves a gluing datum with cocyle conditions. In our definition of a gluing functor, these cocyle conditions are not included; instead, we encode it as a property of the functor itself, which we call an effective gluing functor on a site.
\begin{definition}\label{cocycle}
  Let $(\mathbb{C}, \subsm{\mathbf{Cov}}{}(\mathbb{C}))$ to be a Grothendieck site, and $\mathbf{G}$ be a split $\mathbb{C}^\op$-gluing functor of type $\mathrm{I}$. 
  We say that $\mathbf{G}$ is an {\sf effective split $\mathbb{C}^\op$-gluing functor} if, for all $i,j,k\in \mathrm{ I}$,
  there exist an isomorphism $\widetilde{\varphi_{i,j}}^k$ from $\mathbf{G}(j,i)\times_{\mathbf{G}(j)} \mathbf{G}(j,k)$ to $\mathbf{G}(i,j)\times_{\mathbf{G}(i)} \mathbf{G}(i,k)$ such that, for all $i,j \in \mathrm{I}$, the following assertions hold: 
  \begin{enumerate}
  \item the following diagram commutes: 
  $$\xymatrix{ \mathbf{G}(j,i)\times_{\mathbf{G}(j)} \mathbf{G}(j,k) \ar[d]_{\widetilde{\pi_{j,i}}^k}\ar[r]^{\widetilde{\varphi_{i,j}}^k} &\mathbf{G}(i,j)\times_{\mathbf{G}(i)} \mathbf{G}(j,k) \ar[d]^{\widetilde{\pi_{i,j}}^k} \\ 
  \mathbf{G}(j,i)   \ar[r]_{\mathbf{G}(\tau_{i,j})^\op} & \mathbf{G}(i,j)
  },$$
  where $\widetilde{\pi_{i,j}}^k$ is the canonical pullback map. 
We refer to this condition as the \textsf{cocycle condition}.

  \item The following diagram commutes: 
   $$\xymatrix{ \mathbf{G}(j,i)\times_{\mathbf{G}(j)} \mathbf{G}(j,k) \ar[dr]_{\widetilde{\varphi_{k,j}}^i}\ar[rr]^{\widetilde{\varphi_{i,j}}^k} &&\mathbf{G}(i,j)\times_{\mathbf{G}(i)} \mathbf{G}(j,k) \ar[dl]^{\widetilde{\varphi_{k,i}}^j} \\ 
& \mathbf{G}(k,i)\times_{\mathbf{G}(k)} \mathbf{G}(k,j)  &
  },$$
  \item $\widetilde{\varphi_{j,i}}^k= ({\widetilde{\varphi_{i,j}}^k})^{-1}$;
  \item $\mathbf{G} (\mathfrak{i}_{i,j})^\op$ is a regular epimorphism.
  \end{enumerate} 
\end{definition}
\noindent Even though the cocycle condition cannot be defined for a general non-split gluing functor, we can introduce a stronger version that we show implies the cocycle condition.

\begin{definition}
  Let $(\mathbb{C}, \subsm{\mathbf{Cov}}{}(\mathbb{C}))$ to be a Grothendieck site,  $\mathbf{G}$ be a split $\mathbb{C}^\op$-gluing functor of type $\mathrm{I}$, and $L$ be a glued-up object over $\mathbf{G}$ via the family \(\subsm{(\iota_i)}{i \in \mathrm{I}}\). 

\noindent We say that $\mathbf{G}$ is a \textsf{strong effective (split) $\mathbb{C}^\op$-gluing functor} if, for all $i,j \in \mathrm{I}$ with $i \neq j$, 
\begin{itemize}
\item the canonical morphism
\[
\langle \Gnij{\mbf{G}}{\iuv{i}{j}}^\op, \Gnij{\mbf{G}}{(\tau_{j,i}\circ)\iuv{j}{i}}^\op \rangle:  \mathbf{G}(i,j) \longrightarrow \mathbf{G}(i) \times_L \mathbf{G}(j)
\]
is an isomorphism.
\item $\mathbf{G} (\mathfrak{i}_{i,j})^\op$ is a regular epimorphism.
\end{itemize}
\end{definition}


\begin{lemma}\label{strongsplit}
  Let $(\mathbb{C}, \subsm{\mathbf{Cov}}{}(\mathbb{C}))$ to be a Grothendieck site,  $\mathbf{G}$ be a split $\mathbb{C}^\op$-gluing functor of type $\mathrm{I}$, and $L$ be a glued-up object over $\mathbf{G}$ via the family \(\subsm{(\iota_i)}{i \in \mathrm{I}}\).  Then $\mathbf{G}$ is a split effective $\mathbb{C}^\op$-gluing functor of type $\mathrm{I}$.
\end{lemma}

\begin{proof}
Let $\mathbf{G}$ be a split strong effective $\mathbb{C}^\op$-gluing functor of type $\mathrm{I}$. For each triple $i,j,k \in \mathrm{I}$, define:

\begin{itemize}
    \item \( U_{j,k}^i := (\mathbf{G}(i) \times_L \mathbf{G}(j)) \times_{\mathbf{G}(i)} (\mathbf{G}(i) \times_L \mathbf{G}(k)) \). For simplicity, assume \( U_{j,k}^i = U_{k,j}^i \).
    \item \( V_{j,k}^i := \mathbf{G}(i,j) \times_{\mathbf{G}(i)} \mathbf{G}(i,k) \). Again, assume \( V_{j,k}^i = V_{k,j}^i \).
    \item \( \widetilde{c_{i,j}}^k \) is the canonical isomorphism from \( U_{j,k}^i \) to \( U_{i,k}^j \).
    \item \( \widetilde{d_{j,k}}^i \) is the canonical isomorphism from \( V_{j,k}^i \) to \( U_{j,k}^i \), defined by:
    \[
    \langle \langle \mathbf{G}(\mathfrak{i}_{i,j})^\op, \mathbf{G}(\tau_{j,i} \circ \mathfrak{i}_{j,i})^\op \rangle, \langle \mathbf{G}(\mathfrak{i}_{i,k})^\op, \mathbf{G}(\tau_{k,i} \circ \mathfrak{i}_{k,i})^\op \rangle \rangle.
    \]
    \item \( \widetilde{\varphi_{i,j}}^k \) is the canonical isomorphism from \( V_{j,k}^i \) to \( V_{i,k}^j \), defined by:
    \[
    \widetilde{\varphi_{i,j}}^k :=(\widetilde{d_{i,j}}^k)^{-1} \circ \widetilde{c_{i,j}}^k \circ  \widetilde{d_{j,i}}^k.
    \]
\end{itemize}

\noindent One can verify that the isomorphisms \( \widetilde{\varphi_{i,j}}^k \) satisfy the cocycle condition in Definition~\ref{cocycle}, since the canonical isomorphisms \( \widetilde{c_{i,j}}^k \) do. The other property is also deduce easily from the definition of \( \widetilde{\varphi_{i,j}}^k \)
\end{proof}

\noindent The following proposition shows that the effectiveness of a gluing functor $\mathbf{G}$ is equivalent to the classical notion of gluing data satisfying the cocycle condition. In this setting, in the category of topological spaces, requiring the gluing functor to be effective amounts to \(\mathbf{G}(i)\) being identifiable as a subobject of the glued object, and \(\mathbf{G}(i,j)\) being identifiable to the intersection \(\mathbf{G}(i) \cap \mathbf{G}(j)\) within the glued object, for all $i,j\in \mathrm{I}$.

\begin{proposition} \label{eqrel} In this proposition, we use the notation of Definition-Lemma \ref{genera}. Let $\mathbb{C} \in \big\{ \mathbb{Sets}, \mathbb{(o)Top} \big\}$ with the Grothendieck topology as defined in Example \ref{settop}, and let $\mathbf{G}$ be a functor from $\mathbb{S}_2(\mathrm{I})$ to $\mathbb{C}^{\op}$ such that \(\limi{\mathbf{G}} \simeq (U, \iota)\), where \(\iota = \subsm{(\iota_i)}{i \in \mathrm{I}}\).

The following statements are equivalent:
\begin{enumerate}
    \item $\mathbf{G}$ is a split effective $\mathbb{C}^\op$-gluing functor.
    \item \begin{itemize} 
    \item $\Rel{\mathbf{G}}$ is a congruence relation, and 
    \item for all \(i, j \in \mathrm{I}\), the morphism \(\mathbf{G}(\mathfrak{i}_{i,j})^{\op}\) are one-to-one. (resp. topological embeddings when $\mathbb{C}= \mathbb{(o)Top}$)
    \end{itemize}
    \item{itemize}
    \item For all \(i,j \in \mathrm{I}\),
\begin{itemize} 
\item $    {\dindi{\iota}{Q}{\mathbf{G}}}_{\! i}\left(\Gnij{\mathbf{G}}{\iuv{i}{j}}^{\op}(\mathbf{G}(i,j))\right) = {\dindi{\iota}{Q}{\mathbf{G}}}_{\! i}(\mathbf{G}(i)) \cap {\dindi{\iota}{Q}{\mathbf{G}}}_{\! j}(\mathbf{G}(j)),$
    and 
    \item \( \iota_i\) and \(\mathbf{G}(\mathfrak{i}_{i,j})^{\op}\) are one-to-one. (resp. topological embeddings when $\mathbb{C}= \mathbb{(o)Top}$)
    \end{itemize}
    \item $\mathbf{G}$ is a strong split effective $\mathbb{C}^\op$-gluing functor.
\end{enumerate}

In particular, each \(\iota_i\) is injective (resp. topological embedding when $\mathbb{C}= \mathbb{(o)Top}$) for all \(i \in \mathrm{I}\).
\end{proposition}

\begin{proof}
Let \( U_i := \mathbf{G}(i) \), \( U_{ij} := \Gnij{\mathbf{G}}{\iuv{i}{j}}^{\op}(\mathbf{G}(i,j)) \), \( \varphi_{ij} := \mathbf{G}(\tau_{i,j}^\op) \), and \(\iota_i := {\dindi{\iota}{Q}{\mathbf{G}}}_{\! i}\).

\smallskip

\noindent \textbf{(1) $\Rightarrow$ (2)} \quad
We demonstrate that \(\Rel{\mathbf{G}}\) is an equivalence relation. Reflexivity and symmetry are immediate. For transitivity, consider \((x,i), (y,j), (z,k) \in \coprod_{i \in \mathrm{I}} \Goi{\mathbf{G}}{i}\) such that \((x,i) \Rel{\mathbf{G}} (y,j)\) and \((y,j) \Rel{\mathbf{G}} (z,k)\). By definition, there exist \( u \in \Goij{\mathbf{G}}{j}{i} \) and \( v \in \Goij{\mathbf{G}}{j}{k} \) such that:
\begin{itemize}
    \item \( x = \Gnij{\mathbf{G}}{\tau_{i,j} \circ \iuv{i}{j}}^{\op}(u), \quad y = \Gnij{\mathbf{G}}{\iuv{j}{i}}^{\op}(u) \),
    \item \( y = \Gnij{\mathbf{G}}{\iuv{j}{k}}^{\op}(v), \quad z = \Gnij{\mathbf{G}}{\tau_{k,j} \circ \iuv{k}{j}}^{\op}(v) \).
\end{itemize}

Then, \((u,v) \in \mathbf{G}(j,i) \times_{\mathbf{G}(j)} \mathbf{G}(j,k)\). By the properties of a split effective gluing functor, we obtain \(\widetilde{\varphi_{i,j}}^k(u,v) \in \mathbf{G}(i,j) \times_{\mathbf{G}(i)} \mathbf{G}(i,k)\) for some isomorphism \(\widetilde{\varphi_{i,j}}^k\) as in Definition \ref{cocycle}. Consequently,
$$\widetilde{\pi_{j,i}}^k ( \widetilde{\varphi_{i,j}}^k(u,v)) = \mathbf{G} (\tau_{i,j})^\op ( \widetilde{\pi_{i,j}}^k(u,v)) = \mathbf{G} (\tau_{i,j})^\op ( u)$$
and there exists \( w \in \mathbf{G}(i,k) \) such that:
$$ \Gnij{\mathbf{G}}{\iuv{i}{k}}^{\op} (w) = \Gnij{\mathbf{G}}{\iuv{i}{j}}^{\op}(\mathbf{G} (\tau_{i,j})^\op ( u)) = x. $$

We recall that from the cocycle condition, we have the following commutative diagram:
$$
\xymatrix{
    \mathbf{G}(j,i)\times_{\mathbf{G}(j)} \mathbf{G}(j,k) \ar[dr]_{\widetilde{\varphi_{k,j}}^i} \ar[rr]^{\widetilde{\varphi_{i,j}}^k} && \mathbf{G}(i,j)\times_{\mathbf{G}(i)} \mathbf{G}(i,k) \ar[dl]^{\widetilde{\varphi_{k,i}}^j} \\
    & \mathbf{G}(k,i)\times_{\mathbf{G}(k)} \mathbf{G}(k,j) &
}
$$

Since \(\widetilde{\varphi_{i,j}}^k (u,v) = (\mathbf{G} (\tau_{i,j})^\op ( u), w)\), we obtain from the commutativity of the previous diagram that:
$$ \widetilde{\varphi_{k,j}}^i (u,v) = (\mathbf{G} (\tau_{k,i})^\op (w), \mathbf{G} (\tau_{k,j})^\op (v)) $$

Thus, \( z = \Gnij{\mathbf{G}}{\tau_{k,j} \circ \iuv{k}{j}}^{\op}(v) = \Gnij{\mathbf{G}}{\tau_{k,i} \circ \iuv{k}{i}}^{\op}(w) \), and hence \((x,i) \Rel{\mathbf{G}} (z,k)\), proving transitivity.

Since morphisms \(\Gnij{\mathbf{G}}{\tau_{i,j} \circ \iuv{i}{j}}^{\op}\) and \(\Gnij{\mathbf{G}}{\iuv{i}{j}}^{\op}\) are morphisms in \(\mathbb{C}\), \(\Rel{\mathbf{G}}\) is a congruence relation.

Moreover, since \(\Rel{\mathbf{G}}\) is an equivalence relation, each map \(\iota_i\) is injective, and the resulting maps are topological embeddings when \(\mathbb{C}=\mathbb{(o)Top}\). This follows from the assumption that the morphisms \(\mathbf{G}(\mathfrak{i}_{i,j})^{\op}\) are topological embeddings, as established in Definition-Lemma~\ref{genera}.

\noindent \textsf{(2) $\Rightarrow$ (3)} \quad
Let \( i,j \in \mathrm{I} \). Since \(\iota_i \circ \Gnij{\mathbf{G}}{\iuv{i}{j}}^{\op} = \iota_j \circ \Gnij{\mathbf{G}}{\iuv{j}{i}}^{\op} \), it follows that:
\[
\iota_i(\Gnij{\mathbf{G}}{\iuv{i}{j}}^{\op}(\mathbf{G}(i,j))) \subseteq \iota_i(\mathbf{G}(i)) \cap \iota_j(\mathbf{G}(j)).
\]
Conversely, let \( x \in \iota_i(\mathbf{G}(i)) \cap \iota_j(\mathbf{G}(j)) \). Then \( x = \pi(u, i) = \pi(v, j) \) for some \( u \in \Goi{\mathbf{G}}{i}, v \in \Goi{\mathbf{G}}{j} \), implying \((u,i) \Rel{\mathbf{G}} (v,j)\) since $\Rel{\mathbf{G}}$ is an equivalence relation by assumption. So there exists \( w \in \Goij{\mathbf{G}}{i}{j} \) with \( u = \Gnij{\mathbf{G}}{\iuv{i}{j}}^{\op}(w) \), hence \( x \in \iota_i(\Gnij{\mathbf{G}}{\iuv{i}{j}}^{\op}(\mathbf{G}(i,j))) \).

\smallskip

\noindent \textsf{(3) $\Rightarrow$ (4)} \quad
Let \(\widetilde{\iota_i}\) denote \(\iota_i\) restricted to its image, and likewise for \(\widetilde{\iota_j \circ \Gnij{\mathbf{G}}{\iuv{j}{i}}^{\op}}\). Consider the following diagram:
\[
\xymatrix{
\Goij{\mathbf{G}}{i}{j}
\ar[r]^-{\widetilde{\iota_j \circ \Gnij{\mathbf{G}}{\iuv{j}{i}}^{\op}}}  \ar[d]_{\langle \Gnij{\mathbf{G}}{\iuv{i}{j}}, \Gnij{\mathbf{G}}{\tau_{j,i} \circ \iuv{j}{i}} \rangle}&
{\dindi{\iota}{Q}{\mathbf{G}}}_i\left(\Gnij{\mathbf{G}}{\iuv{i}{j}}^{\op}(\mathbf{G}(i,j))\right)
 \\
 \mathbf{G}(i) \times_{L} \mathbf{G}(j)
 \ar[r]_-{\widetilde{\iota_i} \times \widetilde{\iota_j}} & \iota_i(\mathbf{G}(i)) \cap \iota_j(\mathbf{G}(j)) \ar@{=}[u]
}
\]
All horizontal arrows are isomorphisms by assumption. \\
Therefore, the vertical map \(\langle \Gnij{\mathbf{G}}{\iuv{i}{j}}, \Gnij{\mathbf{G}}{\tau_{j,i} \circ \iuv{j}{i}} \rangle\) is also an isomorphism, and thus \(\mathbf{G}\) is a strong split effective gluing functor.

\noindent \textsf{(4) $\Rightarrow$ (1)} \quad follows from Lemma \ref{strongsplit}.
\end{proof}

\begin{figure}[H]
	\begin{tikzpicture}[fill=gray, scale=0.8]
		
		\path [draw,right hook->](6.3,5.8) to [bend right=30]  node[midway, above left] {\scalebox{0.6}{$\Gnij{\mbf{G}}{\iuv{1}{2}}$}} (2.5,2.3) ;
		\path [draw,left hook->](-4,-1.5) to [bend left=30] node[midway, above left] {\scalebox{0.6}{$\Gnij{\mbf{G}}{\iuv{1}{3}}$}}  (0,1);
		
		\path [draw,right hook->](-3.4,-3) to [bend right=30] node[midway, below left] {\scalebox{0.6}{$\Gnij{\mbf{G}}{\iuv{3}{1}}$}}  (0.1,-5.5);
		\path [draw,left hook->] (10,-9.1) to [bend left=30] node[midway, below left] {\scalebox{0.6}{$\Gnij{\mbf{G}}{\iuv{3}{2}}$}}  (2,-7.4);
		
		\path [draw,right hook->] (11.1,-8) to [bend right=30] node[midway, below right] {\scalebox{0.6}{$\Gnij{\mbf{G}}{\iuv{2}{3}}$}}  (11.5,-1.1);
		
		\path [draw,left hook->](7.5,6) to [bend left=30]  node[midway, above right] {\scalebox{0.6}{$\Gnij{\mbf{G}}{\iuv{2}{1}}$}}  (10,1.4);

		\path [draw,dashed,thick,->](0.2,-5) to [bend left =30]  node[midway, above left] {\scalebox{0.6}{$\varphi_{31}$}}  (0.5,0.2);
		\path [draw,dotted](0.5,0.2) to [bend left =20]  (3,1.5);
		\path [draw,dashed,thick,->](3,1.5) to [bend left =30]  node[midway, above left] {\scalebox{0.6}{$\varphi_{12}$}}  (9,0.4);
		\path [draw,dotted](9,0.4) to [bend left =20]  (10.7,-1.5);
		\path [draw,dashed,thick,->](10.7,-1.5) to [bend left =60] node[midway, below right] {\scalebox{0.6}{$\varphi_{23}$}}   (2.9,-6.8);
		\path [draw,dotted](2.9,-6.8) to [bend left =20]  (0.2,-5);
		
		\path [draw,dashed](2,-0.6) to node[ above ] {\scalebox{0.6}{$\subsm{\iota_{1}|}{U_{1,2}\cap U_{1,3}}$}}   (5,-1.3);
		
		\path [draw,dashed](1.5,1.5) to  (2,-0.6);
		
		\path [draw,dashed] (5,-1.3) to (6,-1.5);
		
		\path [draw,dashed](6,-1.5) to  (6,-3.4);
		
		\path [draw,dashed] (3.2,-5) to node[below right] {\scalebox{0.6}{$\subsm{\iota_{2}|}{U_{3,1}\cap U_{3,2}}$}}   (4.3,-4);
		
		\path [draw,dashed] (1.5,-6) to  (2.5,-5);
		
		\path [draw,dashed] (2.5,-5) to (3.2,-5);
		
		\path [draw,dashed] (4.3,-4) to (6,-3);
		
		\path [draw,dashed] (9.5,-1.7) to node[midway, below right] {\scalebox{0.6}{$\subsm{\iota_{3}|}{U_{2,1}\cap U_{2,3}}$}}  (7.79,-4);
		
		\path [draw,dashed] (10.5,0) to (9,-1);
		
		\path [draw,dashed]  (9,-1) to (9.5,-1.7) ;
		
		\path [draw,dashed]   (7.79,-4) to   (6,-3);
		
		\path [draw,->]   (1.5,-0.7) to [bend right =30]  node[below ] {\scalebox{0.6}{$\iota_1$}}  (4.54,-3) ;
		\path [draw,->]   (8.57,0.5) to [bend right =30]  node[midway, above ] {\scalebox{0.6}{$\iota_2$}}  (7,-1.85);
		\path [draw,->]   (3.4,-6.7) to [bend right =30]  node[midway, below ] {\scalebox{0.6}{$\iota_3$}}  (6,-4.69) ;
		
		\draw (1,1) circle (1) (0.5,1.5)  node [text=black,below] {{\scalebox{0.6}{ $U_{1,3}$}}}
		(2,1.5) circle (1) (0.5,2.7)  node [text=black,below] {{\scalebox{0.6}{ $U_1$}}}
		(1.5,1.3) circle (2) (2,1.7) node [text=black,above] {\scalebox{0.6}{ $U_{1,2}$}};
		
		\draw (6,-3.7) circle (1) (6,-4)  node [text=black,below] {\scalebox{0.6}{$\iota_3(U_{3})$}}
		(5.5,-2.7) circle (1) (7,-2)  node [text=black,below] {\scalebox{0.6}{$\iota_2(U_{2})$}}
		(6.5,-2.7) circle (1) (5,-2)  node [text=black,below] {\scalebox{0.6}{$\iota_1(U_{1})$}};
		
		\draw (10,0.4) circle (1) (10,1.2)  node [text=black,below] {\scalebox{0.6}{$U_{2,1}$}}
		(10.7,-0.5) circle (1) (11,-0.6)  node [text=black,below] {\scalebox{0.6}{$U_{2,3}$}}
		(10.5,0) circle (2) (11.5,1) node [text=black,above] {\scalebox{0.6}{$U_2$}};
		
		\draw (1.1,-5.5) circle (1) (1,-4.7)  node [text=black,below] {\scalebox{0.6}{ $U_{3,1}$}}
		(2,-6.4) circle (1) (2.1,-6.7)  node [text=black,below] {\scalebox{0.6}{$U_{3,2}$}}
		(1.5,-6) circle (2) (2.5,-5) node [text=black,above] {\scalebox{0.6}{$U_3$}};

		\draw (-3,-2) circle (1) (-3,-2)  node [text=black,below] {\scalebox{0.6}{$\mathbf{G}{(3,1)}$}};

		\draw (7,5) circle (1) (7,5)  node [text=black,below] {\scalebox{0.6}{$\mathbf{G}{(1,2)}$}};

		\draw (10,-8) circle (1) (10,-8)  node [text=black,below] {\scalebox{0.6}{$\mathbf{G}{(2,3)}$}};

	\end{tikzpicture}
	\caption{Representation of the process of  an effective gluing of three topological spaces in $\mathbb{Top}^{\operatorname{op}}$}  
\end{figure}
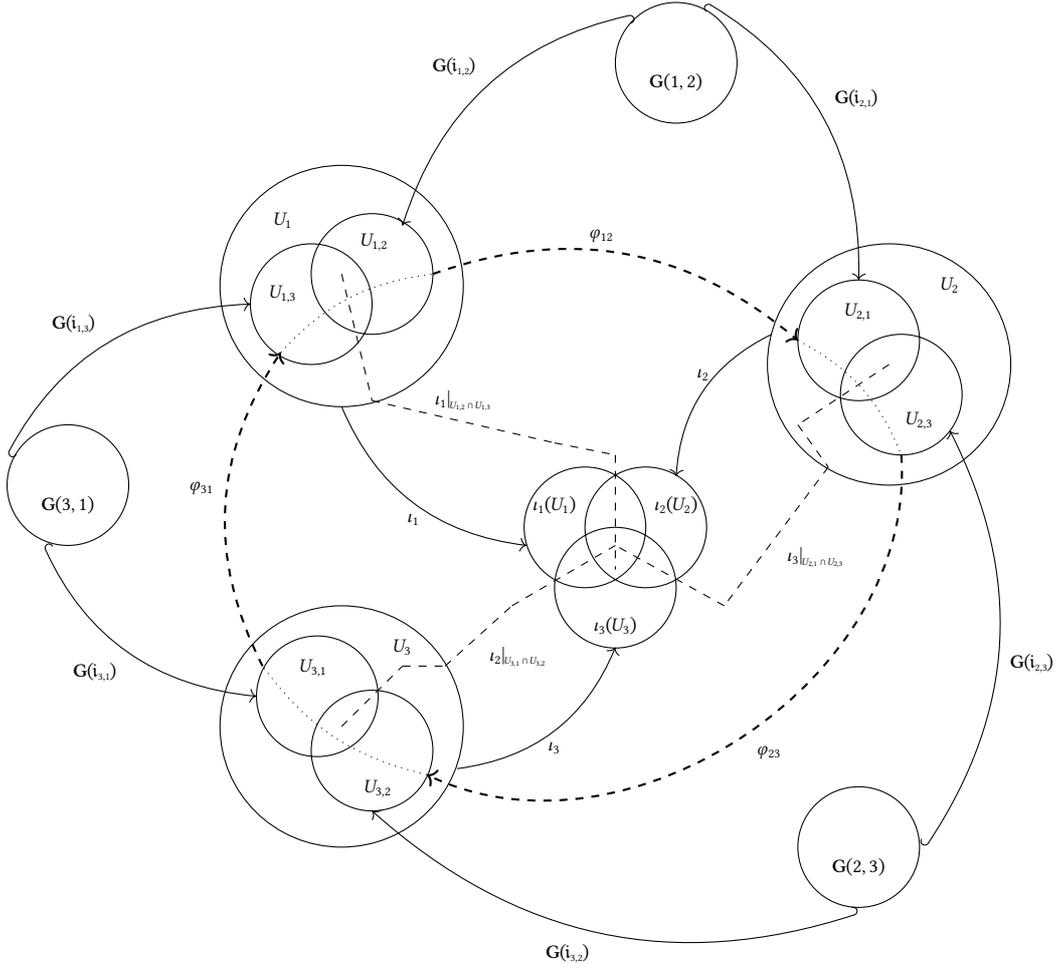

\subsection{Effective Grothendieck site}From this point onward, we assume that \(\mathbb{C}\) is a locally small category for simplicity. We now recall the notions of effective epimorphisms and universal effective epimorphisms.

\begin{definition}\cite[Example (0.3)]{artin}
Let $(U,\subsm{(U_i, \iota_i)}{i\!\in \! \mathrm{I}})$ be a sink in $\mathbb{C}$. We set $\iota:=\subsm{(\iota_i)}{i\!\in \mathrm{I}}$. 

\begin{enumerate} 
\item Let $Z\in \mathbb{C}$. We define the maps 
\begin{itemize} 
\item $\subsm{d}{0}^{Z,\iota}:=\subsm{\langle \iota_i^\ast \rangle}{i\!\in \!\mathrm{I}}:{\textsf{Hom}(U,Z)}\rightarrow {\subsm{\prod}{i\!\in\! \mathrm{I}} \textsf{Hom}(\subsm{U}{i},Z)}$, 
\item $\subsm{d}{1}^{Z,\iota}$ from ${\subsm{\prod}{i\!\in\! \mathrm{I}} \textsf{Hom}(\subsm{U}{i},Z)}$ to ${\subsm{\prod}{(i,j)\! \in\! \mathrm{I}^2 }^{ }\textsf{Hom}(\subsm{U}{i}\subsm{\times}{U} \subsm{U}{j},Z)}$ sending $\subsm{(f_i)}{i\!\in\! \mathrm{I}} $ to $\subsm{(f_i\circ \proj{1}{\subsm{U}{i}\subsm{\times}{\scalebox{0.7}{$U$}} \subsm{U}{j}})}{(i,j)\! \in \!\mathrm{I}^2}$, and 
\item $\subsm{d}{2}^{Z,\iota}$ from ${\subsm{\prod}{i\!\in\! \mathrm{I}} \textsf{Hom}(\subsm{U}{i},Z)}$ to ${\subsm{\prod}{(i,j)\! \in\! \mathrm{I}^2 }^{ }\textsf{Hom}(\subsm{U}{i}\subsm{\times}{U} \subsm{U}{j},Z)}$ sending $\subsm{(f_i)}{i\!\in\! \mathrm{I}} $ to $\subsm{(f_i\circ \proj{2}{\subsm{U}{i}\subsm{\times}{\scalebox{0.7}{$U$}} \subsm{U}{j}})}{(i,j)\! \in \!\mathrm{I}^2}$. 
\end{itemize}
\item We say that $\iota$ is an {\sf effective epimorphism} if the diagram
\begin{center}
{\footnotesize \begin{tikzcd}[column sep=normal]
 {\textsf{Hom}(U,Z)} \arrow{rr}{\subsm{d}{0}^{Z,\iota}}                            &  & {\subsm{\prod}{i\!\in\! \mathrm{I}} \textsf{Hom}(\subsm{U}{i},Z)} \arrow[shift left]{rr}{\subsm{d}{1}^{Z,\iota}} \arrow[shift right,swap]{rr}{\subsm{d}{2}^{Z,\iota}} &  & {\subsm{\prod}{(i,j)\! \in\! \mathrm{I}^2 }^{ }\textsf{Hom}(\subsm{U}{i}\subsm{\times}{U} \subsm{U}{j},Z)} \end{tikzcd}}
\end{center} 
\noindent is an equalizer diagram in the category of sets, for all $Z \in \mathbb{C}$.
\item We say that  $\iota$ is a {\sf universal effective epimorphism} if for all $(V,f) \in {(\mathbb{C} \! \downarrow \! U)}$,  $\subsm{\left( \proj{2}{U_i \subsm{\times}{\scalebox{0.7}{$U$}} V}\right) }{i\!\in \! \mathrm{I}}$ is effective epimorphism. 
\end{enumerate}
\end{definition}



\begin{remark}\label{rem38}
Let \(f: V \rightarrow U\) be a morphism in $\mathbb{C}$ and that $V\times_U V$ exist in $\mathbb{C}$.
\begin{enumerate}
    \item A singleton family \(\{f\}\) is an effective epimorphism if and only if \(f\) is an effective epimorphism. That is, the following diagram is a coequalizer diagram:
    \begin{center}
    \small
    \begin{tikzcd}
        V \times_U V \arrow[shift left]{rr}{\pi_2^{ V \times_U V} } \arrow[shift right,swap]{rr}{\pi_1^{ V \times_U V} } & & V \arrow{rr}{f} && U
    \end{tikzcd}
    \end{center}
    In particular, any isomorphism is an effective epimorphism.
    
    \item A singleton family \(\{f\}\) is a universal effective epimorphism if and only if \(f\) is a descent morphism (see \cite[Definition 1.5]{janelidze1994facets}).
\end{enumerate}
\end{remark}

\noindent In the following proposition, we show that a morphism is an effective epimorphism if and only if the canonical gluing functor associated with the sink defined by this morphism is a split effective gluing functor. This result justifies the terminology "effective gluing functor" used in this context.

\begin{proposition}\label{lemm39}
	Let \(\EuScript{U} = (U, \subsm{(U_i, \iota_i)}{i\! \in\! \mathrm{I}})\) be a sink in \(\mathbb{C}\). The following assertions are equivalent:
	\begin{enumerate}
		\item The family $\subsm{(\iota_i)}{i\!\in \! \mathrm{I}}$ is an effective epimorphism (resp. universal effective epimorphism).
			\item $U$ is a glued-up object over $\gcov{\EuScript{U}}$ through $\subsm{(\iota_i)}{i\!\in \! \mathrm{ I}}$ (resp. $U$ is a universal glued-up object over $\gcov{\EuScript{U}}$ through  through $\subsm{(\iota_i)}{i\!\in \! \mathrm{ I}}$).
			\item $\gcov{\EuScript{U}}$ is a effective split $\mathbb{C}^\op$-gluing functor (resp. $\gcov{\EuScript{U}}$ is a universal effective split $\mathbb{C}^\op$-gluing functor).
		\item $\emph{\textsf{Hom}}(U,-)$ is a glued-up object over $\emph{\textsf{Hom}}(\gcov{\EuScript{U}})$ through $\subsm{\left({\iota_i^\ast} \right)}{i\!\in\! \mathrm{I}}$ (resp.  $\emph{\textsf{Hom}}(U,-)$ is a universal glued-up object over $\emph{\textsf{Hom}}(\gcov{\EuScript{U}})$ through $\subsm{\left({\iota_i^\ast} \right)}{i\!\in\! \mathrm{I}}$).  
	\end{enumerate}
\end{proposition}
\begin{proof}
We set $\iota:=\subsm{(\iota_i)}{i\!\in \! \mathrm{I}}$ and let $Z\in \mathbb{C}$.
\begin{enumerate}
	\item[$(1)\Rightarrow (2)$] Suppose that \(\iota\) is an effective epimorphism. We prove that \((U, \subsm{(U_i, \iota_i)}{i\! \in\! \mathrm{I}})\) is a terminal cone over \(\gcov{\EuScript{U}}\). By taking \(Z = U\) and \(\operatorname{id} \in \textsf{Hom}(U, U)\) in the diagram of Definition \ref{gluingcover} (2), we deduce that \((U, \subsm{(U_i, \iota_i)}{i\! \in\! \mathrm{I}})\) is a cone over \(\gcov{\EuScript{U}}\).

\noindent Let \((Z, \subsm{(Z_a, \iota_{Z_a})}{a\! \in\! \mathbb{P}_2(\mathrm{I})})\) be a cone over \(\gcov{\EuScript{U}}\). We prove that there exists \(\varphi: U \rightarrow Z\) making the following diagram commute.

	\begin{figure}[H]
\begin{center}
{\tiny\begin{tikzcd}[column sep =normal]
                                                                                                                                                          &                                                                                     \\
                                                                                &Z                                 &                                                                                     \\
                                                                                &  U    \arrow[swap, dashed]{u}{\exists ! \varphi}  &                                                                                     \\
U_i \arrow{ru}{\iota_i} \arrow[bend left]{ruu}{\dindsi{\iota}{Z}{i}} &                                                                              & U_j \arrow[swap]{lu}{\iota_j} \arrow[swap,bend right]{luu}{\dindsi{\iota}{Z}{j}}  \\
                                                                                & U_i\subsm{\times}{U} U_j \arrow{lu}{\proj{1}{\subsm{U}{i}\subsm{\times}{\scalebox{0.7}{$U$}} \subsm{U}{j}}} \arrow[swap]{ru}{\proj{2}{\subsm{U}{i}\subsm{\times}{\scalebox{0.7}{$U$}} \subsm{U}{j}}}            &                                                                                                                                                         
\end{tikzcd}}
\end{center}   
\caption{}\label{topoul1130}
\end{figure}
\noindent Let \(\mathbf{1}\) be a one-element set. 
By the universal property of equalizers, there exists a unique morphism \(\mu: \mathbf{1} \rightarrow \textsf{Hom}(U, Z)\) making the following diagram commute:
\begin{figure}[H] 
\begin{center}
{\footnotesize \begin{tikzcd}[column sep=large] {\textsf{Hom}(U,Z)} \arrow{r} {\subsm{d}{0}^{Z,\iota}} & {\subsm{\prod}{i\!\in\! \mathrm{I}} \textsf{Hom}(\subsm{U}{i},Z)}  \arrow[shift left] {r}{\subsm{d}{1}^{Z,\iota}} \arrow[shift right,swap]{r}{\subsm{d}{2}^{Z,\iota}} & {\subsm{\prod}{(i,j)\! \in\! \mathrm{I}^2 }^{ }\textsf{Hom}(\subsm{U}{i}\subsm{\times}{U} \subsm{U}{j},Z)}\\ \mathbf{1}  \arrow[dashed]{u}{\exists ! \mu} \arrow[ swap]{ur}{\subsm{\langle e_{\dindsi{\iota}{Z}{i}} \rangle}{i\! \in\! \mathrm{I}} } \end{tikzcd}}
\end{center}  
\end{figure}
\noindent where $e_{\dindsi{\iota}{Z}{i}} : \mathbf{1}  \rightarrow \textsf{Hom}(\subsm{U}{i},Z)$ denotes the map sending the unique element of $\mathbf{1}$ to $\dindsi{\iota}{Z}{i}$.
\noindent Therefore, this morphism $\mu$ send \(\mathbf{1}\) to \(\varphi\) making the diagram in Figure \ref{topoul1130} commute, the uniqueness of $\mu$ guarantees the uniqueness of \(\varphi\).

\item[$(2)\Rightarrow (1)$] Suppose that \(U\) is a glued-up object over \(\gcov{\EuScript{U}}\) through \(\subsm{(\iota_i)}{i\! \in\! \mathrm{I}}\). By the definition of a glued-up object, we know that \(\subsm{d}{0}^{Z, \iota}\) equalizes \(\subsm{d}{1}^{Z, \iota}\) and \(\subsm{d}{2}^{Z, \iota}\). 

\noindent Now let \(\varphi: T \rightarrow \subsm{\prod}{i\! \in \!\mathrm{I}} \textsf{Hom}(U_i, Z)\) be a set morphism equalizing \(\subsm{d}{1}^{Z, \iota}\) and \(\subsm{d}{2}^{Z, \iota}\). We want to prove that there exists a unique morphism \(\psi: T \rightarrow \textsf{Hom}(U, Z)\) making the following diagram commute:

\begin{figure}[H] 
\begin{center}
{\footnotesize \begin{tikzcd}[column sep=large] {\textsf{Hom}(U,Z)} \arrow{r} {\subsm{d}{0}^{Z,\iota}} & {\subsm{\prod}{i\!\in\! \mathrm{I}} \textsf{Hom}(\subsm{U}{i},Z)}  \arrow[shift left] {r}{\subsm{d}{1}^{Z,\iota}} \arrow[shift right,swap]{r}{\subsm{d}{2}^{Z,\iota}} & {\subsm{\prod}{(i,j)\! \in\! \mathrm{I}^2 }^{ }\textsf{Hom}(\subsm{U}{i}\subsm{\times}{U} \subsm{U}{j},Z)}\\ T  \arrow[dashed]{u}{\exists ! \psi}  \arrow[ swap]{ur}{\varphi } \end{tikzcd}}
\end{center}   \caption{}\label{topoul1132} 
\end{figure}
\noindent Let \(t \in T\). We write \(\varphi(t) = \subsm{(t_i)}{i\! \in\! \mathrm{I}}\) where \(t_i \in \textsf{Hom}(U_i, Z)\) for all $i \in \mathrm{I}$. We have 
$$t_i \circ \proj{2}{U_i \times_U U_j} = t_j \circ \proj{2}{U_i \times_U U_j},  \text{ for all } i, j \in \mathrm{I}.$$ Thus, by the universal property of the glued-up object, there exists a unique morphism \(f_t \in \textsf{Hom}(U, Z)\) such that \(\subsm{d}{0}^{Z, \iota}(f_t) = \subsm{( t_i)}{i \!\in\! \mathrm{I}}\). Hence, by setting \(\psi(t) := f_t\) for all \(t \in T\), we prove the existence of \(\psi\) making the diagram in Figure \ref{topoul1132} commute et the uniqueness of such a map is guaranteed by the uniqueness of $f_t$ for all $t\in T$.
\item[$(2)\Leftrightarrow (3)$] by definition.
		\item[$(3)\Leftrightarrow (4)$] by Proposition \ref{HomG}.
	\end{enumerate}
\end{proof}


\noindent The following proposition follows directly from Remark \ref{rem38}.

\begin{proposition}
We assume furthermore that \(\mathbb{C}\) admits coproduct, and that pullback commute with coproducts. Let \(\EuScript{U} = (U, \subsm{(U_i, \iota_i)}{i\! \in\! \mathrm{I}})\) be a sink in \(\mathbb{C}\). Then the following assertions are equivalent:
\begin{enumerate}
    \item \(\subsm{(\iota_i)}{i\! \in \!\mathrm{I}}\) is an effective epimorphism.
    \item the following sequence is an equalizer diagram in the category of sets:
    \begin{center}
   \adjustbox{scale=0.8,center}{%
   \begin{tikzcd}[column sep=large]
\emph{\textsf{Hom}}(U, Z) \arrow{rr}{\subsm{\langle \iota_i \rangle}{i\! \in\! \mathrm{I}}^\ast} & & \emph{\textsf{Hom}}\left(\subsm{\coprod}{i\! \in\! \mathrm{I}} U_i, Z\right) \arrow[shift left]{rr}{\subsm{\left\langle {\proj{1}{\subsm{\coprod}{i\! \in\! \mathrm{I}} U_i \times_U \subsm{\coprod}{j\! \in \!\mathrm{J}} U_j}} \right \rangle^\ast }{\!\!(i,j)\!\in \! \mathrm{I}^2}} \arrow[shift right,swap]{rr}{\subsm{\left \langle {\proj{2}{\subsm{\coprod}{i\! \in\! \mathrm{I}} U_i \times_U \subsm{\coprod}{j\! \in \!\mathrm{J}} U_j}} \right \rangle^\ast }{\!\!(i,j)\!\in \! \mathrm{I}^2}} & & \emph{\textsf{Hom}}\left( \subsm{\coprod}{i\! \in\! \mathrm{I}} U_i \times_U  \subsm{\coprod}{j\! \in\! \mathrm{I}} U_j, Z\right),
    \end{tikzcd}}
    \end{center}  
    \item the canonical morphism \(\subsm{\langle \iota_i \rangle}{i\! \in\! \mathrm{I}}\) from \(\subsm{\coprod}{i\! \in\! \mathrm{I}} U_i\) to \(U\) is a descent morphism.
\end{enumerate}
\end{proposition}


\noindent The following lemma shows that the composition of universal effective epimorphisms is again a universal effective epimorphism.

\begin{lemma}\label{compoeff} Let $(U,\subsm{(U_i,\iota_i)}{i\!\in \! \mathrm{I}})$ and $(U_i, \subsm{(\subsm{V}{ij}, \ell_{ij})}{j\!\in \! \mathrm{J}_i})$ be sinks in $\mathbb{C}$. If $\subsm{(\iota_i)}{i\!\in \! \mathrm{I}}$ and $\subsm{(\ell_{ij})}{j\!\in \! \mathrm{J}_i}$ are universal effective epimorphisms, then $\subsm{(\iota_i\circ \ell_{ij})}{(i,j)\!\in\! \mathrm{L} }$ is a universal effective epimorphism where $\mathrm{L}=\big\{(i,j)\;| \;i\in \mathrm{I}, j\in \mathrm{J}_i\big\}$.
\end{lemma}
\begin{proof} 
Let $(V,f)\in (\mathbb{C} \downarrow U)$ and $i \in \mathrm{ I}$. 
	 We set \(\eta := \subsm{\left( \proj{2}{U_i \subsm{\times}{\scalebox{0.7}{$U$}} V}\right) }{i\!\in \! \mathrm{I}}\), \(\eta_i := \subsm{( \proj{2}{V_{ij}\subsm{\times}{\scalebox{0.7}{$U_i$}}(U_i \subsm{\times}{\scalebox{0.7}{$U$}} V)})}{j\! \in\! \mathrm{J}_i}\), and \(\kappa := \subsm{( \proj{2}{V_{ij} \subsm{\times}{\scalebox{0.7}{$U$}} V})}{(i,j)\! \in\! \mathrm{L}}\). Consider the figure below:
\begin{figure}[H]
\begin{center}
{\tiny\begin{tikzcd}[column sep=large]
 {\textsf{Hom}(V,Z)} \arrow[thick]{rr}{\subsm{d}{0}^{Z,\eta}}                            &  & {\subsm{\prod}{i\!\in\! \mathrm{I}} \textsf{Hom}(U_i \subsm{\times}{\scalebox{0.7}{$U$}} V,Z)} \arrow[shift left]{rr}{\subsm{d}{1}^{Z,\eta}} \arrow[shift right,swap]{rr}{\subsm{d}{2}^{Z,\eta}}\arrow[thick]{dd}{\subsm{\left(\subsm{d}{0}^{Z,\eta_i}\right)}{i\!\in \! \mathrm{I}}} &  & {\subsm{\prod}{(i,j)\! \in\! \mathrm{I}^2 }^{ }\textsf{Hom}(\subsm{U}{i}\subsm{\times}{U} \subsm{U}{j}\subsm{\times}{U} V,Z)} \arrow{dd} {r}\\ && &&\\ 
 T \arrow[swap]{rr}{f=\subsm{\langle f_i\rangle}{i\!\in \! \mathrm{I}}}     \arrow[dashed]{uu}{\exists ! h}  \arrow[dashed]{rruu}{\exists ! g}                     &  & {\subsm{\prod}{(i,k)\!\in\! \mathrm{L}}\textsf{Hom}(\subsm{V}{ik}\subsm{\times}{U} V,Z)} \arrow[shift left, pos=0.4, outer sep=-2pt]{ddddrr}{\subsm{\left(\subsm{d}{1}^{Z,\eta_i}\right)}{i\!\in \! \mathrm{I}}} \arrow[shift right, swap, pos=0.4, outer sep=-4pt]{ddddrr}{\subsm{\left(\subsm{d}{2}^{Z,\eta_i}\right)}{i\!\in \! \mathrm{I}}} \arrow[shift left, thick]{rr}{\subsm{d}{1}^{Z,\kappa}} \arrow[shift right,swap, thick]{rr}{\subsm{d}{2}^{Z,\kappa}} &  & {\subsm{\prod}{((i,k),(j,l)) \!\in \! \mathrm{L}^2}\textsf{Hom}(\subsm{V}{ik}\subsm{\times}{U} \subsm{V}{jl}\subsm{\times}{U}V,Z)} \arrow[]{dddd}{s}
\\ && && \\ && &&   \\ && && \\ && && {\subsm{\prod}{i\in \rm{I} ,(k,l) \!\in \! \mathrm{J_i}^2}\textsf{Hom}(\subsm{V}{ik}\subsm{\times}{U_i} \subsm{V}{il}\subsm{\times}{U}V,Z)}. 
 \end{tikzcd}}
\end{center} \caption{}\label{topoul1134}
\end{figure}
\noindent where
\begin{equation*}
\begin{array}{rcl}
r: & \subsm{\prod}{(i,j)\! \in\! \mathrm{I}^2} \textsf{Hom}(U_i \times_U U_j\times_U V, Z) & \rightarrow \subsm{\prod}{((i,k),(j,l))\! \in\! \mathrm{L}^2} \textsf{Hom}(V_{ik} \times_U V_{jl}\times_U V, Z) \\
   & \subsm{(f_{ij})}{(i,j)\! \in\! \mathrm{I}^2} & \mapsto \subsm{(\subsm{f}{ij} \circ (\subsm{\ell}{ik} \times \subsm{\ell}{jl}\times \subsm{\operatorname{id}}{V}))}{((i,k),(j,l)) \!\in \! \mathrm{L}^2}.
\end{array}
\end{equation*}
\noindent
We note that \(r\) is the composite of:
\begin{equation*}
\subsm{\left(d_0^{Z,p_{ij}}\right)}{(i,j) \!\in \!\mathrm{I}^2}: \subsm{\prod}{(i,j)\! \in\! \mathrm{I}^2} \textsf{Hom}(U_i \times_U U_j\times_U V, Z) \rightarrow \subsm{\prod}{k \!\in\! \mathrm{J}_i} \subsm{\prod}{(i,j)\! \in\! \mathrm{I}^2} \textsf{Hom}(V_{ik} \times_U U_j\times_UV, Z)
\end{equation*}
where $p_{ij}=\subsm{( \proj{2}{\subsm{V}{ik}\subsm{\times}{\scalebox{0.7}{$U_{i}$}} ((U_{i}\subsm{\times}{U}U_j)\times_UV)})}{k\!\in \! \mathrm{J}_i}$, for all \((i,j) \in \mathrm{I}^2\), and
\begin{equation*}
\subsm{\left(d_0^{Z,q_{ij}^k}\right)}{ k\in J_i,(i,j)\! \in\! \mathrm{I}^2}: \subsm{\prod}{k \!\in \!\mathrm{J}_i} \subsm{\prod}{(i,j)\! \in\! \mathrm{I}^2} \textsf{Hom}(V_{ik} \times_U U_j\times_U V, Z) \rightarrow \subsm{\prod}{((i,k),(j,l))\! \in \!\mathrm{L}^2} \textsf{Hom}(V_{ik} \times_U V_{jl} \times_U V , Z)
\end{equation*}
where $q_{ij}^k=\subsm{( \proj{2}{(\subsm{V}{ik}\subsm{\times}{\scalebox{0.7}{$U$}} U_{j}\subsm{\times}{\scalebox{0.7}{$U$}} V)\subsm{\times}{U_j}V_{jl}})}{l\!\in \! \mathrm{J}_i}$, for any \( k \in \mathrm{J}_i,(i,j) \in \mathrm{I}^2\).

\noindent On the other hand $s$ is a composite of:
\begin{equation*}
\begin{array}{rcl}
s_1: & \subsm{\prod}{((i,k),(j,l))\! \in\! \mathrm{L}^2} \textsf{Hom}(V_{ik} \times_U V_{jl}\subsm{\times}{\scalebox{0.7}{$U$}}V, Z) & \rightarrow \subsm{\prod}{i\in \rm{I} ,(k,l)) \!\in \! \mathrm{J_i}^2} \textsf{Hom}(V_{ik} \times_U V_{il}\subsm{\times}{\scalebox{0.7}{$U$}}V, Z) \\
     & \subsm{(f_{(i,k)(j,l)})}{((i,k),(j,l))\! \in\! \mathrm{L}^2}  &\mapsto \subsm{(f_{(i,k)(i,l)})}{i\in \mathrm{ I},(k,l)\! \in\! \mathrm{J_i}^2},
\end{array}
\end{equation*}\noindent
and
\begin{equation*}
	\begin{array}{rcl}
		s_2: & {\subsm{\prod}{i\in \rm{I} ,(k,l)) \!\in \! \mathrm{J_i}^2}\textsf{Hom}(\subsm{V}{ik}\subsm{\times}{U} \subsm{V}{il}\subsm{\times}{U} V,Z)} & \rightarrow {\subsm{\prod}{i\in \rm{I} ,(k,l) \!\in \! \mathrm{J_i}^2}\textsf{Hom}(\subsm{V}{ik}\subsm{\times}{U_i} \subsm{V}{il}\subsm{\times}{U}V,Z)} \\
		& \subsm{(f_{(i,k,l)})}{i\in \rm{I} ,(k,l) \!\in \! \mathrm{J_i}^2}  &\mapsto \subsm{(f_{(i,k,l)}\circ \varphi_{i,k,l})}{i\in \rm{I} ,(k,l) \!\in \! \mathrm{J_i}^2}.
	\end{array}
\end{equation*}
where \(\varphi_{i,k,l}\) is the canonical morphism given by the universal property of pullback from $V_{ik}\times_{U_i}V_{il}$ to $V_{ik}\times_{U}V_{il}$ for any $i\in \rm{I} ,(k,l) \!\in \! \mathrm{J_i}^2$.

\noindent We note that, for any $t\in \{1,2\}$, we have:
\begin{enumerate}
\item[(a)] \(d_0^{Z,\kappa} = \subsm{\left( d_0^{Z,\eta_i}\right)}{i \!\in\! \mathrm{I}} \circ d_0^{Z,\eta}\),
\item[(b)] \(s \circ d_t^{Z,\kappa} = \subsm{\left(d_t^{Z,\eta_i}\right)}{i\! \in \!\mathrm{I}}\) ,
\item[(c)] \(d_t^{Z,\kappa} \circ \subsm{\left( d_0^{Z,\eta_i}\right)}{i\! \in\! \mathrm{I}} = r \circ d_t^{Z,\eta}\).
\item[(d)] For all \(k \in \mathrm{J}_i\), \(\subsm{\left(d_0^{Z,p_{ij}}\right)}{(i,j)\! \in\! \mathrm{I}^2}\) and \(\subsm{\left(d_0^{Z,q_{ij}^k}\right)}{ (i,j)\! \in\! \mathrm{I}^2}\) are monomorphisms, since \(\subsm{(\ell_{ij})}{j\!\in \! \mathrm{J}_i}\) and therefore \(\eta_i\) is a universal effective epimorphism for all \(i \in \mathrm{I}\), by \cite[Proposition 3.26]{awodey}. Hence \(r\) is a monomorphism as a composite of monomorphisms. 
\end{enumerate}
\noindent We want to prove that the bold diagram in Figure \ref{topoul1134} is an equalizer diagram. 

\noindent Let \(f: T \rightarrow \subsm{\prod}{(i,k)\!\in\! \mathrm{L}}\textsf{Hom}(\subsm{V}{ik}\subsm{\times}{U} V,Z)\) such that 
\begin{equation}\label{eqn22}
d_1^{Z,\kappa} \circ f = d_2^{Z,\kappa} \circ f.
\end{equation}
\noindent We can write \(f\) uniquely as \(\subsm{\langle f_i \rangle}{i\! \in\! \mathrm{I}}\) where \(f_i: T \rightarrow \subsm{\prod}{(i,k)\!\in\! \mathrm{L}}\textsf{Hom}(\subsm{V}{ik}\subsm{\times}{U} V,Z)\). We want to prove that there exists a unique \(h: T \rightarrow \textsf{Hom}(V, Z)\) such that \(d_0^{Z,\kappa} \circ h = f\) as shown in Figure \ref{topoul1134}.

\noindent Let \(i \in \mathrm{I}\). Composing the equality in Equation (\ref{eqn22}) with \(s_2 \circ s_1\) on the left, we obtain:
\[
d_1^{Z,\eta_i} \circ f_i = d_2^{Z,\eta_i} \circ f_i.
\]

\noindent Since \(\subsm{(\ell_{ij})}{j\!\in \! \mathrm{J}_i}\) is a universal effective epimorphism, \(\eta_i\) is an effective epimorphism and there exists a unique \(g_i: T \rightarrow \subsm{\prod}{i\! \in\! \mathrm{I}} \textsf{Hom}(U_i\times_U V, Z)\) such that:
\[
\subsm{\left(d_0^{Z,\eta_i}\right)}{i\! \in\! \mathrm{I}} \circ g_i = f_i.
\]
\noindent Thus, \(g = \subsm{\langle g_i \rangle}{i \!\in\! \mathrm{I}}\) is the unique map from \(T\) to \(\subsm{\prod}{i\! \in\! \mathrm{I}} \textsf{Hom}(U_i\times_U V, Z)\) such that \(\subsm{(d_0^{Z,\eta_i})}{i\! \in\! \mathrm{I}} \circ g = f\). From this, we obtain:
\[
d_1^{Z,\kappa} \circ \subsm{\left(d_0^{Z,\eta_i}\right)}{i\! \in\! \mathrm{I}} \circ g = d_2^{Z,\kappa} \circ \subsm{\left(d_0^{Z,\eta_i}\right)}{i\! \in\! \mathrm{I}} \circ g.
\]

\noindent Using note (c) above, we deduce:
\begin{equation}\label{eqn23}
r \circ d_1^{Z,\eta} \circ g = r \circ d_2^{Z,\eta} \circ g.
\end{equation}

\noindent Since \(r\) is a monomorphism, Equation \ref{eqn23} implies:
\[
d_1^{Z,\eta} \circ g = d_2^{Z,\eta} \circ g.
\]

\noindent Since \(\subsm{(\iota_i)}{i\!\in \! \mathrm{I}}\) is a universal effective epimorphism, there exists a unique morphism \(h: T \rightarrow \textsf{Hom}(V, Z)\) such that:
\[
d_0^{Z,\eta} \circ h = g.
\]
\noindent
\noindent This completes the proof.

\end{proof}

\noindent A direct consequence of the definition of universal effective epimorphisms, together with Remark~\ref{rem38} and Lemma~\ref{compoeff}, is that the class of universal effective epimorphisms defines a Grothendieck topology on any finitely complete, locally small category admitting pullbacks. This leads to the following definition-lemma.

\begin{deflem}\label{effctive1}
We define the {\sf effective Grothendieck topology of $\mathbb{C}$} denoted by $\mathbf{Cov}_{\mathbf{eff}}(\mathbb{C})$ to be the set
\[
\mathbf{Cov}_{\mathbf{eff}}(\mathbb{C}) := \big\{ (U, \subsm{(U_i, \iota_i)}{i\! \in\! \mathrm{I}}) \in \mathbf{Sink}(\mathbb{C}) \mid \subsm{(\iota_i)}{i\! \in\! \mathrm{I}} \text{ is a universal effective epimorphism} \big\}. \]
We define the {\sf effective Grothendieck site of $\mathbb{C}$} to be the pair $(\mathbb{C},\subsm{\mathbf{Cov}}{\mathbf{eff}}(\mathbb{C}))$.
\end{deflem}

\subsection{An alternative approach to composing gluing functors}
Let $\mathrm{ I}$ be an index set and for all $i\in \mathrm{ I}$, $\rm{J_i}$ be an index set. For all $i\in \mathrm{ I}$, let $\mathbf{G_i}$ be a split $\mathbb{C}$-gluing functor of type $\mathrm{J_i}$ with glued-up object $U_i$ over $\mathbf{G_i}$ through $\subsm{(\ell_{ij})}{j\!\in \! \mathrm{J}_i}$. Suppose that there is a split $\mathbb{C}$-gluing functor such that $\mathbf{G}(i)=U_i$ for all $i\in \mathrm{ I}$.  Let \(\EuScript{U}=(U, \subsm{(\mathbf{G}(i), \iota_i)}{i\! \in\! \mathrm{I}})\) and \(\EuScript{U}_i=(U_i,\subsm{(\mathbf{G_i}(j), \ell_{ij})}{j\!\in \! \mathrm{J}_i} )\). We will prove that in that context we can compose the family \(\subsm{(\mathbf{G_i})}{i\!\in \!\mathrm{ I}}\) together. By Lemma \ref{gluthru}, we know that $U_i$ is a glued-up object over \(\gcovv{\EuScript{U}_i}\) through \( \subsm{(\ell_{ij})}{j\!\in \! \mathrm{J}_i}\) and $U$ is a glued-up object over \(\gcovv{\EuScript{U}}\) through \(\subsm{(\iota_i)}{i\!\in \! \mathrm{ I}}\). Let $\EuScript{V}=(U, \subsm{(\mathbf{G_i}(j), \ell_{ij})}{j\!\in \! \mathrm{J}_i})$, then \(\gcovv{\EuScript{V}}\) can be viewed as a canonical split effective composite $\mathbb{C}$-gluing functor of the family of gluing functors \(\subsm{(\mathbf{G_i})}{i\!\in \!\mathrm{ I}}\) and $U$ is a composite glued-up object over the family \(\subsm{(\mathbf{G_i})}{i\!\in \!\mathrm{ I}}\) through $\subsm{(\ell_{ij})}{j\!\in \! \mathrm{J}_i}$.

\subsection{(Pre)sheaf on Grothendieck site}
In this section, we aim to recall the concept of (pre)sheaves in the more general framework of sites. 
 
\begin{definition}
    Let \((\mathbb{C}, \mathbf{Cov}(\mathbb{C}))\) be a Grothendieck site, and let \(\mathbb{D}\) be a category admitting products.
    
    \begin{enumerate}
        \item A {\sf presheaf} on \(\mathbb{C}\) with values in \(\mathbb{D}\) is a functor \(\mathbf{S}\) from \(\mathbb{C}^{\op}\) to \(\mathbb{D}\).
        \item A {\sf sheaf} on \(\mathbb{C}\) with values in \(\mathbb{D}\) is a presheaf such that for all $\EuScript{U} \in \mathbf{Cov}(\mathbb{C})$, $\mathbf{S} 
\circ \gcov{\EuScript{U}}$ is a $\mathbb{D}$-gluing functor. In other words, for all coverings \((U, \subsm{(U_i, \iota_i)}{i\! \in\! \mathrm{I}}) \in \mathbf{Cov}(\mathbb{C})\),
        \begin{equation}
        \begin{tikzcd}[column sep=large]
            \mathbf{S}(U) \arrow{r}{\subsm{\left\langle \mathbf{S}(\iota_i^{\op})\right \rangle }{i\! \in \!\mathrm{I}}} & \subsm{\prod\limits}{i \!\in\! \mathrm{I}} \mathbf{S}(U_i) \arrow[shift left]{rr}{\subsm{\left\langle \mathbf{S}((\proj{1}{U_i \times_U U_j})^{\op})\circ p_i \right\rangle}{(i,j)\! \in\! \mathrm{I}^2} } \arrow[shift right, swap]{rr}{ \subsm{\left\langle \mathbf{S}((\proj{2}{U_i \times_U U_j})^{\op}) \circ p_j\right\rangle}{(i,j)\! \in\! \mathrm{I}^2}} & & \subsm{\prod\limits}{(i,j)\! \in\! \mathrm{I}^2} \mathbf{S}(U_i \times_U U_j).
        \end{tikzcd}
        \label{fig1}
        \end{equation}

        \noindent is an equalizer diagram, where $p_i :  \subsm{\prod\limits}{i \!\in\! \mathrm{I}} \mathbf{S}(U_i)  \rightarrow \mathbf{S}(U_i)$ is the canonical projection.
   \item A {\sf separated presheaf} on \(\mathbb{C}\) with values in \(\mathbb{D}\) is a presheaf such that for all coverings \((U, \subsm{(U_i, \iota_i)}{i\! \in\! \mathrm{I}}) \in \mathbf{Cov}(\mathbb{C})\), the canonical morphism \( \subsm{\langle\mathbf{S}(\iota_i^{\op})\rangle}{i\! \in \!\mathrm{I}}:  \mathbf{S}(U) \rightarrow  \subsm{\prod\limits}{i \!\in\! \mathrm{I}} \mathbf{S}(U_i) $ is a monomorphism.
    \end{enumerate}
\end{definition}

\begin{remark}
\leavevmode
\begin{enumerate}
    \item When $\mathbb{C}$ is a locally small category. For any \(Z \in \mathbb{C}\), by definition, the representable functor \(\mathbf{Hom}(-, Z)\) is always a sheaf over the site \((\mathbb{C}, \mathbf{Cov}_{\mathbf{eff}}(\mathbb{C}))\) with values on $\mathbb{Sets}$.
    
    \item For all \(i, j \in \mathrm{I}\), by the definition of the fiber product, we have:
    \[
    \proj{1}{U_i \times_U U_j} \circ \iota_i = \proj{2}{U_i \times_U U_j} \circ \iota_j.
    \]
    Therefore, the family \(\left\langle \mathbf{S}(\iota_i^{\op}) \right\rangle\) always equalizes the two families
    \[
    \left\langle \mathbf{S}((\proj{1}{U_i \times_U U_j})^{\op}) \circ p_i \right\rangle
    \quad \text{and} \quad
    \left\langle \mathbf{S}((\proj{2}{U_i \times_U U_j})^{\op}) \circ p_j \right\rangle.
    \]

    \item Given a presheaf with values in $\mathbb{Sets}$, Yoneda Lemma gives us a natural isomorphism:
    \[
    \mathbf{S}(U) \simeq \mathbf{Nat}({\sf h}_U, \mathbf{S}), \quad \text{where } {\sf h}_U := \operatorname{Hom}(-, U).
    \]
    Hence, a presheaf \(\mathbf{S}\) on \(\mathbb{C}\) with values in $\mathbb{Sets}$ is a sheaf if and only if for every covering \((U, (U_i, \iota_i)_{i \in \mathrm{I}}) \in \mathbf{Cov}(\mathbb{C})\), the following diagram is an equalizer:
    
    \begin{center}
    \adjustbox{scale=0.85,center}{%
    \begin{tikzcd}[column sep=large]
    \mathbf{Nat}({\sf h}_U, \mathbf{S}) \arrow{r} &
    \displaystyle\prod_{i \in \mathrm{I}} \mathbf{Nat}({\sf h}_{U_i}, \mathbf{S}) 
    \arrow[shift left=0.6ex]{r} \arrow[shift right=0.6ex, swap]{r} &
    \displaystyle\prod_{(i,j) \in \mathrm{I}^2} \mathbf{Nat}({\sf h}_{U_i \times_U U_j}, \mathbf{S}),
    \end{tikzcd}}
    \end{center}

    \noindent where the arrows are induced by the morphisms in Diagram~(3) and the Yoneda isomorphisms.

    \item We suppose that
    \[
    \mathbb{D} \in \left\{ \mathbb{Sets}, \mathbb{Grp}, \mathbb{Ab}, R\text{-}\mathbb{Mod}, R\text{-}\mathbb{Alg} \mid R \text{ a ring} \right\},
    \]
    and let \(\mathbf{S}\) be a presheaf on \(\mathbb{C}\) with values in \(\mathbb{D}\).
    
    \begin{enumerate}
        \item \(\mathbf{S}\) is separated if and only if for any covering \((U, (U_i, \iota_i)_{i \in \mathrm{I}})\) and any \(s \in \mathbf{S}(U)\) such that \(\mathbf{S}(\iota_i^{\op})(s) = 0\) for all \(i\), then \(s = 0\).
        
        \item The following are equivalent:
        \begin{enumerate}
            \item \(\mathbf{S}\) is a sheaf.
                 \item The sequence
        \[
        0 \longrightarrow \mathbf{S}(U) 
        \xrightarrow{d_0} \prod_{i \in \mathrm{I}} \mathbf{S}(U_i) 
        \xrightarrow{d_1} \prod_{(i,j) \in \mathrm{I}^2} \mathbf{S}(U_i \times_U U_j)
        \]
        is exact, where
        \[
        d_0(s) := \langle \mathbf{S}(\iota_i^{\op})(s) \rangle_{i \in \mathrm{I}}, \quad
        d_1(\langle s_i \rangle) := \left\langle \mathbf{S}((\proj{1}{U_i \times_U U_j})^{\op})(s_i) 
        - \mathbf{S}((\proj{2}{U_i \times_U U_j})^{\op})(s_j) \right\rangle_{(i,j) \in \mathrm{I}^2}.
        \]
            \item \(\mathbf{S}\) is separated, and for any \(\langle s_i \rangle_{i \in \mathrm{I}} \in \prod_{i \in \mathrm{I}} \mathbf{S}(U_i)\) satisfying
            \[
            \mathbf{S}((\proj{1}{U_i \times_U U_j})^{\op})(s_i) = \mathbf{S}((\proj{2}{U_i \times_U U_j})^{\op})(s_j)
            \quad \text{for all } (i,j) \in \mathrm{I}^2,
            \]
            there exists \(s \in \mathbf{S}(U)\) such that \(\mathbf{S}(\iota_i^{\op})(s) = s_i\). We refer to this as the \textsf{gluing property} of \(\mathbf{S}\), and such a presheaf is called a \textsf{gluable presheaf}.
        \end{enumerate}
    \end{enumerate}
\end{enumerate}
\end{remark}

\begin{definition}
Let $(\mathbb{C},\mathbf{Cov}(\mathbb{C}))$ be a site and $\mathbb{D}$ is a category.
\begin{enumerate}
\item The category of presheaves on $\mathbb{C}$ with values in $\mathbb{D}$ is a functor category denoted as $\Psh{\mathbb{C}}{\mathbb{D}}$.
\item The category of separated presheaves on $\mathbb{C}$ with values in $\mathbb{D}$ denoted as $\SPsh{S}{\mathbb{D}}$ is the full subcategory of $\Psh{\mathbb{C}}{\mathbb{D}}$ whose objects are separated presheaves.

\item The category of sheaves on $\mathbb{C}$ with values in $\mathbb{D}$ denoted $\Sh{\mathbb{C}}{\mathbb{D}}$ as is the full subcategory of $\Psh{\mathbb{C}}{\mathbb{D}}$ whose objects are sheaves on a site.
\end{enumerate}
\end{definition}

\subsection{Direct image and restriction of a presheaf}
The following definition associates two functors, extension of scalars and restriction, with a morphism in a category.
\begin{definition}
	Let $f:V \rightarrow U $ be a morphism in $\mathbb{C}$. We define the functor
	\begin{enumerate} 
	\item $\mathbf{P}_f: (\mathbb{C} \! \downarrow \! U)^\op \rightarrow (\mathbb{C} \! \downarrow \! V)^\op$ such that $\mathbf{P}_f(W,\delta^\op )=(W\subsm{\times}{U} V, (\proj{2}{W\subsm{\times}{U} V})^\op)$ and $\mathbf{P}_f(h^\op)=(h\times_U \operatorname{id}_V)^\op$, for all $h : W \rightarrow W'$ morphism in $\mathbb{C}$.
	\item $\mathbf{R}_f: (\mathbb{C} \! \downarrow \! V)^\op \rightarrow (\mathbb{C} \! \downarrow \! U)^\op$ such that $\mathbf{R}_f(W,\delta^\op )=(W, (\delta \circ f )^\op)$ and $\mathbf{R}_f(h^\op)=h^\op$, for all $h : W \rightarrow W'$ morphism in $\mathbb{C}$.
	\end{enumerate} 
\end{definition}
\noindent We also define the direct image and the restriction of a functor.
\begin{deflem}\label{directimage}
Let $(\mathbb{C},\mathbf{Cov}(\mathbb{C}))$  be a Grothendieck site, $U, V \in \mathbb{C}$, and $f:V\rightarrow U$ be a morphism in $\mathbb{C}$.
\begin{enumerate} 
\item when $\mathbf{S}$ be a presheaf on $(\mathbb{C} \! \downarrow \! V)$ with values in $\mathbb{D}$, we define the { \sf direct image presheaf with respect $f$} denoted as $\subsm{f}{\smallstar}\mathbf{S}$ to be the presheaf on $(\mathbb{C} \! \downarrow \! U)$ with values in $\mathbb{D}$ defined as $\mathbf{S} \circ\mathbf{P}_f$. 
\item when $\mathbf{S}$ be a presheaf on $(\mathbb{C} \! \downarrow \! U)$ with values in $\mathbb{D}$, we define the { \sf restriction presheaf with respect $f$} denoted as $\subsm{\mathbf{S}|}{f}$ to be the presheaf on $(\mathbb{C} \! \downarrow \! V)$ with values in $\mathbb{D}$ defined as $\mathbf{S} \circ\mathbf{R}_f$. 
\end{enumerate}
\end{deflem}

\noindent Finally, we define the direct image and restriction of a natural transformation.
\begin{deflem}\label{directimage}
Let $(\mathbb{C},\mathbf{Cov}(\mathbb{C}))$  be a Grothendieck site, $U, V \in \mathbb{C}$, and $f:V\rightarrow U$ be a morphism in $\mathbb{C}$.
\begin{enumerate} 
\item when $ \alpha: \mathbf{S} \rightarrow \mathbf{T}$ is a natural tranformation of presheaves on $(\mathbb{C} \! \downarrow \! V)$ with values in $\mathbb{D}$, we define the { \sf direct image of $\alpha$ via $f$} denoted as $\subsm{f}{\smallstar}\alpha$ to be the natural transformation from $f_\smallstar \mathbf{S}$ to $f_\smallstar \mathbf{T}$  defined as $\subsm{(\alpha_{\mathbf{P}_f (W,\delta^\op )})}{(W,\delta^\op )\!\in \!(\mathbb{C}  \downarrow  U)^\op}$.
\item when $ \alpha: \mathbf{S} \rightarrow \mathbf{T}$ is a a natural tranformation of presheaf on $(\mathbb{C} \! \downarrow \! U)$ with values in $\mathbb{D}$, we define the { \sf restriction of $\alpha$ through $f$} denoted as $\subsm{\alpha|}{f}$ to be the natural transformation from $\subsm{\mathbf{S}|}{f}$ to $\subsm{\mathbf{T}|}{f}$  defined as $\subsm{(\alpha_{\mathbf{R}_f (W,\delta^\op )})}{(W,\delta^\op ) \!\in\! (\mathbb{C}  \downarrow  V)}$. 
\end{enumerate}
\end{deflem}
\begin{remark}
Given morphisms $f:V\rightarrow U$ and $g: W\rightarrow V$, $\mathbf{S},\mathbf{T}$ be presheaves on  $(\mathbb{C} \downarrow U)$ with values in $\mathbb{D}$, and a natural transformation $\alpha: \mathbf{S}\rightarrow \mathbf{T}$, when it is well-defined,  the composition of a restriction morphism induced by $f$ on $\alpha$ with a restriction morphism induced by $g$ on $\subsm{f}{\smallstar}(\subsm{\alpha|}{f})$ is a restriction morphism induced by $f\circ g$ on $\alpha$. Indeed, we have $$\subsm{(f\circ g)}{\smallstar}(\subsm{\alpha|}{f\circ g})=\subsm{g}{\smallstar}(\subsm{\subsm{f}{\smallstar}(\subsm{\alpha|}{f})|}{g}).$$
\end{remark}

\section{Gluing morphisms of (pre)sheaves on a site}

\noindent The following theorem reformulates the gluing property for morphisms between sheaves, showing its equivalence to the sheaf condition for an associated functor. We omit the proof, as it follows directly from the structure of the statement.

\begin{theorem}
Let \((\mathbb{C}, \mathbf{Cov}(\mathbb{C}))\) be a Grothendieck site.  
Let \(\mathbf{S}\) be a presheaf and \(\mathbf{T}\) a sheaf on \((\mathbb{C} \downarrow U)\), both valued in a category \(\mathbb{D}\).

\noindent Define the functor
\[
\mathsf{N}_{\mathbf{S}, \mathbf{T}} : (\mathbb{C} \downarrow U)^\op \to \mathbb{Sets}
\quad \text{by} \quad
(V, \delta^\op) \mapsto \operatorname{Nat}(\mathbf{S}|_{\delta}, \mathbf{T}|_{\delta}).
\]
Then \(\mathsf{N}_{\mathbf{S}, \mathbf{T}}\) is a sheaf on the site \(((\mathbb{C} \downarrow U), \mathbf{Cov}(\mathbb{C} \downarrow U))\), with values in \(\mathbb{Sets}\). In other words, for any covering 
\[
\EuScript{U} = (U, (U_i, \iota_i)_{i \in \mathrm{I}}) \in \mathbf{Cov}(\mathbb{C}),
\]
the composite functor \(\mathsf{N}_{\mathbf{S}, \mathbf{T}} $ $\circ \gcov{\EuScript{U}}\) is a \(\mathbb{Sets}\)-gluing functor. More precisely, for any family of natural transformations \(\bm{\alpha} = (\alpha_i)_{i \in \mathrm{I}}\), where
\[
\alpha_i: \mathbf{S}|_{\iota_i} \longrightarrow \mathbf{T}|_{\iota_i},
\]
such that for all \(i, j \in \mathrm{I}\),
\[
(\iota_i \circ \proj{1}{U_i \times_U U_j})_{\smallstar} \left( \alpha_i|_{\proj{1}{U_i \times_U U_j}} \right) 
= (\iota_j \circ \proj{2}{U_i \times_U U_j})_{\smallstar} \left( \alpha_j|_{\proj{2}{U_i \times_U U_j}} \right),
\]
there exists a unique natural transformation \(\alpha_{\EuScript{U}, \bm{\alpha}} : \mathbf{S} \to \mathbf{T}\) defined, for each \((V, \delta) \in (\mathbb{C} \downarrow U)\), by the unique morphism \({\alpha_{\EuScript{U}, \bm{\alpha}}}_{(V, \delta^\op)}\) making the following diagram commute:

\begin{figure}[H]
    \centering
    {\tiny
    \begin{tikzcd}[column sep=large]
        \mathbf{S}(V,\delta^\op) \arrow[dashed]{dd}{{\alpha_{\EuScript{U}, \bm{\alpha}}}_{\left(V,\delta^\op\right)}} \arrow{rr}{} 
        && \prod_{i \in \mathrm{I}} \mathbf{S}(V \times_U U_i, (\delta \circ \proj{1}{V \times_U U_i})^\op) 
        \arrow[shift left]{rr}{} \arrow[shift right,swap]{rr}{} 
        \arrow{dd}{\left( {\alpha_{i}}_{\left(V \times_U U_i, \left(\delta \circ \proj{1}{V \times_U U_i}\right)^\op\right) } \right)_{i \in \mathrm{I}}} 
        && \prod_{i,j \in \mathrm{I}} \mathbf{S}(V \times_U U_i \times_U U_j, (\delta \circ \proj{1}{V \times_U U_i \times_U U_j})^\op) 
        \arrow{dd}{\left( {\alpha_{i}}_{\left(V \times_U U_i \times_U U_j, \left(\delta \circ \proj{1}{V \times_U U_i \times_U U_j}\right)^\op\right)} \right)_{i,j \in \mathrm{I}}} \\
        &&&&\\
        \mathbf{T}(V, \delta^\op) \arrow{rr}{} 
        && \prod_{i \in \mathrm{I}} \mathbf{T}(V \times_U U_i, (\delta \circ \proj{1}{V \times_U U_i})^\op) 
        \arrow[shift left]{rr}{} \arrow[shift right,swap]{rr}{} 
        && \prod_{i,j \in \mathrm{I}} \mathbf{T}(V \times_U U_i \times_U U_j, (\delta \circ \proj{1}{V \times_U U_i \times_U U_j})^\op)
    \end{tikzcd}
    }
    \caption{Gluing diagram for natural transformations}
    \label{topwel1134}
\end{figure}

\noindent Moreover, for all \(i \in \mathrm{I}\), one has
\[
\alpha_{\EuScript{U}, \bm{\alpha}}|_{\iota_i} = \alpha_i.
\]
\end{theorem}

\section{Gluing (pre)sheaves on a site}\label{gluprshv}
In this section $(\mathbb{C},\subsm{\mathbf{Cov}}{}(\mathbb{C}))$ is a Grothendieck site, and $\mathbb{D}$ is a category.
\subsection{Gluing functor in the category of presheaves}
\noindent In this section, we revisit the gluing of (pre)sheaves through the lens of a gluing functor, providing a categorical formulation of the construction.

\begin{definition}
	Let $\mathbb{E}\in \{\Psh{\mathbb{C}}{\mathbb{D}}, \SPsh{\mathbb{C}}{\mathbb{D}}, \Sh{\mathbb{C}}{\mathbb{D}}\}$, ${\rm I}$ be a set  
	and $\mathbf{G}$ be a functor from $\mathbb{S}_2({\rm I})$ to $\mathbb{E}$. 
\begin{enumerate}
\item For each $V\in \mathbb{C}$, we define the induced split $\mathbb{D}$-gluing functor at $V$, denoted by $\mbf{G}^{V}$, as the functor such that, for all $a \in \mathbb{S}_2({\rm I})$ and $f$ morphism in $\mathbb{S}_2({\rm I})$,
	\begin{enumerate}
		\item $\mbf{G}^V(a) = {\mbf{G}(a)}({V})$;
		\item ${\mbf{G}}^{V}(f)= {{\mbf{G}(f)}}_V$.		
	\end{enumerate}
\item For each $f: V \rightarrow W$, we define the refinement map $\rho_{\mathbf{G},f}$ from $\mathbf{G}^W$ to $\mathbf{G}^V$ as follows:  $\subsm{ \rho_{\mathbf{G},f}}{a} = {\mbf{G}(a)}({f})$, for each $a\in \mathbb{S}_2 ({\rm I} )$.
\end{enumerate}
\end{definition}

In the following definition, we introduce the notion of a split gluing functor for presheaves.
\begin{definition}\label{standsheaf1}
	Let \begin{itemize} 
		\item ${\rm I}$ be a set,
	\item $\EuScript{U}:=(U, \subsm{(U_i, \iota_i)}{i\! \in\! \mathrm{I}})\in \subsm{\mathbf{Cov}}{}(\mathbb{C})$,
	\item $\mathbb{E}\in \{\Psh{(\mathbb{C}\! \downarrow \! U)}{\mathbb{D}}, \SPsh{(\mathbb{C}\! \downarrow \! U)}{\mathbb{D}}, \Sh{(\mathbb{C}\! \downarrow \! U)}{\mathbb{D}}\}$,  
	\item $\mathbf{S}=\subsm{\left( \mathbf{S}_i\right)}{i \in {\rm I}}$ be a family of functor where $\mathbf{S}_i\in \mathbb{E}_{U_i}$, for any $i \in {\rm I}$. 
	\item $\mathbf{\Phi} =\subsm{\left( \Phi_{i,j}\right)}{(i,j) \in {\rm I\times I}}$ be a family of natural correspondence from $ \subsm{(\subsm{
		\iota}{i} \circ \proj{1}{U_i\subsm{\times}{U} U_j} )}{\smallstar}  \subsm{\subsm{\mathbf{S}}{i}|}{ \proj{1}{U_i\subsm{\times}{U} U_j}}$ to  $\subsm{(\subsm{
		\iota}{j} \circ \proj{1}{U_j\subsm{\times}{U} U_i} )}{\smallstar}  \subsm{\subsm{\mathbf{S}}{j}|}{ \proj{1}{U_j\subsm{\times}{U} U_i}}$
		
	\end{itemize} 

\begin{enumerate}	
\item We say that the tuple $\mathbf{\mathcal{I}}$ $:= (\EuScript{U} , \mathbf{S}, \mathbf{\Phi} )$ is a {\sf gluing datum}.
\item Let $\mathbf{\mathcal{I}}$ $:= (\EuScript{U} , \mathbf{S}, \mathbf{\Phi} )$ be a gluing datum. We define $\mathbf{H}_{\mathbf{\mathcal{I}}}$ as follows: for all $i,j \in \mathrm{I}$, we have
	\begin{enumerate}
		\item $\Goi{\mathbf{H}_{\mathbf{\mathcal{I}}}}{i} = \subsm{\subsm{
		\iota}{i}}{\smallstar}\subsm{\mathbf{S}}{i}$; 
		\item $\Goij{\mathbf{H}_{\mathbf{\mathcal{I}}}}{i}{j}=$ $ \subsm{(\subsm{
		\iota}{i} \circ \proj{1}{U_i\subsm{\times}{U} U_j} )}{\smallstar}  \subsm{\subsm{\mathbf{S}}{i}|}{ \proj{1}{U_i\subsm{\times}{U} U_j}}$;
				\item $\Goij{\mathbf{H}_{\mathbf{\mathcal{I}}}}{j}{i}=$ $ \subsm{(\subsm{
		\iota}{j} \circ \proj{1}{U_j\subsm{\times}{U} U_i} )}{\smallstar}  \subsm{\subsm{\mathbf{S}}{j}|}{ \proj{1}{U_j\subsm{\times}{U} U_i}}$;
			\item $\Gnij{\text{$\mathbf{H}_{\mathbf{\mathcal{I}}}$}}{\text{$\iuv{i}{j}$}}=
		\subsm{\text{$(\mathbf{S}_i( (\proj{1}{(U_i \subsm{\times}{U}V)\subsm{\times}{U} U_j})^\op))$}}{(V, \delta) \in (\mathbb{C}\! \downarrow \! U)}$.
		\item $\Gnij{\text{$\mathbf{H}_{\mathbf{\mathcal{I}}}$}}{\text{$\tau_{i,j}$}}=\Phi_{i,j}$
		\end{enumerate}
\end{enumerate}
\end{definition} 
	
We provide a canonical representative for the limit of a split gluing functor in certain categories.		\begin{deflem}\label{standsheaf2}   We continue with the notation from Definition-Lemma~\ref{standsheaf1}, and further assume that $\mathbb{D} \in \{ \mathbb{Sets}, \mathbb{Grp}, \mathbb{Ab}, R\text{-}\mathbb{Mod}, R\text{-}\mathbb{Alg} \mid R \text{ is a ring} \}$. We define the {\sf standard representative of a limit of $\mathbf{H}_{\mathbf{\mathcal{I}}}$} as the pair
\[
\left(\dindsi{\mathbf{L}}{{\mathbf{\mathcal{I}}}}{}, \dindiv{\Psi}{\mathbf{L}}{{\mathbf{\mathcal{I}}}}{}\right),
\]
where:
\begin{itemize}
    \item $\dindsi{\mathbf{L}}{{\mathbf{\mathcal{I}}}}{}$ is the (pre)sheaf on $(\mathbb{C}\! \downarrow \! U)$ defined by
    \[
    \dindsi{\mathbf{L}}{{\mathbf{\mathcal{I}}}}{}(V, \delta^\op) := L^{(V, \delta^\op)}_{{\mathbf{\mathcal{I}}}}, \quad \text{for every } V \in (\mathbb{C}\! \downarrow \! U),
    \]
    as in Definition-Lemma~\ref{stdrep}, and on morphisms $f: (V,\delta) \to (W,\rho)$ in $(\mathbb{C}\! \downarrow \! U)$, it is given by
    \[
    \dindsi{\mathbf{L}}{{\mathbf{\mathcal{I}}}}{}(f^\op) := \operatorname{lim} \rho_{\mathbf{H}_{\mathbf{\mathcal{I}}},f^\op}.
    \]
    More concretely, for $\subsm{(s_i)}{i \in \mathrm{I}} \in L^{(W, \rho^\op)}_{{\mathbf{\mathcal{I}}}}$, we have:
    \[
    \dindsi{\mathbf{L}}{{\mathbf{\mathcal{I}}}}{}(f^\op)\left( \subsm{(s_i)}{i \in \mathrm{I}} \right) = \subsm{\langle\Goi{\mathbf{H}_{\mathbf{\mathcal{I}}}}{i}(f^\op)(s_i)\rangle}{i \in \mathrm{I}}.
    \]
    
    \item $\dindiv{\Psi}{\mathbf{L}}{{\mathbf{\mathcal{I}}}}{} := \left( \dindiv{\Psi}{\mathbf{L}}{{\mathbf{\mathcal{I}}}}{_a} \right)_{a \in \mathbb{P}_2(\mathrm{I})}$ is the family of natural transformations such that for any $(V,\delta^\op) \in (\mathbb{C}\! \downarrow \! U)$:
    \begin{itemize}
        \item For each $i \in \mathrm{I}$, 
        \[
        \dindgiv{\Psi}{\mathbf{L}}{{\mathbf{\mathcal{I}}}}{i \ (V,\delta^\op)}{} : \dindsi{\mathbf{L}}{{\mathbf{\mathcal{I}}}}{}(V,\delta^\op) \longrightarrow \Goi{\mathbf{H}_{\mathbf{\mathcal{I}}}}{i}(V,\delta^\op)
        \]
        sends $\subsm{(s_k)}{k \in \mathrm{I}}$ to $s_i$.
        
        \item For each $(i,j) \in \mathbb{S}_2(\mathrm{I})$,
        \[
        \dindgiv{\Psi}{\mathbf{L}}{{\mathbf{\mathcal{I}}}}{(i,j)\ (V,\delta^\op)}{} := \subsm{\Gnij{\text{$\mathbf{H}_{\mathbf{\mathcal{I}}}$}}{\text{$\iuv{i}{j}$}}}{(V,\delta^\op)} \circ \dindgiv{\Psi}{\mathbf{L}}{\mathbf{H}_{\mathbf{\mathcal{I}}}}{i\ (V,\delta^\op)}{}.
        \]
    \end{itemize}
In particular, the following diagram commutes: 
$$\xymatrix{    \dindsi{\mathbf{L}}{{\mathbf{\mathcal{I}}}}{}\!\!\mid_{\iota_i \circ \proj{1}{\subsm{U}{i}\subsm{\times}{\scalebox{0.7}{$U$}} U_{j}}} \ar[rd]_{{\dindiv{\Psi}{\mathbf{L}}{{\mathbf{\mathcal{I}}}}{j}}\!\mid_{\iota_j \circ\proj{1}{\subsm{U}{j}\subsm{\times}{\scalebox{0.7}{$U$}} U_{i}}}} \ar[rr]^{ {\dindiv{\Psi}{\mathbf{L}}{{\mathbf{\mathcal{I}}}}{i}}\!\mid_{\iota_i \circ\proj{1}{\subsm{U}{i}\subsm{\times}{\scalebox{0.7}{$U$}} U_{j}}}} &&\ar[dl]^{\Phi_{i,j}} \left( \iota_i \smallstar \mathbf{S}_i \right) \!\!\mid_{\iota_i \circ \proj{1}{\subsm{U}{i}\subsm{\times}{\scalebox{0.7}{$U$}} U_{j}}}\\
&  \left( \iota_j \smallstar \mathbf{S}_j \right) \!\!\mid_{\iota_j \circ \proj{1}{\subsm{U}{j}\subsm{\times}{\scalebox{0.7}{$U$}} U_{i}}} }$$ 
\end{itemize}

\end{deflem}
\begin{proof}

 Let $((V, \delta), \subsm{(V_k,\delta_k), \eta_k)}{k\! \in\! \mathrm{K}})\in \subsm{\mathbf{Cov}}{}(\mathbb{C}\downarrow U)$.
We recall that:

{ $$  \dindsi{\mathbf{L}}{{\mathbf{\mathcal{I}}}}{}\!(V, \delta^\op):=\begin{Bmatrix}\subsm{(\subsm{s}{i})}{i\!\in \!\rm{I}} \in \subsm{\prod\nolimits}{i\! \in\! \rm{I} }  \subsm{\mathbf{S}}{i}(U_i\times_U V) \;|\; \subsm{\mathbf{S}_i( (\proj{1}{(U_i \subsm{\times}{U}V)\subsm{\times}{U} U_j})^\op)}{} ({\subsm{s}{i}})=\Phi_{j,i}( \subsm{\mathbf{S}_j( (\proj{1}{(U_j \subsm{\times}{U}V)\subsm{\times}{U} U_j})^\op)}{} ( {\subsm{s}{j}} )), \; \forall \; i,j\in\rm{I}\end{Bmatrix}.$$ }
$\dindsi{\mathbf{L}}{{\mathbf{\mathcal{I}}}}{}\! $ is a presheaf, since $\Goi{\mathbf{H}_{\mathbf{\mathcal{I}}}}{i}$  is a presheaf for all $i \in {\rm I}$.  $\dindsi{\mathbf{L}}{{\mathbf{\mathcal{I}}}}{}\! $ is a separated presheaf, since $\Goi{\mathbf{H}_{\mathbf{\mathcal{I}}}}{i}$  is a separated presheaf for all $i \in {\rm I}$ and $\mathbf{Cov}(\mathbb{C}\downarrow U)$ is stable under base change.

Next, we prove that $\dindsi{\mathbf{L}}{{\mathbf{\mathcal{I}}}}{}\! $ is a sheaf.
To prove $\dindsi{\mathbf{L}}{{\mathbf{\mathcal{I}}}}{}\!$ is gluable, let $\subsm{{\bf t}}{k}=(\subsm{{\subsm{t}{k_i}})}{i\!\in\! \rm{I}} \!\in\! \dindsi{\mathbf{L}}{{\mathbf{\mathcal{I}}}}{}\!(\subsm{V}{k}, \delta_k^\op)$ for $k \in \mathrm{K}$ be a family of sections such that
\begin{equation*}
	\dindsi{\mathbf{L}}{{\mathbf{\mathcal{I}}}}{}\!((\proj{1}{\subsm{V}{k}\subsm{\times}{\scalebox{0.7}{$V$}} V_{k'}})^\op) (\subsm{\mbf{t}}{k})=	\dindsi{\mathbf{L}}{{\mathbf{\mathcal{I}}}}{}\!((\proj{2}{\subsm{V}{k}\subsm{\times}{\scalebox{0.7}{$V$}} V_{k'}})^\op) (\subsm{\mbf{t}}{k'}) 
\end{equation*}
for all $k,k'\in \mathrm{K}$. 
Let $i\in \mathrm{I}$ and $k, k'\in \mathrm{K}$. 
\begin{equation*}
\mathbf{S}_i((\proj{1}{(U_i\subsm{\times}{\scalebox{0.7}{$V$}} \subsm{V}{k})\subsm{\times}{\scalebox{0.7}{$U_i$}} (U_i \subsm{\times}{\scalebox{0.7}{$V$}} V_{k'})})^\op) (\subsm{t}{k_i})=	\mathbf{S}_i((\proj{2}{(U_i \subsm{\times}{\scalebox{0.7}{$V$}} \subsm{V}{k})\subsm{\times}{\scalebox{0.7}{$U_i$}} (U_i\subsm{\times}{\scalebox{0.7}{$V$}} V_{k'})})^\op) (\subsm{t}{k_i'}) 
\end{equation*}
We deduce that, since $\mathbf{S}_i$ is a sheaf and $$\left((U_i, \iota_i) , \left((U_i\times_V V_{k}, \iota_i \circ \pi_1^{U_i \times_V V_{k}}), \pi_1^{U_i \times_V V_{k}}\right)_{k \in \mathrm{K}} \right) \in  \subsm{\mathbf{Cov}}{}(\mathbb{C}\downarrow U),$$ there exist $\subsm{s}{i} \!\in\! \Goi{{\mathbf{H}_{\mathbf{\mathcal{I}}}}}{i}(V, \delta^\op)$, such that
\begin{equation}\label{Eq3}
\Goi{{\mathbf{H}_{\mathbf{\mathcal{I}}}}}{i}(\subsm{\eta}{k}^\op)(\subsm{s}{i})=  {\subsm{t}{\subsm{k}{i}}}.
\end{equation}
We set $\mathbf{s}:=\subsm{(\subsm{s}{i})}{i\!\in\! \rm{I}}$. Let $i, j \in \mathrm{I}$. We want to prove $\mathbf{s}\in \dindsi{\mathbf{L}}{{\mathbf{\mathcal{I}}}}{}\!(V)$. That is, 
\begin{equation}\label{Eq4}
\text{$\subsm{\Gnij{\text{${\mathbf{H}_{\mathbf{\mathcal{I}}}}$}}{\text{$\iuv{i}{j}^\op$}}}{}$} ({\subsm{s}{i}})=\Phi_{i,j} (\text{$\subsm{\Gnij{\text{${\mathbf{H}_{\mathbf{\mathcal{I}}}}$}}{\text{$\iuv{j}{i}^\op$}}}{}$} ( {\subsm{s}{j}} )).\end{equation}
Since $\subsm{\mbf{t}}{k} \in \dindsi{\mathbf{L}}{{\mathbf{\mathcal{I}}}}{}\!(\subsm{V}{k}, \delta_k^\op)$, Equation (\ref{Eq3}) implies 
\begin{equation} \label{Em}\nonumber
	\subsm{\Gnij{\text{${\mathbf{H}_{\mathbf{\mathcal{I}}}}$}}{\text{$\iuv{i}{j}^\op$}}}{}(\Goi{{\mathbf{H}_{\mathbf{\mathcal{I}}}}}{i}(\subsm{\eta}{k}^\op)(\subsm{s}{i})) =	\Phi_{j,i}(\subsm{\Gnij{\text{${\mathbf{H}_{\mathbf{\mathcal{I}}}}$}}{\text{$\iuv{j}{i}^\op$}}}{}(\Goi{{\mathbf{H}_{\mathbf{\mathcal{I}}}}}{j}(\subsm{\eta}{k}^\op)(\subsm{s}{j}))) .
\end{equation}
Since $\Gnij{\text{${\mathbf{H}_{\mathbf{\mathcal{I}}}}$}}{\text{$\iuv{i}{j}^\op$}}$ and $\Phi_{j,i} \circ \subsm{\Gnij{\text{${\mathbf{H}_{\mathbf{\mathcal{I}}}}$}}{\text{$\iuv{j}{i}^\op$}}}{}$ are natural transformations, the previous equality can be rewritten as 
\begin{center}

$ {\Goij{{\mathbf{H}_{\mathbf{\mathcal{I}}}}}{i}{j}}$$(\eta_k^\op)(\text{$\subsm{\Gnij{\text{${\mathbf{H}_{\mathbf{\mathcal{I}}}}$}}{\text{$\iuv{i}{j}^\op$}}}{}$} ({\subsm{s}{i}}))
={\Goij{{\mathbf{H}_{\mathbf{\mathcal{I}}}}}{j}{i}}$$(\eta_k^\op)(\Phi_{j,i} (\text{$\subsm{\Gnij{\text{${\mathbf{H}_{\mathbf{\mathcal{I}}}}$}}{\text{$\iuv{j}{i}^\op$}}}{}$} ({\subsm{s}{j}}))),     
$
	
\end{center}
for all $k\in \mathrm{K}$. By the separation property of the sheaf $\Goij{{\mathbf{H}_{\mathbf{\mathcal{I}}}}}{i}{j}$ and 
$$\left((U_i\times_U U_j , \iota_i\circ  \pi_1^{U_i \times_U U_j}) , \left((U_i\times_U U_j\times_V V_{k}, \iota_i \circ \pi_1^{U_i \times_U U_j\times_V V_{k}}), \pi_1^{(U_i \times_U U_j)  \times_V V_{k}}\right)_{k \in \mathrm{K}} \right) \in  \subsm{\mathbf{Cov}}{}(\mathbb{C}\downarrow U),$$ we obtain that Equation (\ref{Eq4}) is satisfied. Hence $\mathbf{s}\in \dindsi{\mathbf{L}}{{\mathbf{\mathcal{I}}}}{}(V, \delta^\op)$. $(\dindsi{\mathbf{L}}{{\mathbf{\mathcal{I}}}}{}\! , \dindiv{\Psi}{\mathbf{L}}{{\mathbf{\mathcal{I}}}}{})$ is clearly a cone over ${\mathbf{H}_{\mathbf{\mathcal{I}}}}$.

Suppose that $(L' , \Psi')$ is a cone over ${\mathbf{H}_{\mathbf{\mathcal{I}}}}$. Let  $i,j\in \mathrm{ I}$ and $(V, \delta^\op)\in (\mathbb{C}\downarrow U)$. By the universal property of $ L^{(V, \delta^\op)}_{{\mathbf{\mathcal{I}}}}$, there exists a unique morphism $\mu_{(V, \delta^\op)}: L'  \rightarrow \dindsi{\mathbf{L}}{{\mathbf{\mathcal{I}}}}{}(V, \delta^\op)$ making the following diagram 
\begin{figure}[H]
	\begin{center}
			\adjustbox{scale=1.0,center}{
\begin{tikzcd}
		\scalebox{0.7}{$L'(V, \delta^\op)$} \arrow[rrd,bend left,"\Psi'_{i (V, \delta^\op)}"]
	\arrow[ddr,bend right,swap,"\Psi'_{j  (V, \delta^\op)}"]
	\arrow[dr,dashed,"\mu_{(V, \delta^\op)}"] \\
	& \scalebox{0.7}{$\dindsi{\mathbf{L}}{{\mathbf{\mathcal{I}}}}{} \!(V, \delta^\op) $} \arrow[d,"\dindgiv{\Psi}{\mathbf{L}}{{\mathbf{\mathcal{I}}}}{i \ (V,\delta^\op)}{} "'] \arrow[r,"\dindgiv{\Psi}{\mathbf{L}}{{\mathbf{\mathcal{I}}}}{j \ (V,\delta^\op)}{} "] & \scalebox{0.7}{${\Goi{{\mathbf{H}_{\mathbf{\mathcal{I}}}}}{i}(V, \delta^\op)}$} \arrow[d,"\subsm{\Gnij{\text{${\mathbf{H}_{\mathbf{\mathcal{I}}}}$}}{\text{$\iuv{i}{j}$}}}{(V, \delta^\op)}"]  \\
	& \scalebox{0.7}{${\Goi{{\mathbf{H}_{\mathbf{\mathcal{I}}}}}{j}(V, \delta^\op)}$} \arrow[r,swap,"\subsm{\Gnij{(\Phi_{j,i}\circ \text{${\mathbf{H}_{\mathbf{\mathcal{I}}}}$}}{\text{$\iuv{j}{i}$}})}{(V, \delta^\op)}"]  & \scalebox{0.7}{$\Goij{{\mathbf{H}_{\mathbf{\mathcal{I}}}}}{i}{j}(V, \delta^\op)$}
	\end{tikzcd}}\caption{}\label{topkel11346}
	\end{center}
\end{figure}	
\noindent commute. In particular, $ \mu_{(V, \delta^\op)}=\langle \Psi'_{i (V, \delta^\op)}\rangle_{i \in \mathrm{I}}$. Finally, we prove that $\mu_{(V, \delta^\op)}$ is a natural transformation. Let $f:(V, \delta^\op)\rightarrow (W, \rho^\op)$ be a morphism in $(\mathbb{C}\downarrow U)$ and $s \in L'(W, \rho^\op)$. We have
\begin{align*}
\dindsi{\mathbf{L}}{{\mathbf{\mathcal{I}}}}{}(f^\op)\circ \mu_{(V, \delta^\op)}(s)&= \subsm{\langle \Goi{{\mathbf{H}_{\mathbf{\mathcal{I}}}}}{i} (f^\op)\circ \ \Psi'_{i (V, \delta^\op)} (s)\rangle}{i\in {\rm I}}\\&=\subsm{\langle \Psi'_{i (W, \rho^\op)}\circ L'(f^\op)(s)\rangle}{i\!\in \! \mathrm{ I}},\; \text{since $L'$ is a natural transformation}\\&=\mu_{(W, \rho^\op)}\circ L'(f^\op)(s).
\end{align*} 
\end{proof}
\begin{remark}
As shown in the proof of Definition-Lemma~\ref{standsheaf1}, verifying that a presheaf obtained by gluing a family of sheaves satisfies the sheaf condition requires both the separability and the gluing properties for each sheaf $\mathbf{S}_i$, where $i \in \mathrm{I}$. In contrast, to ensure that the resulting presheaf is merely separated, it suffices that each $\mathbf{S}_i$ is separated.
\end{remark}

\subsection{Canonical Functor associated with a presheaf and a Cover}

Building on the notion of a canonical functor associated with a sink, we now define the canonical functor associated with a given presheaf and a cover.

\begin{definition}
Let $\EuScript{U} := (U, \subsm{(U_i, \iota_i)}{i \in \mathrm{I}}) \in \subsm{\mathbf{Cov}}{}(\mathbb{C})$ be a cover, and let $\mathbf{F}$ be a presheaf on $(\mathbb{C},\subsm{\mathbf{Cov}}{}(\mathbb{C}))$ with value on $\mathbb{D}$. 

We define the \sf{canonical functor associated with $\mathbf{F}$} and the cover $\EuScript{U}$ as
\[
\mathbf{C}^{\mathbf{F}}_{\text{$\EuScript{U}$}} := \mathbf{H}_{\mathbf{\mathcal{I}}} ,
\]
where ${\mathcal{I}} $ $= (\text{$\EuScript{U}$}, \mathbf{S}, \mathbf{\Phi})$ is the associated gluing datum, given by:
\begin{itemize}
    \item $\mathbf{S} = \left( \mathbf{F}|_{\iota_i} \right)_{i \in \mathrm{I}}$, the family of restrictions of $\mathbf{F}$ to each $\iota_i$;
    \item $\mathbf{\Phi} = ( \mathbf{S} (\varphi_{i,j}\times-))_{i \in \mathrm{I}}$ where $\varphi_{i,j}$ is the canonical pullback morphism between $U_i \times_U U_j$ and $U_j \times_U U_i$.
\end{itemize}
\end{definition}

A direct corollary of Definition-Lemma~\ref{standsheaf1} is the following result.

\begin{proposition}
Let $\mathbf{F}$ be a presheaf on $(\mathbb{C}, \subsm{\mathbf{Cov}}{}(\mathbb{C}))$ with values in $\mathbb{D}$. The following statements are equivalent:
\begin{enumerate}
    \item $\mathbf{F}$ is a sheaf.
    \item For every cover $\EuScript{U} := (U, \subsm{(U_i, \iota_i)}{i \in \mathrm{I}}) \in \subsm{\mathbf{Cov}}{}(\mathbb{C})$, the canonical gluing functor $\mathbf{C}^{\mathbf{F}}_{\EuScript{U}}$ is a split $\Psh{(\mathbb{C} \downarrow U)}{\mathbb{D}}$-gluing functor.
\end{enumerate}
\end{proposition}

\subsection{Effective gluing for (pre)sheaves}
We now define the notion of \emph{effective gluing} in the context of presheaves.

\begin{definition}
We adopt the same notation as in Definition~\ref{standsheaf1}.  
Let $\mathcal{I} :=$ $ (\EuScript{U}, \mathrm{I}, \mathbf{S}, \mathbf{\Phi})$ be a gluing datum. We say that $\mathbf{H}_{\mathcal{I}}$ is an \textsf{effective split $\mathbb{E}$-gluing functor with respect to $\mathcal{I}$} if for all $i, j, k \in \mathrm{I}$:
\begin{enumerate}
    \item \textsf{Identity condition:} $\Phi_{i,i}= \operatorname{id}$,
    \item \textsf{Cocycle condition:} the following diagram commutes:
    \[
    \xymatrix{
        \subsm{\mathbf{S}}{i}|_{\proj{1}{U_i \subsm{\times}{U} U_j \subsm{\times}{U} U_k}} \ar[rr]^{\Phi_{i,j}|_{\proj{1}{U_i \subsm{\times}{U} U_j \subsm{\times}{U} U_k}}} 
        && \subsm{\mathbf{S}}{j}|_{\proj{1}{U_j \subsm{\times}{U} U_i \subsm{\times}{U} U_k}} \ar[dl]^{\Phi_{j,k}|_{\proj{1}{U_j \subsm{\times}{U} U_k \subsm{\times}{U} U_i}} } \\
        & \subsm{\mathbf{S}}{k}|_{\proj{1}{U_k \subsm{\times}{U} U_i \subsm{\times}{U} U_j}} \ar[ul]^{ \Phi_{k,i}|_{\proj{1}{U_k \subsm{\times}{U} U_i \subsm{\times}{U} U_j}} }
    }
    \]
\end{enumerate}
\end{definition}
Passing from a simple gluing to an effective gluing for presheaves allows us to identify the restriction of the glued presheaf to each component of the cover with the direct image of the corresponding member of the glued family.
\begin{proposition} \label{preashf}
We adopt the same notation and assumptions as in Definition~\ref{standsheaf1} and Definition-Lemma~\ref{standsheaf2}. Let $\mathcal{I} :=$ $ (\EuScript{U}, \mathrm{I}, \mathbf{S}, \mathbf{\Phi})$ be a gluing datum. The following statements are equivalent:
\begin{enumerate}
    \item $\mathbf{H}_{\mathcal{I}}$ is an effective split $\mathbb{E}$-gluing functor with respect to $\mathcal{I}$;
    \item For all $i \in \mathrm{I}$, the morphism $\subsm{\dindgiv{\Psi}{\mathbf{L}}{{\mathbf{\mathcal{I}}}}{i}{}|}{\iota_i}$ is a natural correspondence from $\subsm{\dindsi{\mathbf{L}}{{\mathbf{\mathcal{I}}}}{}|}{ \iota_i}$ to $\subsm{(\subsm{\iota_i}{\smallstar}  \mathbf{S}_i)|}{ \iota_i}$.
\end{enumerate}
\end{proposition}

\begin{proof}
\textsf{(1) $\Rightarrow$ (2)}\quad
Assume that $\mathbf{H}_{\mathcal{I}}$ is an effective split $\mathbb{E}$-gluing functor with respect to $\mathcal{I}$. We prove that $\subsm{\dindgiv{\Psi}{\mathbf{L}}{{\mathbf{\mathcal{I}}}}{i}{}|}{\iota_i}$ is a natural correspondence.

To this end, we construct its inverse: for each $(V, \delta^\op) \in (\mathbb{C} \downarrow U_i)$, define
\[
\begin{array}{rcl}
{\Theta_i}_{(V, \delta^\op)} : \subsm{(\subsm{\iota_i}{\smallstar} \mathbf{S}_i)|}{\iota_i}(V, \delta^\op) & \longrightarrow & \subsm{\dindsi{\mathbf{L}}{{\mathbf{\mathcal{I}}}}{}|}{\iota_i}(V, \delta^\op) \\
s_i & \longmapsto & (\Phi_{j,i}(s_i))_{j \in \mathrm{I}}.
\end{array}
\]
The cocycle condition ensures that this map is well-defined; that is, the image lies in $\subsm{\dindsi{\mathbf{L}}{{\mathbf{\mathcal{I}}}}{}|}{\iota_i}(V, \delta^\op)$.

Since $\Phi_{i,i} = \operatorname{id}$ and by the definition of $\dindsi{\mathbf{L}}{{\mathbf{\mathcal{I}}}}{}$, it follows that the composition of $\dindgiv{\Psi}{\mathbf{L}}{{\mathbf{\mathcal{I}}}}{i}{}$ with $\Theta_i$ is the identity in both directions.

\medskip

\textsf{(2) $\Rightarrow$ (1)}\quad
Assume that for all $i \in \mathrm{I}$, the morphism $\subsm{\Psi_i|}{\iota_i}$ is a natural correspondence from $\subsm{\mathbf{S}|}{\iota_i}$ to $\subsm{(\subsm{\iota_i}{\smallstar} \mathbf{S}_i)|}{\iota_i}$. Then, from the final diagram in Definition-Lemma~\ref{standsheaf2}, we have:
\[
\Phi_{i,j} = \left.\dindiv{\Psi}{\mathbf{L}}{{\mathbf{\mathcal{I}}}}{j}\right|_{\iota_j \circ \proj{1}{\subsm{U}{j} \subsm{\times}{U} U_i}} \circ \left(\left.\dindiv{\Psi}{\mathbf{L}}{{\mathbf{\mathcal{I}}}}{i}\right|_{\iota_i \circ \proj{1}{\subsm{U}{i} \subsm{\times}{U} U_j}}\right)^{-1},
\]
and the cocycle condition follows easily.
\end{proof}

\section{Gluing of enriched presheaf}
In categories such as locally ringed spaces or schemes, one typically considers both a presheaf and the topological space on which it is defined. This perspective offers flexibility, particularly when working over varying topological spaces. Motivated by this viewpoint, we introduce the notion of an {enriched (pre)sheaf}, defined simply as a pair consisting of an object in a site and a presheaf on that object.

\begin{definition}\text{}
\begin{itemize}
\item An {\sf enriched (separated) (pre)sheaf over $(\mathbb{C}, \subsm{\mathbf{Cov}}{}(\mathbb{C}))$ with values in $\mathbb{D}$} is defined to be a pair $(U, \mathbf{S})$, where $U$ is an object of $\mathbb{C}$, and $\mathbf{S}$ is a (separated) (pre)sheaf on $(\mathbb{C}, \subsm{\mathbf{Cov}}{}(\mathbb{C}))$ with values in $\mathbb{D}$.
\item A {\sf morphism of enriched presheaves $f : (U, \mathbf{S}) \rightarrow (V, \mathbf{T})$} is a pair $(f, f^\sharp)$ where $f: U \rightarrow V$ is a morphism in $\mathbb{C}$ and $f^\sharp: \mathbf{T} \rightarrow f \smallstar \mathbf{S}$ is a natural transformation. 
\end{itemize}
\end{definition}

\begin{definition}\text{}
\begin{enumerate}
    \item The category of enriched presheaves on $\mathbb{C}$ with values in $\mathbb{D}$ is the functor category denoted by $\EPsh{\mathbb{C}}{\mathbb{D}}$.
    
    \item The category of separated enriched presheaves on $\mathbb{C}$ with values in $\mathbb{D}$, denoted $\ESPsh{S}{\mathbb{D}}$, is the full subcategory of $\EPsh{\mathbb{C}}{\mathbb{D}}$ consisting of separated presheaves.
    
    \item The category of enriched sheaves on $\mathbb{C}$ with values in $\mathbb{D}$, denoted $\ESh{\mathbb{C}}{\mathbb{D}}$, is the full subcategory of $\EPsh{\mathbb{C}}{\mathbb{D}}$ consisting of sheaves on the site.
\end{enumerate}
\end{definition}

Given a functor from a split truncated power category to an enriched presheaf category, one can naturally induce a functor from the same category to the underlying site.

\begin{definition}
Let:
\begin{itemize}
    \item $\mathbb{E}^e\in \{ \EPsh{\mathbb{C}}{\mathbb{D}}, \ESPsh{\mathbb{C}}{\mathbb{D}}, \ESh{\mathbb{C}}{\mathbb{D}} \}$,
    \item $\mathbf{G}$ be a functor from $\mathbb{S}_2(\mathrm{I})$ to $\mathbb{E}^e$.
\end{itemize}

For each $a \in \mathbb{S}_2(\mathrm{I})$, we write $\mathbf{G}(a) = (U_a, \mathbf{S}_a)$.
We define $\mathbf{G}_s : \mathbb{S}_2(\mathrm{I}) \to \mathbb{C}$ as the functor assigning $\mathbf{G}_s(a) := U_a$,
\end{definition}

Given a functor from a split truncated power category to an enriched presheaf category, one can, under suitable additional assumptions described below, naturally induce a pair of functors from the same category to the underlying site and presheaf categories, respectively. Under these assumptions, it follows directly from the definitions that a glued object in the enriched category is simply the pair consisting of the glued objects of the two induced functors.

\begin{deflem}
Let:
\begin{itemize}
    \item $\mathbb{E}^e\in \{ \EPsh{\mathbb{C}}{\mathbb{D}}, \ESPsh{\mathbb{C}}{\mathbb{D}}, \ESh{\mathbb{C}}{\mathbb{D}} \}$,
    \item $\mathbf{G}$ be a functor from $\mathbb{S}_2(\mathrm{I})$ to $\mathbb{E}^e$.
\end{itemize}

For each $a \in \mathbb{S}_2(\mathrm{I})$, write $\mathbf{G}(a) = (U_a, \mathbf{S}_a)$, and for each $f : a \rightarrow a'$ in $\mathbb{S}_2(\mathrm{I})$, write $\mathbf{G}(f) = (\varphi_f, \psi_f)$.

Assume:
\begin{itemize} 
    \item $\mathbf{G}_s$ is a split $\mathbb{C}$-gluing functor with $\lim \mathbf{G} = (U,(\iota_i)_{i\in {\rm I}})$,
    \item $(U, \subsm{(\mathbf{G}(i), \iota_i)}{i\! \in\! \mathrm{I}}) \in \subsm{\mathbf{Cov}}{}(\mathbb{C})$,
    \item $\mathbf{S}_{(i,j)} = \subsm{\subsm{\mathbf{S}}{i}|}{ \proj{1}{U_i\subsm{\times}{U} U_j}}$.
\end{itemize}

We set:
\begin{itemize} 
  \item $\mathbb{E}$ to be the corresponding category to $\mathbb{E}^e$  in $\{\Psh{(\mathbb{C}\! \downarrow \! U)}{\mathbb{D}}, \SPsh{(\mathbb{C}\! \downarrow \! U)}{\mathbb{D}}, \Sh{(\mathbb{C}\! \downarrow \! U)}{\mathbb{D}}\}$,
	\item $\EuScript{U} := (U, \subsm{(\mathbf{G}(i), \iota_i)}{i \in \mathrm{I}})$,
	\item $\mathbf{S} := \subsm{\left( \mathbf{S}_i\right)}{i \in {\rm I}}$,
	\item $\mathbf{\Phi} := \subsm{\left(\subsm{(\subsm{\iota}{i} \circ \proj{1}{U_i\subsm{\times}{U} U_j} )}{\smallstar}  \psi_{\tau_{i,j}}\right)}{(i,j) \in {\rm I\times I}}$,
	\item ${\mathcal{I}} :=$ $ (\EuScript{U}, \mathbf{S}, \mathbf{\Phi})$.
\end{itemize}

We define $\mathbf{G}_p : \mathbb{S}_2(\mathrm{I}) \to \mathbb{E}$ to be the functor $\mathbf{H}_{\mathbf{\mathcal{I}}}$. We have, $\mathbf{G}_s$ is split $\mathbb{D}$-gluing functors and $\mathbf{G}_p$ is a split $\mathbb{E}$-gluing functors if and only if $\mathbf{G}$ is a split $\mathbb{E}^e$-gluing functor.
\end{deflem}

\begin{remark}
To derive the gluing construction for locally ringed spaces from the previous result, observe that the category of locally ringed spaces can be viewed as a subcategory of an enriched presheaf category, where the Grothendieck site is that of topological spaces (Example \ref{settop}), and the presheaves take values in $\mathbb{Rings}$. The additional structure consists of the requirement that for each point, the stalk of the presheaf is a local ring, and morphisms of such presheaves induce local ring homomorphisms on stalks.

We have already described how to glue both the underlying topological spaces and the associated presheaves. The key remaining observation is that the stalks of the glued presheaf are naturally isomorphic to the stalks of the presheaves in the gluing data. This ensures that the locality condition is preserved under gluing. Consequently, the gluing construction extends naturally to locally ringed spaces, and hence also to schemes and morphisms of schemes.
\end{remark}

\bibliographystyle{abnt-num}

\end{document}